\documentclass[12pt]{amsart}
\pdfoutput=1

\usepackage[paperheight=11in, paperwidth=8.5in, left=1in, top=1.25in, right=1in, bottom=1.25in]{geometry}
\usepackage[linktocpage=true, linktoc=all, colorlinks=true, linkcolor=black, citecolor=black, filecolor=black, urlcolor=black, pagebackref=false, pdfstartpage={1}, pdfstartview={FitH}, pdftitle={}, pdfauthor={}, pdfsubject={}, pdfcreator={}, pdfproducer={}, pdfkeywords={}]{hyperref}
\usepackage{afterpage}
\usepackage{amsfonts}
\usepackage{amsmath}
\allowdisplaybreaks[4]
\usepackage{amssymb}
\usepackage{amsthm}
\usepackage{array}
\usepackage{bbm}
\usepackage[justification=centering]{caption}
\usepackage[capitalize,nameinlink,noabbrev,compress]{cleveref}

\usepackage{enumerate}
\usepackage{enumitem}
\usepackage{float}
\restylefloat{figure}
\usepackage{graphicx}
\usepackage{mathrsfs}
\usepackage[framemethod=tikz,ntheorem,xcolor]{mdframed}
\usepackage{parcolumns}
\usepackage{scalerel}
\usepackage{tabularx}
\usepackage{tikz}\pdfpageattr{/Group <</S /Transparency /I true /CS /DeviceRGB>>}
\usetikzlibrary{arrows,calc,cd,intersections,matrix,positioning,through}
\tikzset{commutative diagrams/.cd,every label/.append style = {font = \normalsize}}
\usepackage[all]{xy}
\usepackage[aligntableaux=center]{ytableau}

\numberwithin{equation}{section}
\newtheorem{thm}{Theorem}
\AfterEndEnvironment{thm}{\noindent\ignorespaces}
\numberwithin{thm}{section}

\AfterEndEnvironment{conj}{\noindent\ignorespaces}
\newtheorem{cor}[thm]{Corollary}
\AfterEndEnvironment{cor}{\noindent\ignorespaces}
\newtheorem{lem}[thm]{Lemma}
\AfterEndEnvironment{lem}{\noindent\ignorespaces}

\AfterEndEnvironment{prob}{\noindent\ignorespaces}
\newtheorem{prop}[thm]{Proposition}
\AfterEndEnvironment{prop}{\noindent\ignorespaces}

\theoremstyle{definition}
\newtheorem{defn}[thm]{Definition}
\AfterEndEnvironment{defn}{\noindent\ignorespaces}
\newtheorem{eg_no_qed}[thm]{Example}
\newenvironment{eg}[1][]{\begin{eg_no_qed}[#1]\pushQED{\qed}}{\popQED\end{eg_no_qed}}
\AfterEndEnvironment{eg}{\noindent\ignorespaces}
\newtheorem{rmk}[thm]{Remark}
\AfterEndEnvironment{rmk}{\noindent\ignorespaces}

\theoremstyle{remark}

\AfterEndEnvironment{claim}{\noindent\ignorespaces}
\newtheorem*{claimpf_no_qed}{Proof of Claim}

\AfterEndEnvironment{claimpf}{\noindent\ignorespaces}

\DeclareMathOperator{\ad}{ad}
\newcommand{\adjoint}[1]{#1 ^*}

\renewcommand{\eqref}[1]{\hyperref[#1]{\textup{(\ref*{#1})}}}

\DeclareMathOperator{\Fl}{Fl}

\DeclareMathOperator{\GL}{GL}
\newcommand{\gl}{\mathfrak{gl}}
\DeclareMathOperator{\Gr}{Gr}
\DeclareMathOperator{\grad}{grad}

\DeclareMathOperator{\id}{id}
\DeclareMathOperator{\im}{im}

\newcommand{\interior}[1]{#1^{\circ}}

\renewcommand{\O}{\operatorname{O}}

\DeclareMathOperator{\rank}{rank}

\renewcommand{\sl}{\mathfrak{sl}}

\DeclareMathOperator{\SO}{SO}

\DeclareMathOperator{\spn}{span}

\DeclareMathOperator{\SU}{SU}
\newcommand{\su}{\mathfrak{su}}

\DeclareMathOperator{\tr}{tr}
\newcommand{\transpose}[1]{{#1}{\hspace*{-0.2pt}\raisebox{4pt}{$\scriptstyle\mathsf{T}$}\hspace*{0.5pt}}}

\newcommand{\adinverse}[1]{\ad_{#1}^{-1}}

\newcommand{\B}{\operatorname{B}}

\newcommand{\bflambda}{\lambda}
\newcommand{\bfmu}{\mu}

\newcommand{\bfomega}[1]{\omega_{#1}}

\newcommand{\Bminus}{\operatorname{B}^{-}}

\newcommand{\ccon}[1]{\hspace*{1pt}\overline{#1}\hspace*{1pt}}
\renewcommand{\dashrightarrow}{\mathrel{\ThisStyle{\ooalign{$\SavedStyle\rightarrow$\cr\hfil\textcolor{white}{\rule{2\LMpt}{1\LMex}}\kern2\LMpt\hfil}}}}
\newcommand{\Diag}[1]{\hspace*{1pt}\mathsf{Diag}(#1)\hspace*{1pt}}

\newcommand{\eval}[1]{\hspace*{-1pt}\big\rvert_{#1}\hspace*{2pt}}
\newcommand{\field}{\mathbb{F}}
\newcommand{\flip}{\rho}

\renewcommand{\H}{\operatorname{T}}
\newcommand{\h}{\mathfrak{t}}
\newcommand{\I}{I}
\newcommand{\ii}{\mathrm{i}\hspace*{0.5pt}}
\newcommand{\inv}[2]{\mathsf{inv}(#1,#2)}
\newcommand{\Jac}{\mathcal{J}}

\DeclareMathOperator{\Kterm}{\pi_{U}}
\DeclareMathOperator{\kterm}{\pi_{\mathfrak{u}}}
\newcommand{\Kry}[2]{\mathsf{Vand}(#1,#2)}

\newcommand{\limitorbit}[2]{\mathsf{L}^{\infty}(#1;#2)}
\newcommand{\limitproj}[2]{\mathsf{P}^{\infty}_{#2}(#1)}

\newcommand{\modulo}[1]{\hspace*{-6pt}\pmod{#1}}
\newcommand{\N}{\operatorname{N}}
\newcommand{\n}{\mathfrak{n}}
\newcommand{\norm}{\nu}
\newcommand{\oo}{\mathfrak{o}}

\newcommand{\Orbit}{\mathcal{O}}
\renewcommand{\P}[2]{\operatorname{P}_{#1;\hspace*{0.5pt}#2}}
\newcommand{\PFl}[2]{\Fl_{#1;\hspace*{0.5pt}#2}}

\newcommand{\Proj}[1]{\mathsf{P}_{#1}}
\newcommand{\rev}{\mathsf{rev}}

\newcommand{\stableorbit}[2]{\Orbit_{#2}^{s}(#1)}
\newcommand{\sumof}[1]{\sum\hspace*{-1pt}{#1}}
\newcommand{\T}{\operatorname{T}^{\textnormal{U}}}

\newcommand{\twist}{\vartheta}
\newcommand{\twistorbit}{\vartheta_{\bflambda}}
\newcommand{\U}{\operatorname{U}}
\newcommand{\uu}{\mathfrak{u}}
\newcommand{\Vand}{\mathcal{V}}

\title[Gradient flows and adjoint orbits]{Gradient flows, adjoint orbits, and the topology of totally nonnegative flag varieties}

\author{Anthony M. Bloch}
\address{Department of Mathematics, University of Michigan}
\email{\href{mailto:abloch@umich.edu}{abloch@umich.edu}}
\author{Steven N. Karp}
\address{LaCIM, Universit\'{e} du Qu\'{e}bec \`{a} Montr\'{e}al}
\email{\href{mailto:karp.steven@courrier.uqam.ca}{karp.steven@courrier.uqam.ca}}

\thanks{A.M.B.\ was partially supported by NSF grants DMS-1613819 and DMS-2103026, and AFOSR grant FA 0550-18-0028. S.N.K.\ was partially supported by an NSERC postdoctoral fellowship.}

\begin{document}

\begin{abstract}
One can view a partial flag variety in $\mathbb{C}^n$ as an adjoint orbit $\mathcal{O}_{\lambda}$ inside the Lie algebra of $n\times n$ skew-Hermitian matrices. We use the orbit context to study the totally nonnegative part of a partial flag variety from an algebraic, geometric, and dynamical perspective. The paper has three main parts:

(1) We introduce the totally nonnegative part of $\mathcal{O}_{\lambda}$, and describe it explicitly in several cases. We define a twist map on it, which generalizes (in type $A$) a map of Bloch, Flaschka, and Ratiu (1990) on an isospectral manifold of Jacobi matrices.

(2) We study gradient flows on $\mathcal{O}_{\lambda}$ which preserve positivity, working in three natural Riemannian metrics. In the K\"{a}hler metric, positivity is preserved in many cases of interest, extending results of Galashin, Karp, and Lam (2017, 2019). In the normal metric, positivity is essentially never preserved on a generic orbit. In the induced metric, whether positivity is preserved appears to depends on the spacing of the eigenvalues defining the orbit.

(3) We present two applications. First, we discuss the topology of totally nonnegative flag varieties and amplituhedra. Galashin, Karp, and Lam (2017, 2019) showed that the former are homeomorphic to closed balls, and we interpret their argument in the orbit framework. We also show that a new family of amplituhedra, which we call {\itshape twisted Vandermonde amplituhedra}, are homeomorphic to closed balls. Second, we discuss the symmetric Toda flow on $\mathcal{O}_{\lambda}$. We show that it preserves positivity, and that on the totally nonnegative part, it is a gradient flow in the K\"{a}hler metric up to applying the twist map. This extends a result of Bloch, Flaschka, and Ratiu (1990).
\end{abstract}

\maketitle
\setcounter{tocdepth}{1}
\tableofcontents

\section{Introduction}\label{sec_introduction}

\noindent Let $\Fl_n(\mathbb{C})$ denote the {\itshape complete flag variety}, consisting of all sequences $V_1 \subset \cdots \subset V_{n-1}$ of nested subspaces of $\mathbb{C}^n$ such that each $V_k$ has dimension $k$. We may view $\Fl_n(\mathbb{C})$ as the quotient of $\GL_n(\mathbb{C})$ by the subgroup of upper-triangular matrices $\B_n(\mathbb{C})$, where $V_k$ is the subspace spanned by the first $k$ columns of a matrix representative in $\GL_n(\mathbb{C})/\B_n(\mathbb{C})$. Lusztig \cite{lusztig94, lusztig98} introduced two remarkable subsets of the real points of $\Fl_n(\mathbb{C})$, called the {\itshape totally positive} and {\itshape totally nonnegative flag varieties}, denoted $\Fl_n^{>0}$ and $\Fl_n^{\ge 0}$, and defined as follows. Let $\GL_n^{>0}$ be the subset of $\GL_n(\mathbb{C})$ of all {\itshape totally positive} matrices, i.e., matrices whose minors are all positive. Then $\Fl_n^{>0}$ is the image of $\GL_n^{>0}$ inside $\GL_n(\mathbb{C})/\B_n(\mathbb{C})$, and $\Fl_n^{\ge 0}$ is its closure. Equivalently, $\Fl_n^{>0}$ (respectively, $\Fl_n^{\ge 0}$) is the set of flags which can be represented by an element of $\GL_n(\mathbb{C})$ whose left-justified minors (i.e.\ those which use an initial subset of columns) are all positive (respectively, nonnegative).

More generally, for any subset $K$ of $\{1, \dots, n-1\}$, we have the partial flag variety $\PFl{K}{n}(\mathbb{C})$, consisting of nested sequences of subspaces of dimensions $k\in K$. Its totally nonnegative part $\PFl{K}{n}^{\ge 0}$ is defined to be the image of $\Fl_n^{\ge 0}$ under the natural projection which forgets the subspaces of dimensions $k\notin K$. Of particular interest is the case $K = \{k\}$, whence we obtain the {\itshape Grassmannian} $\Gr_{k,n}(\mathbb{C})$ and its totally nonnegative part $\Gr_{k,n}^{\ge 0}$. The totally nonnegative parts of Grassmannians and of more general partial flag varieties have been widely studied, with connections to representation theory \cite{lusztig94}, combinatorics \cite{postnikov07}, cluster algebras \cite{fomin_williams_zelevinsky}, high-energy physics \cite{arkani-hamed_bourjaily_cachazo_goncharov_postnikov_trnka16, arkani-hamed_bai_lam17}, mirror symmetry \cite{rietsch_williams19}, topology \cite{galashin_karp_lam22}, and many other topics.

It is well-known that one can view a partial flag variety as an adjoint orbit inside a corresponding Lie algebra. The purpose of this paper is to use the orbit context to study total positivity. We approach this analysis from an algebraic, geometric, and dynamical perspective.

There are two main inspirations for our work. The first is work of Galashin, Karp, and Lam \cite{galashin_karp_lam, galashin_karp_lam19}, who constructed a contractive flow on any totally nonnegative partial flag variety, in order to show that it is homeomorphic to a closed ball. One of our goals was to situate these flows in a more general and geometric context. We will see that these contractive flows are gradient flows in the K\"{a}hler metric on an adjoint orbit. The second inspiration is work of Bloch, Flaschka, and Ratiu \cite{bloch_flaschka_ratiu90}, who studied the tridiagonal Toda flow on an adjoint orbit. They showed that after applying a certain involution, the flow becomes a gradient flow in the K\"{a}hler metric, and then projecting by the moment map gives a homeomorphism from the underlying isospectral manifold onto the moment polytope. Another of our goals was to clarify and extend this construction using total positivity, and to relate it to the work of Galashin, Karp, and Lam above. Here we introduce a generalization of this involution called the {\itshape twist map}, which plays an important role throughout the paper. In order to accomplish these goals, we have developed the fundamentals of total positivity for adjoint orbits.

The paper consists of three main parts. In the first part (\cref{sec_background,sec_tnn_U,sec_tnn_orbit}), we introduce the totally nonnegative part of an adjoint orbit and the twist map. In the second part (\cref{sec_gradient}), we study gradient flows on adjoint orbits in three different Riemannian metrics: the K\"{a}hler, normal, and induced metrics. We focus on characterizing which gradient flows are compatible with positivity. In the third part (\cref{sec_ball,sec_amplituhedron,sec_toda}), we consider two applications of the theory thus developed: to the topology of totally nonnegative flag varieties and amplituhedra, and to the symmetric Toda flow. Below we give further details and highlight our main results.

\subsection*{Adjoint orbits}
Let $\U_n$ denote the group of $n\times n$ unitary matrices, and let $\uu_n$ denote the Lie algebra of $n\times n$ skew-Hermitian matrices. For a weakly decreasing sequence $\bflambda = (\lambda_1, \dots, \lambda_n) \in \mathbb{R}^n$, we let $\Orbit_{\bflambda}$ denote the adjoint orbit inside $\uu_n$ consisting of all matrices with eigenvalues $\ii \lambda_1, \dots, \ii \lambda_n$ (where $\ii = \sqrt{-1}$). We may identify $\Orbit_{\bflambda}$ with a partial flag variety $\PFl{K}{n}(\mathbb{C})$, where $K$ depends on the multiplicities of the entries of $\lambda$. Namely, $K$ is the set of $k\in\{1, \dots, n-1\}$ such that $\lambda_k > \lambda_{k+1}$, and the matrix $L\in\Orbit_{\bflambda}$ corresponds to the flag $V\in\PFl{K}{n}(\mathbb{C})$, where $V_k$ is the span of the eigenvectors of $L$ corresponding to the eigenvalues $\ii\lambda_1, \dots, \ii\lambda_k$. In the generic case (i.e. when $\lambda$ is strictly decreasing), we have $\Orbit_{\bflambda} \cong \Fl_n(\mathbb{C})$. At another extreme we have the case $\bflambda = \bfomega{k} := (1, \dots, 1, 0, \dots, 0)$, with $k$ ones followed by $n-k$ zeros; then $\Orbit_{\bflambda}$ consists of matrices $\ii P$ such that $P$ is a projection matrix of rank $k$, and $\Orbit_{\bflambda} \cong \Gr_{k,n}(\mathbb{C})$.

The totally nonnegative part of $\PFl{K}{n}(\mathbb{C})$ defines a corresponding subset $\Orbit_{\bflambda}^{\ge 0}$, the totally nonnegative part of an adjoint orbit. It is a distinguished subset of the purely imaginary matrices in $\Orbit_{\bflambda}$. Similarly, we obtain the totally positive part $\Orbit_{\bflambda}^{>0}$. We show that in several cases of interest, $\Orbit_{\bflambda}^{\ge 0}$ can be described using notions familiar in the literature (see \cref{complete_to_orbit}, \cref{projection_tnn_description}, and \cref{tridiagonal_orbit}):
\begin{thm}\label{intro_orbit_description}
Let $\ii L\in\Orbit_{\bflambda}$.
\begin{enumerate}[label=(\roman*), leftmargin=*, itemsep=2pt]
\item\label{intro_orbit_description_complete} If $\lambda_1 > \cdots > \lambda_n > 0$, then $\ii L\in\Orbit_{\bflambda}^{>0}$ if and only if $L$ is eventually totally positive, i.e., $L^m\in\GL_n^{>0}$ for some $m > 0$.
\item\label{intro_orbit_description_projection} If $\bflambda = \bfomega{k}$, then $\ii L\in\Orbit_{\bflambda}^{>0}$ (respectively, $\ii L\in\Orbit_{\bflambda}^{\ge 0}$) if and only if all $k\times k$ minors of $L$ are real and positive (respectively, nonnegative).
\item\label{intro_orbit_description_tridiagonal} If $L$ is tridiagonal, then $\ii L\in\Orbit_{\bflambda}^{>0}$ (respectively, $\ii L\in\Orbit_{\bflambda}^{\ge 0}$) if and only if $L$ is real and its entries immediately above and below the diagonal are positive (respectively, nonnegative).
\end{enumerate}

\end{thm}

The tridiagonal subset of $\Orbit_{\bflambda}^{\ge 0}$ (known as a space of {\itshape Jacobi matrices}) will reappear several times in key places throughout the paper.

\subsection*{The twist map}
We introduce an involution $\twist$ on $\Fl_n^{\ge 0}$ called the {\itshape twist map}, defined as follows. Given $V\in\Fl_n^{\ge 0}$, we represent $V$ by a (unique) orthogonal matrix $g$ whose left-justified minors are all nonnegative. Then $\twist(V)$ is defined to be the element represented by the matrix $((-1)^{i+j}g_{j,i})_{1 \le i,j \le n}$, which is obtained by inverting (or transposing) $g$ and changing the sign of every other entry. Amazingly, this operation is compatible with positivity (see \cref{twist_action}):
\begin{thm}\label{intro_twist}
The twist map $\twist$ defines an involution on $\Fl_n^{\ge 0}$ and on $\Fl_n^{>0}$.
\end{thm}

For example, the twist map $\twist$ sends
$$
\begin{bmatrix}
\frac{\sqrt{3}}{2} & -\frac{1}{2\sqrt{2}} & \frac{1}{2\sqrt{2}} \\[6pt]
\frac{\sqrt{3}}{4} & \frac{1}{4\sqrt{2}} & -\frac{5}{4\sqrt{2}} \\[6pt]
\frac{1}{4} & \frac{3\sqrt{3}}{4\sqrt{2}} & \frac{\sqrt{3}}{4\sqrt{2}}
\end{bmatrix}\hspace*{1pt} \text{ to } \hspace*{1pt}\begin{bmatrix}
\frac{\sqrt{3}}{2} & -\frac{\sqrt{3}}{4} & \frac{1}{4} \\[6pt]
\frac{1}{2\sqrt{2}} & \frac{1}{4\sqrt{2}} & -\frac{3\sqrt{3}}{4\sqrt{2}} \\[6pt]
\frac{1}{2\sqrt{2}} & \frac{5}{4\sqrt{2}} & \frac{\sqrt{3}}{4\sqrt{2}}
\end{bmatrix}\hspace*{1pt} \text{ in } \Fl_3^{>0}.
$$
We call $\twist$ the `twist map' since it is analogous to the twist maps introduced by Berenstein, Fomin, and Zelevinsky, but with the key difference that our map is based on the Iwasawa (or $QR$-) decomposition of $\GL_n(\mathbb{C})$, rather than the Bruhat decomposition.

We obtain a corresponding involution for any generic adjoint orbit, given by
$$
\twistorbit : \Orbit_{\bflambda}^{\ge 0} \to \Orbit_{\bflambda}^{\ge 0}, \quad g\Lambda g^{-1} \mapsto \delta_ng^{-1}\Lambda g\delta_n,
$$
related to the dressing transformations of Poisson geometry. Above, $\Lambda$ is the diagonal matrix with diagonal entries $\ii \lambda_1, \dots, \ii \lambda_n$, $\delta_n$ is the diagonal matrix with diagonal entries $1, -1, 1, \dots, (-1)^{n-1}$, and $g\in\U_n$ is chosen so that all its left-justified minors are nonnegative. The key point is that in general, $g^{-1}\Lambda g$ depends on $g$ (and not just on the element $g\Lambda g^{-1}$ of the orbit), and total nonnegativity provides a canonical way of selecting the representative $g$.

\subsection*{Gradient flows}
Inspired by \cite{galashin_karp_lam}, we study flows on $\Orbit_{\bflambda}$ which {\itshape strictly preserve positivity}, which means that the flow sends $\Orbit_{\bflambda}^{\ge 0}$ inside $\Orbit_{\bflambda}^{>0}$ after any positive time. We focus on gradient flows for height functions of the form $L\mapsto \tr(LN)$ (coming from the Killing form) for fixed $N\in\uu_n$, and work in three different Riemannian metrics: the K\"{a}hler, normal, and induced metrics. In several cases we are able to classify which flows strictly preserve positivity.

One such case is when $\Orbit_{\bflambda}\cong\Gr_{k,n}(\mathbb{C})$, in which case the three metrics coincide up to dilation. In this case, we have the following classification (see \cref{positivity_preserving_kahler_grassmannian}, which also contains the corresponding result for $k=1,n-1$):
\begin{thm}\label{intro_positivity_preserving_grassmannian}
Let $2 \le k \le n-2$. Then the gradient flow of $L\mapsto \tr(LN)$ on $\Orbit_{\bfomega{k}}$ strictly preserves positivity if and only if $\ii N$ is real, $N_{i,j} = 0$ for $i-j \not\equiv -1, 0, 1 \hspace*{4pt}\modulo{n}$,
$$
\ii N_{1,2},\, \ii N_{2,3},\, \dots,\, \ii N_{n-1,n},\, (-1)^{k-1}\ii N_{n,1} \ge 0,
$$
and at least $n-1$ of the $n$ inequalities above are strict.
\end{thm}

When $\Orbit_{\bflambda}$ is not isomorphic to a Grassmannian, then the three metrics are different, and their gradient flows exhibit markedly different behavior with respect to positivity. In the case of the K\"{a}hler metric, the flows admit a beautiful explicit solution (see \cref{gradient_flow_kahler}). We use it to obtain the following complete classification (see \cref{positivity_preserving_kahler_full}):
\begin{thm}\label{intro_positivity_preserving_full}
Let $\bflambda\in\mathbb{R}^n$ be weakly decreasing with at least three distinct entries. Then the gradient flow of $L\mapsto \tr(LN)$ on $\Orbit_{\bflambda}$ in the K\"{a}hler metric strictly preserves positivity if and only if $\ii N$ is a real tridiagonal matrix whose entries immediately above and below the diagonal are positive.
\end{thm}

By contrast, we show that in the normal metric, in the generic case there are no flows which strictly preserve positivity (see \cref{positivity_preserving_normal}):
\begin{thm}
Let $\bflambda\in\mathbb{R}^n$ be strictly decreasing. Then for all $N\in\uu_n$, the gradient flow of $L\mapsto \tr(LN)$ on $\Orbit_{\bflambda}$ in the normal metric does not strictly preserve positivity.
\end{thm}

We leave the consideration of positivity-preserving flows in the normal metric for other choices of $\bflambda$ to future work. For the induced metric, our results are much less complete. However, our preliminary investigations indicate that in this case, the existence of gradient flows on $\Orbit_{\bflambda}$ which strictly preserve positivity may depend on the spacing between the entries of $\lambda$; see \cref{induced_extended_example} and \cref{induced_no_preserving}.

We establish analogues of the results stated above for gradient flows on $\Orbit_{\bflambda}$ in the K\"{a}hler, normal, and induced metrics which {\itshape weakly} preserve positivity, i.e., which send $\Orbit_{\bflambda}^{\ge 0}$ inside itself after any positive time (see \cref{positivity_preserving_kahler_grassmannian}, \cref{positivity_preserving_kahler_full}, \cref{induced_extended_example} and \cref{induced_no_preserving}).

\subsection*{Topology}
Galashin, Karp, and Lam \cite{galashin_karp_lam, galashin_karp_lam19} used certain flows which strictly preserve positivity to show that the totally nonnegative part of a partial flag variety (in arbitrary Lie type) is homeomorphic to a closed ball. We rephrase their argument in the orbit language for any gradient flow on $\Orbit_{\bflambda}$ in the K\"{a}hler metric, and show that the height function provides a strict Lyapunov function for such a flow. This leads to the following result (see \cref{compact_invariant_ball}):
\begin{thm}\label{intro_ball}
Suppose that $\bflambda,\bfmu\in\mathbb{R}^k$ such that $\mu_k > \mu_{k+1}$ for all $1 \le k \le n-1$ such that $\lambda_k > \lambda_{k+1}$. Consider the gradient flow of $L\mapsto\tr(LN)$ on $\Orbit_{\bflambda}$ in the K\"{a}hler metric, where $-N\in\Orbit_{\bfmu}$. Let $S$ be a nonempty compact subset of the stable manifold of the global attractor, such that any flow beginning in $S$ remains in the interior of $S$ for all positive time. Then $S$ is homeomorphic to a closed ball, its interior is homeomorphic to an open ball, and its boundary is homeomorphic to a sphere.
\end{thm}

In particular, by applying \cref{intro_ball} in the setting of the gradient flows in \cref{intro_positivity_preserving_grassmannian} and \cref{intro_positivity_preserving_full}, we obtain that $\Orbit_{\bflambda}^{\ge 0}$ is homeomorphic to a closed ball, as shown in \cite{galashin_karp_lam, galashin_karp_lam19}.

We also apply \cref{intro_ball} to study the topology of {\itshape amplituhedra} $\mathcal{A}_{n,k,m}(Z)$. These are generalizations of the totally nonnegative Grassmannian $\Gr_{k,n}^{\ge 0}$, introduced by Arkani-Hamed and Trnka \cite{arkani-hamed_trnka14} in order to give a geometric basis for calculating scattering amplitudes in planar $\mathcal{N} = 4$ supersymmetric Yang--Mills theory. The amplituhedron $\mathcal{A}_{n,k,m}(Z)$ depends on a certain auxiliary $(k+m)\times n$ matrix $Z$, where $m$ is an additional parameter satisfying $k+m \le n$. Much recent work has focused on the combinatorics and topology of amplituhedra. It is believed that every amplituhedron $\mathcal{A}_{n,k,m}(Z)$ is homeomorphic to a closed ball. This is known when $k+m = n$ \cite[Theorem 1.1]{galashin_karp_lam} (in which case $\mathcal{A}_{n,k,m}(Z)$ is $\Gr_{k,n}^{\ge 0}$), when $k=1$ (in which case $\mathcal{A}_{n,k,m}(Z)$ is a cyclic polytope \cite{sturmfels88}), when $m=1$ \cite[Corollary 6.18]{karp_williams19}, for the family of {\itshape cyclically symmetric amplituhedra} \cite[Theorem 1.2]{galashin_karp_lam}, and when $n-k-m = 1$ with $m$ even \cite[Theorem 1.8]{blagojevic_galashin_palic_ziegler19}.

We extend the methods of \cite{galashin_karp_lam} to show that a new family of amplituhedra, which we call {\itshape twisted Vandermonde amplituhedra}, are homeomorphic to closed balls. These are amplituhedra for which the matrix $Z$ arises by applying the twist map $\twist$ to a {\itshape Vandermonde flag} (see \cref{defn_twisted_vandermonde_amplituhedra}). This family of amplituhedra includes all amplituhedra $\mathcal{A}_{n,k,m}(Z)$ satisfying $n-k-m \le 2$. We obtain the following result (see \cref{twisted_vandermonde_amplituhedra_ball} and
\cref{amplituhedra_ball_case_ball}):
\begin{thm}\label{intro_amplituhedra_ball}
Every twisted Vandermonde amplituhedron (in particular, every amplituhedron $\mathcal{A}_{n,k,m}(Z)$ with $n-k-m \le 2$) is homeomorphic to a closed ball, its interior is homeomorphic to an open ball, and its boundary is homeomorphic to a sphere.
\end{thm}

\subsection*{The symmetric Toda flow}
The {\itshape Toda lattice} \cite{toda67b} is an integrable Hamiltonian system which has been widely studied since it was introduced in 1967. It may be viewed as the flow $\dot{L} = [L, \kterm(-\ii L)]$ evolving on an adjoint orbit $\Orbit_{\bflambda}$, where $\kterm(-\ii L)$ is the skew-Hermitian part of $-\ii L$. Classically, $L$ is assumed to be a purely imaginary tridiagonal matrix, but more generally, we can take $L$ to be any element of $\Orbit_{\bflambda}$.

We observe that the Toda flow provides an example of a gradient flow which weakly preserves positivity (in both time directions), in two different ways. First, in the tridiagonal case, the Toda flow is a gradient flow in the normal metric; this follows from work of Bloch \cite{bloch90}. Second, in the general case, the Toda flow starting at a point in $\Orbit_{\bflambda}$ is a twisted gradient flow (see \cref{full_symmetric_toda_gradient}):
\begin{thm}\label{intro_toda}
Let $\bflambda\in\mathbb{R}^n$ be strictly decreasing, and let $L(t)$ denote the Toda flow on $\Orbit_{\bflambda}$ beginning at a point in $\Orbit_{\bflambda}^{\ge 0}$. Then $\twistorbit(L(t))$ is a gradient flow of the function $M\mapsto \tr(MN)$ in the K\"{a}hler metric, where $\ii N$ is the diagonal matrix with diagonal entries $\lambda_1, \dots, \lambda_n$.
\end{thm}

\cref{intro_toda} generalizes a result of Bloch, Flaschka, and Ratiu \cite{bloch_flaschka_ratiu90} on the subset of tridiagonal matrices in $\Orbit_{\bflambda}^{\ge 0}$ (i.e.\ Jacobi matrices). Their construction of the twist map $\twistorbit$ in this case involves an intricate analysis involving the Bruhat decomposition. The perspective of positivity gives a natural way to define $\twistorbit$ on Jacobi matrices, and to generalize it to all of $\Orbit_{\bflambda}^{\ge 0}$.

\subsection*{Outline}
In \cref{sec_background} we recall some background material. In \cref{sec_tnn_U} we introduce the totally nonnegative part of the unitary group $\U_n$ and define the twist map $\twist$. In \cref{sec_tnn_orbit} we introduce the adjoint orbit $\Orbit_{\bflambda}$ and its totally nonnegative part. In \cref{sec_gradient} we study gradient flows on $\Orbit_{\bflambda}$ in the K\"{a}hler, normal, and induced metrics. In \cref{sec_ball} we show that certain subsets of $\Orbit_{\bflambda}$, including $\Orbit_{\bflambda}^{\ge 0}$, are homeomorphic to closed balls. In \cref{sec_amplituhedron} we study gradient flows on amplituhedra and show that certain amplituhedra are homeomorphic to closed balls. In \cref{sec_toda} we study the symmetric Toda flow and its relation to total positivity.

We expect that many of the results and techniques in this paper extend to the case of an arbitrary complex semisimple Lie group $\mathfrak{g}$ and its compact real form $\mathfrak{k}$; the case we consider corresponds to $\mathfrak{g} = \sl_n(\mathbb{C})$ and $\mathfrak{k} = \su_n$ (i.e.\ type $A$). We have decided to focus on this case, and to work instead with $\gl_n(\mathbb{C})$ and $\uu_n$, both for the sake of simplicity and concreteness, and to emphasize the connections with the classical theory of total positivity.

\subsection*{Acknowledgments}
We thank Roger Brockett, Sergey Fomin, Pavel Galashin, Thomas Lam, and Tudor Ra\c{t}iu for valuable discussions, and Jonathan Boretsky for helpful feedback on the paper.

\addtocontents{toc}{\protect\setcounter{tocdepth}{2}}
\section{Background}\label{sec_background}

\noindent In this section, we collect notation and background on matrix Lie groups and Lie algebras, and on total positivity, which we will use throughout the paper. For further details on Lie groups and Lie algebras, we refer to \cite{knapp02}. For further details on total positivity, we refer to \cite{gantmaher_krein50, karlin68, lusztig94, fomin_zelevinsky00a, pinkus10, fallat_johnson11}, as well as the original references.
\subsection{Notation}\label{sec_notation}
Let $\mathbb{N} := \{0, 1, 2, \dots, \}$. For $n\in\mathbb{N}$, we let $[n]$ denote $\{1, 2, \dots, n\}$, and for $i,j\in\mathbb{Z}$, we let $[i,j]$ denote the interval of integers $\{i, i+1, \dots, j\}$. Given a set $S$ and $k\in\mathbb{N}$, we let $\binom{S}{k}$ denote the set of $k$-element subsets of $S$.

Given an $m\times n$ matrix $L$, we let $\transpose{L}$ denote its transpose, and let $\adjoint{L} := \ccon{\transpose{L}}$ denote its conjugate transpose. For subsets $I\subseteq [m]$ and $J\subseteq [n]$, we let $L_{I,J}$ denote the submatrix of $L$ using rows $I$ and columns $J$. If $|I| = |J|$, we let $\Delta_{I,J}(L)$ denote $\det(L_{I,J})$, called a {\itshape minor} of $L$. If $J = [k]$, where $k = |I|$, we call $\Delta_{I,J}(L)$ a {\itshape left-justified minor} of $L$, which we denote by $\Delta_I(L)$. We also let $\sumof{I}$ denote the sum of the elements in $I$, and let $\inv{I}{J}$ denote the number of pairs $(i,j)\in I\times J$ such that $i > j$. We let $\Diag{\lambda_1, \dots, \lambda_n}$ denote the $n\times n$ diagonal matrix with diagonal entries $\lambda_1, \dots, \lambda_n$, and let $\delta_n := \Diag{1, -1, 1, \dots, (-1)^{n-1}}$.

Given a field $\field$ and $n\in\mathbb{N}$, we let $e_1, \dots, e_n$ denote the unit vectors of $\field^n$. We define the following spaces:\vspace*{1pt}
\begin{itemize}[itemsep=3pt]
\item $\mathbb{P}^n(\field) := (\field^{n+1}\setminus\{0\})/\field^{\times} = \text{projective $n$-space over $\field$}$;
\item $\GL_n(\field) := \{\text{invertible } n\times n \text{ matrices with entries in }\field\}$;
\item $\B_n(\field) := \{g\in\GL_n(\field) : g \text{ is upper-triangular}\}$;
\item $\N_n(\field) := \{g\in\B_n(\field) : g_{i,i} = 1 \text{ for } 1 \le i \le n\}$;
\item $\Bminus_n(\field) := \{g\in\GL_n(\field) : g \text{ is lower-triangular}\} = \transpose{\B_n(\field)}$;
\item $\H_n(\field) := \{g\in\GL_n(\field) : g \text{ is diagonal}\}$;
\item $\U_n := \{g\in\GL_n(\mathbb{C}) : \adjoint{g}g = \I_n\}$;
\item $\T_n := \H_n(\mathbb{C})\cap\U_n$;
\item $\O_n := \U_n\cap\GL_n(\mathbb{R})$;
\item $\gl_n(\field) := \{n\times n \text{ matrices with entries in }\field\}$;
\item $\n_n(\field) := \{L\in\gl_n(\field) : L \text{ is strictly upper-triangular}\}$;
\item $\h_n(\field) := \{L\in\gl_n(\field) : L \text{ is diagonal}\}$;
\item $\uu_n := \{L\in\gl_n(\mathbb{C}) : \adjoint{L} + L = 0\}$;
\item $\oo_n := \uu_n\cap\gl_n(\mathbb{R})$.
\end{itemize}

The {\itshape Lie bracket} $[\cdot,\cdot]$ on $\gl_n(\field)$ is given by
$$
[L,M] := LM - ML \quad \text{ for all } L,M\in\gl_n(\field).
$$
We define the {\itshape adjoint operator} $\ad_L$ for $L\in\gl_n(\field)$ by
$$
\ad_L(M) := [L,M] \quad \text{ for all } M\in\gl_n(\field).
$$
When $\field = \mathbb{C}$, we define the {\itshape exponential map} $\exp : \gl_n(\mathbb{C}) \to \GL_n(\mathbb{C})$ by
$$
\exp(L) := \sum_{m=0}^\infty \frac{1}{m!}L^m = \lim_{m\to\infty}\Big(\I_n + \frac{1}{m}L\Big)^m.
$$

We recall some properties of the determinant:
\begin{prop}[{\cite[Chapter I]{gantmacher59}}]\label{determinant_properties}~
\begin{enumerate}[label=(\roman*), leftmargin=*, itemsep=2pt]
\item\label{determinant_properties_laplace} (Laplace expansion) Let $M$ be an $n\times n$ matrix, let $0 \le k \le n$, and let $I\in\binom{[n]}{k}$. Then
\begin{align}\label{laplace}
\det(M) = \sum_{J\in\binom{[n]}{k}}(-1)^{\sumof{I} + \sumof{J}}\Delta_{I,J}(M)\Delta_{[n]\setminus I,[n]\setminus J}(M).
\end{align}
\item\label{determinant_properties_cauchy-binet} (Cauchy--Binet identity) Let $L$ be an $m\times n$ matrix, and let $M$ be an $n\times p$ matrix. Then for $1 \le k \le m,p$, we have
\begin{align}\label{cauchy-binet}
\Delta_{I,J}(LM) = \sum_{K\in\binom{[n]}{k}}\Delta_{I,K}(L)\Delta_{K,J}(M) \quad \text{ for all } I\in\textstyle\binom{[m]}{k} \text{ and } J\in\binom{[p]}{k}.
\end{align}
\item\label{determinant_properties_jacobi} (Jacobi's formula) Let $g\in\GL_n(\field)$, and let $I,J\subseteq [n]$ have the same size. Then
\begin{align}\label{jacobi}
\Delta_{I,J}(g^{-1}) = \frac{(-1)^{\sumof{I} + \sumof{J}}}{\det(g)}\Delta_{[n]\setminus J,[n]\setminus I}(g).
\end{align}
\item\label{determinant_properties_vandermonde} (Vandermonde's determinantal identity) We have
\begin{align}\label{vandermonde}
\det((\lambda_i^{j-1})_{1 \le i,j \le n}) = \prod_{1 \le i < j \le n}(\lambda_j - \lambda_i).
\end{align}

\end{enumerate}

\end{prop}

We have the {\itshape Trotter product formula} for the exponential map:
\begin{prop}[{\cite[p.\ 256]{abraham_marsden_ratiu88}}]\label{trotter}
Let $L,M\in\gl_n(\mathbb{C})$. Then
\begin{gather*}
\exp(L+M) = \lim_{m\to\infty}\big(\hspace*{-2pt}\exp(\textstyle\frac{1}{m}L)\exp(\frac{1}{m}M)\big)^m.
\end{gather*}

\end{prop}

We also recall a classical result of Perron \cite{perron07}:
\begin{thm}[Perron--Frobenius {\cite[Theorem XIII.2.1]{gantmacher59}}]\label{perron}
Let $A$ be an $n\times n$ matrix with positive real entries, and let $r$ be the spectral radius of $A$.
\begin{enumerate}[label=(\roman*), leftmargin=*, itemsep=2pt]
\item\label{perron_eigenvalue} The value $r$ is the unique eigenvalue of $A$ with modulus $r$, and it has algebraic multiplicity $1$.
\item\label{perron_eigenvector} There exists $x\in\mathbb{R}_{>0}^n$ such that $Ax = rx$.
\end{enumerate}

\end{thm}

\subsection{Partial flag varieties}\label{sec_background_flags}
We now introduce partial flag varieties inside $\field^n$.
\begin{defn}\label{defn_Fl}
Let $\field$ be a field and $n\in\mathbb{N}$. Given a subset $K = \{k_1 < \cdots < k_l\} \subseteq [n-1]$, let $\P{K}{n}(\field)$ denote the subgroup of $\GL_n(\field)$ of block upper-triangular matrices with diagonal blocks of sizes $k_1, k_2 - k_1, \dots, k_l - k_{l-1}, n - k_l$. We define the {\itshape partial flag variety}
$$
\PFl{K}{n}(\field) := \GL_n(\field)/\P{K}{n}(\field).
$$
We have the embedding
\begin{align}\label{plucker_embedding}
\begin{gathered}
\PFl{K}{n}(\field) \hookrightarrow \mathbb{P}^{\left(\hspace*{-1pt}\binom{n}{k_1}-1\right)}(\field) \times \cdots \times \mathbb{P}^{\left(\hspace*{-1pt}\binom{n}{k_l}-1\right)}(\field), \\
g \mapsto \Big((\Delta_I(g))_{I\in\binom{[n]}{k_1}}, \dots, (\Delta_I(g))_{I\in\binom{[n]}{k_l}}\Big).
\end{gathered}
\end{align}
(We can check that the right-hand side of the second line only depends on $g$ modulo the right action of $\P{K}{n}(\field)$.) We call the left-justified minors $\Delta_I(g)$ appearing above the {\itshape Pl\"{u}cker coordinates} of $g\in\PFl{K}{n}(\field)$ (also known as {\itshape flag minors}).

We may identify $\PFl{K}{n}(\field)$ with the variety of partial flags of subspaces in $\field^n$
$$
\{V = (V_{k_1}, \dots, V_{k_l}) : 0 \subset V_{k_1} \subset \cdots \subset V_{k_l} \subset \field^n \text{ and } \dim(V_{k_i}) = k_i \text{ for } 1 \le i \le l\}.
$$
The identification sends $g\in\GL_n(\field)/\P{K}{n}(\field)$ to the tuple $(V_k)_{k\in K}$, where $V_k$ is the span of the first $k$ columns of $g$.

Note that for any $K'\subseteq K$, we have $\P{K}{n}(\field) \subseteq \P{K'}{n}(\field)$. This gives a projection map
\begin{align}\label{defn_Fl_surjection}
\PFl{K}{n}(\field) \twoheadrightarrow \PFl{K'}{n}(\field).
\end{align}
In terms of partial flags of subspaces, the map \eqref{defn_Fl_surjection} retains only the subspaces whose dimensions lie in $K'$.

There are two instances of $\PFl{K}{n}(\field)$ which will be of particular interest to us. If $K = [n-1]$, then $\PFl{K}{n}(\field)$ is the {\itshape complete flag variety} of $\field^n$, which we denote by $\Fl_n(\field)$. If $K$ is the singleton $\{k\}$, then $\PFl{K}{n}(\field)$ is the {\itshape Grassmannian} of $k$-dimensional subspaces of $\field^n$, which we denote by $\Gr_{k,n}(\field)$. We represent an element of $\Gr_{k,n}(\field)$ by an $n\times k$ matrix of rank $k$ modulo column operations. We also extend the definition of $\Gr_{k,n}(\field)$ to $k=0$ and $k=n$.
\end{defn}

\begin{eg}\label{eg_Fl}
Let $n := 4$ and $K := \{1, 3\}$. Then
$$
\P{\{1,3\}}{4}(\field) = \left\{\begin{bmatrix}
\ast & \ast & \ast & \ast \\
0 & \ast & \ast & \ast \\
0 & \ast & \ast & \ast \\
0 & 0 & 0 & \ast
\end{bmatrix}\right\}\subseteq\GL_4(\field) \quad \text{ and } \quad \PFl{\{1,3\}}{4}(\field) = \GL_4(\field)/\P{\{1,3\}}{4}(\field). 
$$
We can write a generic element of $\PFl{\{1,3\}}{4}(\field)$ as
$$
g = \begin{bmatrix}
1 & 0 & 0 & 0 \\
a & 1 & 0 & 0 \\
b & 0 & 1 & 0 \\
c & d & e & 1
\end{bmatrix}, \quad \text{ where } a,b,c,d,e\in\field. 
$$
(Note that not all elements $g$ of $\PFl{\{1,3\}}{4}(\field)$ are of this form, such as those with $g_{1,1} = 0$.) Then the embedding \eqref{plucker_embedding} takes $g$ to
\begin{multline*}
\big((\Delta_1(g) : \Delta_2(g) : \Delta_3(g) : \Delta_4(g)), (\Delta_{123}(g) : \Delta_{124}(g) : \Delta_{134}(g) : \Delta_{234}(g))\big) \\
= \big((1 : a : b : c), (1 : e : -d : c-ad-be)\big) \in \mathbb{P}^3(\field)\times\mathbb{P}^3(\field).
\end{multline*}
Furthermore, we can identify $g\in\PFl{\{1,3\}}{4}(\field)$ with the partial flag $(V_1, V_3)$, where $V_1\subseteq\field^4$ is the span of the first column of $g$, and $V_3\subseteq\field^4$ is the span of the first three columns of $g$.
\end{eg}

\subsection{Total positivity and total nonnegativity}\label{sec_total_positivity}
We now introduce the totally positive and totally nonnegative parts of several of the spaces defined above.
\begin{defn}\label{defn_tnn_intro}
Let $n\in\mathbb{N}$. We define the {\itshape totally positive} parts of the following spaces:
\begin{itemize}[itemsep=3pt]
\item $\mathbb{P}^n_{>0} := \{(x_0 : \cdots : x_n)\in\mathbb{P}^n(\mathbb{R}) : x_0, \dots, x_n > 0\}$;
\item $\GL_n^{>0} := \{g \in \GL_n(\mathbb{R}) : \Delta_{I,J}(g) > 0 \text{ for all } I,J \subseteq [n] \text{ with } |I| = |J|\}$;
\item $\H_n^{>0} := \{g\in\H_n(\mathbb{R}) : g_{i,i} > 0 \text{ for } 1 \le i \le n\}$;
\item $\gl_n^{>0} := \{L\in\gl_n(\mathbb{R}) : \exp(tL)\in\GL_n^{>0} \text{ for all } t > 0\}$ \\[2pt]
$\phantom{\gl_n^{>0}}\mathbin{\phantom{:}}= \{L\in\gl_n(\mathbb{R}) : L \text{ is tridiagonal and } L_{i,i+1}, L_{i+1,i} > 0 \text{ for } 1 \le i \le n-1\}$.
\end{itemize}
We also define the {\itshape totally nonnegative} parts by taking closures in the Euclidean topology:
\begin{itemize}[itemsep=3pt]
\item $\mathbb{P}^n_{\ge 0} := \overline{\mathbb{P}^n_{>0}} = \{(x_0 : \cdots : x_n)\in\mathbb{P}^n(\mathbb{R}) : x_0, \dots, x_n \ge 0\}$;
\item $\GL_n^{\ge 0} := \overline{\GL_n^{>0}} = \{g \in \GL_n(\mathbb{R}) : \Delta_{I,J}(g) \ge 0 \text{ for all } I,J \subseteq [n] \text{ with } |I| = |J|\}$;
\item $\gl_n^{\ge 0} := \overline{\gl_n^{>0}} = \{L\in\gl_n(\mathbb{R}) : \exp(tL)\in\GL_n^{\ge 0} \text{ for all } t \ge 0\}$ \\[2pt]
$\phantom{\gl_n^{\ge 0} := \overline{\gl_n^{>0}}} = \{L\in\gl_n(\mathbb{R}) : L \text{ is tridiagonal and } L_{i,i+1}, L_{i+1,i} \ge 0 \text{ for } 1 \le i \le n-1\}$.
\end{itemize}
(We do not consider $\H_n^{\ge 0}$, since $\H_n^{>0}$ is already closed.)

The alternative descriptions of $\gl_n^{>0}$, $\GL_n^{\ge 0}$, and $\gl_n^{\ge 0}$ above are due, respectively, to Karlin \cite[Theorem 3.3.4]{karlin68}, Gantmakher and Krein \cite[Lemma p.\ 18]{gantmakher_krein37}, and Loewner \cite{loewner55} (cf.\ \cite{rietsch97}).

We note that $\GL_n^{>0}$ and $\GL_n^{\ge 0}$ are semigroups by \eqref{cauchy-binet}. Also, $\gl_n^{>0}$ and $\gl_n^{\ge 0}$ are convex cones.
\end{defn}

\begin{eg}\label{eg_tnn_GL_n=2}
We have $\GL_2^{>0} = \left\{\begin{bmatrix}a & b \\[3pt] c & d + \frac{bc}{a} \\[1pt]\end{bmatrix} : a,b,c,d > 0\right\}$.
\end{eg}

\begin{defn}[\cite{lusztig94,lusztig98}]\label{defn_tnn_Fl}
Let $n\in\mathbb{N}$ and $K\subseteq [n-1]$. We define the {\itshape totally positive part} of $\PFl{K}{n}(\mathbb{C})$, denoted by $\PFl{K}{n}^{>0}$, as the image of $\GL_n^{>0}$ inside $\PFl{K}{n}(\mathbb{C}) = \GL_n(\mathbb{C})/\P{K}{n}(\mathbb{C})$. We define the {\itshape totally nonnegative part} of $\PFl{K}{n}(\mathbb{C})$ by taking the closure in the Euclidean topology:
$$
\PFl{K}{n}^{\ge 0} := \overline{\PFl{K}{n}^{>0}}.
$$
Note that for any $K'\subseteq K$, the projection map $\PFl{K}{n}(\mathbb{C}) \twoheadrightarrow \PFl{K'}{n}(\mathbb{C})$ from \eqref{defn_Fl_surjection} restricts to surjections
\begin{align}\label{defn_tnn_Fl_surjections}
\PFl{K}{n}^{>0} \twoheadrightarrow \PFl{K'}{n}^{>0} \quad \text{ and } \quad \PFl{K}{n}^{\ge 0} \twoheadrightarrow \PFl{K'}{n}^{\ge 0}.
\end{align}

\end{defn}

We remark that we could instead have defined $\PFl{K}{n}^{>0}$ and $\PFl{K}{n}^{\ge 0}$ by replacing $\mathbb{C}$ with $\mathbb{R}$. It will turn out to be more convenient for us to work over $\mathbb{C}$.
\begin{eg}\label{eg_tnn_Fl}
We have
\begin{gather*}
\Fl_3^{>0} = \left\{\begin{bmatrix}1 & 0 & 0 \\ a+c & 1 & 0 \\ bc & b & 1\end{bmatrix} : a,b,c > 0\right\} \quad \text{ and } \quad \Gr_{2,4}^{>0} = \left\{\begin{bmatrix}1 & 0 \\ a & b \\ 0 & 1 \\ -c & d\end{bmatrix} : a,b,c,d > 0\right\}.\qedhere
\end{gather*}

\end{eg}

\begin{rmk}\label{proj_tnn_GL_to_Fl}
It follows from \cref{defn_tnn_Fl} that the image of $\GL_n^{\ge 0}$ inside $\PFl{K}{n}(\mathbb{C})$ is contained in $\PFl{K}{n}^{\ge 0}$. However, this containment is strict unless $K = \emptyset$. For example, the element $\scalebox{0.8}{$\begin{bmatrix}0 & -1 \\ 1 & 0\end{bmatrix}$}\in\Fl_2^{\ge 0}$ cannot be represented by an element of $\GL_2^{\ge 0}$.
\end{rmk}

One can show that the Pl\"{u}cker embedding \eqref{plucker_embedding} is compatible with total positivity (see \cref{tnn_Fl}), in that it takes $\PFl{K}{n}^{>0}$ inside $\mathbb{P}^{\left(\hspace*{-1pt}\binom{n}{k_1}-1\right)}_{>0} \times \cdots \times \mathbb{P}^{\left(\hspace*{-1pt}\binom{n}{k_l}-1\right)}_{>0}$, and similarly with ``${>}\hspace*{2pt}0$'' replaced with ``${\ge}\hspace*{2pt}0$''. It is natural to ask whether the preimage of $\mathbb{P}^{\left(\hspace*{-1pt}\binom{n}{k_1}-1\right)}_{>0} \times \cdots \times \mathbb{P}^{\left(\hspace*{-1pt}\binom{n}{k_l}-1\right)}_{>0}$ equals $\PFl{K}{n}^{>0}$, and similarly with ``${>}\hspace*{2pt}0$'' replaced with ``${\ge}\hspace*{2pt}0$''. This motivates the following definition.
\begin{defn}\label{defn_plucker_positive}
Let $n\in\mathbb{N}$ and $K\subseteq [n-1]$. We define the {\itshape Pl\"{u}cker-positive part} of $\PFl{K}{n}(\mathbb{C})$, denoted by $\PFl{K}{n}^{\Delta > 0}$, as the preimage of $\mathbb{P}^{\left(\hspace*{-1pt}\binom{n}{k_1}-1\right)}_{>0} \times \cdots \times \mathbb{P}^{\left(\hspace*{-1pt}\binom{n}{k_l}-1\right)}_{>0}$ under the Pl\"{u}cker embedding \eqref{plucker_embedding}. That is, $\PFl{K}{n}^{\Delta > 0}$ consists of all $V\in \PFl{K}{n}(\mathbb{C})$ such that for every $k\in K$, we have
$$
\Delta_I(V) > 0 \quad \text{ for all } I\in\textstyle\binom{[n]}{k}.
$$
We similarly define the {\itshape Pl\"{u}cker-nonnegative part} $\PFl{K}{n}^{\Delta \ge 0}$ by replacing ``${>}\hspace*{2pt}0$'' with ``${\ge}\hspace*{2pt}0$'' everywhere above.
\end{defn}

\begin{eg}\label{eg_plucker_positive}
We consider an example when $K := \{1,3\}$ and $n := 4$. Let $V = (V_1, V_3) \in \PFl{\{1,3\}}{4}(\mathbb{C})$ be represented by the matrix
$$
g := \begin{bmatrix}
1 & 0 & 0 & 0 \\
1 & 1 & 0 & 0 \\
1 & 1 & 2 & 0 \\
1 & 0 & 1 & 1
\end{bmatrix}.
$$
Then $V\in\PFl{\{1,3\}}{4}^{\Delta > 0}$, since all its Pl\"{u}cker coordinates are positive:
\begin{gather*}
\Delta_1(V) = \Delta_2(V) = \Delta_3(V) = \Delta_4(V) = 1, \\
\Delta_{123}(V) = \Delta_{234}(V) = 2, \quad \Delta_{124}(V) = \Delta_{134}(V) = 1.
\end{gather*}
However, we can verify that $V\notin\PFl{\{1,3\}}{4}^{>0}$, for example, by showing that $gh\notin\GL_4^{>0}$ for all $h\in\B_4(\mathbb{C})$ (cf.\ \cite[Example 10.1]{chevalier11}).
\end{eg}

As we observed above, Lusztig's notion of total positivity is stronger than Pl\"{u}cker positivity:
\begin{lem}\label{tnn_Fl}
Let $n\in\mathbb{N}$ and $K\subseteq [n-1]$.
\begin{enumerate}[label=(\roman*), leftmargin=*, itemsep=2pt]
\item\label{tnn_Fl_tp} We have $\PFl{K}{n}^{>0} \subseteq \PFl{K}{n}^{\Delta > 0}$. That is, if $V\in\PFl{K}{n}^{>0}$, then for every $k\in K$ we have
$$
\Delta_I(V) > 0 \quad \text{ for all } I\in\textstyle\binom{[n]}{k}.
$$
\item\label{tnn_Fl_tnn} We have $\PFl{K}{n}^{\ge 0} \subseteq \PFl{K}{n}^{\Delta \ge 0}$. That is, if $V\in\PFl{K}{n}^{\ge 0}$, then for every $k\in K$ we have
$$
\Delta_I(V) \ge 0 \quad \text{ for all } I\in\textstyle\binom{[n]}{k}.
$$
\end{enumerate}
In other words, the Pl\"{u}cker embedding \eqref{plucker_embedding} preserves total positivity and total nonnegativity.
\end{lem}

The following result of Bloch and Karp \cite{bloch_karp2} characterizes when Lusztig's notion of total positivity coincides with Pl\"{u}cker positivity. We refer to \cite{bloch_karp2} for further background and previous related work.
\begin{thm}[{Bloch and Karp \cite{bloch_karp2}}]\label{tnn_Fl_converse}
Let $K\subseteq [n-1]$. Then the following are equivalent:
\begin{enumerate}[label=(\roman*), leftmargin=*, itemsep=2pt]
\item\label{tnn_Fl_converse_tp} $\PFl{K}{n}^{>0} = \PFl{K}{n}^{\Delta > 0}$;
\item\label{tnn_Fl_converse_tnn} $\PFl{K}{n}^{\ge 0} = \PFl{K}{n}^{\Delta \ge 0}$; and
\item\label{tnn_Fl_converse_consecutive} the set $K$ consists of consecutive integers.
\end{enumerate}

\end{thm}

We now make several comments about the notion of Pl\"{u}cker positivity.
\begin{rmk}\label{naive_part}
The space $\PFl{K}{n}^{\Delta \ge 0}$ was explicitly introduced by Arkani-Hamed, Bai, and Lam \cite[Section 6.3]{arkani-hamed_bai_lam17}, who called it the {\itshape naive nonnegative part}. Indeed, this space arises naturally in the physics of scattering amplitudes, in particular, for {\itshape loop amplituhedra} \cite{arkani-hamed_trnka14}. For example, the space $\PFl{\{k,k+2\}}{n}^{\Delta \ge 0}$ is a special case of a {\itshape $1$-loop amplituhedron}; the case $k=1$ was studied in detail by Bai, He, and Lam \cite{bai_he_lam16}.
\end{rmk}

\begin{rmk}\label{cyclic_action}
An important aspect of the applications to physics mentioned in \cref{naive_part} is the {\itshape cyclic symmetry} of $\PFl{K}{n}^{\Delta \ge 0}$, in the case that all elements of $K$ have the same parity. An important special case is when $K = \{k\}$ is a singleton, so that $\PFl{K}{n}^{\ge 0} = \PFl{K}{n}^{\Delta \ge 0} = \Gr_{k,n}^{\ge 0}$; see \cite[Section 4]{karp19} for a survey of various applications of the cyclic symmetry for $\Gr_{k,n}^{\ge 0}$. The cyclic action is defined as follows. Let $\sigma\in\GL_n(\mathbb{C})$ be the signed permutation matrix
$$
\sigma := \begin{bmatrix}
0 & 1 & 0 & \cdots & 0 \\
0 & 0 & 1 & \cdots & 0 \\
\vdots & \vdots & \vdots & \ddots & \vdots \\
0 & 0 & 0 & \cdots & 1 \\
\pm 1 & 0 & 0 & \cdots & 0
\end{bmatrix},
$$
where the bottom-left entry is $1$ is all elements of $K$ are odd, and $-1$ if all elements of $K$ are even. Then $\sigma$ acts on $\PFl{K}{n}(\mathbb{C})$; it has order $n$, since $\sigma^n = \pm\I_n$. In terms of Pl\"{u}cker coordinates, $\sigma$ acts by rotating the set $[n]$. In particular, $\sigma$ preserves $\PFl{K}{n}^{\Delta \ge 0}$. However, unless $K = \{k\}$ is a singleton, then $\sigma$ does {\itshape not} preserve $\PFl{K}{n}^{>0}$; see \cite{bloch_karp2}.
\end{rmk}

\begin{rmk}
While we will use Lusztig's notion of total positivity throughout the paper, most of our proofs only use the weaker notion of Pl\"{u}cker positivity (via \cref{tnn_Fl}), and therefore the corresponding results hold for both notions of positivity. An important exception is our classification of gradient flows on an adjoint orbit which preserve positivity in the K\"{a}hler metric (\cref{positivity_preserving_kahler_full}), where for certain orbits the classification differs depending on which notion of positivity one uses; see \cref{plucker_gradient_flows}.
\end{rmk}

\begin{rmk}\label{plucker_positive_projection}
Note that for $V\in\PFl{K}{n}(\mathbb{C})$, we have $V\in\PFl{K}{n}^{\Delta > 0}$ (respectively, $V\in\PFl{K}{n}^{\Delta \ge 0}$) if and only if $V_k\in\Gr_{k,n}^{>0}$ (respectively, $V_k\in\Gr_{k,n}^{\ge 0}$) for all $k\in K$. This follows from \cref{defn_plucker_positive} along with \cref{tnn_Fl_converse} applied to $\Gr_{k,n}(\mathbb{C})$.
\end{rmk}

We will also need the following result from \cite{bloch_karp2}:
\begin{lem}[{Bloch and Karp \cite{bloch_karp2}}]\label{tnn_counterexample}
Let $V\in\Gr_{k,n}^{\ge 0}$ and $W\in\Gr_{k+1,n}^{\ge 0}$ such that $V\subseteq W$. If $e_1 + ce_n\in V$ for some $c\in\mathbb{R}$, then $e_1 \in W$.
\end{lem}

We have the following refinement of \cref{tnn_Fl_converse} in the case of $\Fl_n^{>0}$, which follows from a classical result of Fekete \cite{fekete_polya12}.
\begin{lem}[Fekete {\cite[Theorem V.8]{gantmaher_krein50}}]\label{fekete}
Let $V\in\Fl_n(\mathbb{C})$. Then $V\in\Fl_n^{>0}$ if and only if
\begin{gather*}
\Delta_{[i,j]}(V) > 0 \quad \text{ for all } 1 \le i \le j \le n.
\end{gather*}

\end{lem}

The group $\GL_n(\mathbb{C})$ acts on $\PFl{K}{n}(\mathbb{C})$ by left multiplication. This action is compatible with total positivity:
\begin{lem}\label{matrices_acting_on_flags}
Let $n\in\mathbb{N}$ and $K\subseteq [n-1]$.
\begin{enumerate}[label=(\roman*), leftmargin=*, itemsep=4pt]
\item\label{matrices_acting_on_flags_tp} We have $g\cdot\PFl{K}{n}^{\ge 0} \subseteq \PFl{K}{n}^{>0}$ for all $g\in\GL_n^{>0}$.
\item\label{matrices_acting_on_flags_tnn} We have $g\cdot\PFl{K}{n}^{>0} \subseteq \PFl{K}{n}^{>0}$ for all $g\in\GL_n^{\ge 0}$.
\end{enumerate}

\end{lem}

\begin{proof}
By \eqref{defn_tnn_Fl_surjections}, it suffices to prove the result for the complete flag variety (i.e.\ when $K = [n-1]$). This case follows from \cref{tnn_Fl_converse} and the Cauchy--Binet identity \eqref{cauchy-binet}.
\end{proof}

\begin{rmk}\label{torus_action}
The torus $\H_n(\mathbb{C})$ acts on $\PFl{K}{n}(\mathbb{C})$ by left multiplication. Then \cref{matrices_acting_on_flags} implies that the totally positive part of the torus $\H_n^{>0}$ acts on $\PFl{K}{n}^{>0}$ and $\PFl{K}{n}^{\ge 0}$. This torus action will arise repeatedly throughout the paper.
\end{rmk}

A classical result of Gantmakher and Krein \cite{gantmakher_krein37} (cf.\ \cite[Chapter V]{gantmaher_krein50}) gives an explicit connection between $\GL_n^{>0}$ and $\Fl_n^{>0}$. We will need the following refinement for matrices whose minors of a fixed order are positive. Our proof follows \cite{gantmakher_krein37}, and is based on the Perron--Frobenius theorem.
\begin{thm}\label{gk_k}
Let $1 \le k \le n$, and let $g$ be a complex $n\times n$ matrix whose $k\times k$ minors are all positive.
\begin{enumerate}[label=(\roman*), leftmargin=*, itemsep=2pt]
\item\label{gk_k_eigenvalues} The eigenvalues of $g$ over $\mathbb{C}$ may be enumerated as $\lambda_1, \dots, \lambda_n$, such that
$$
|\lambda_1| \ge \cdots \ge |\lambda_k| > |\lambda_{k+1}| \ge \cdots \ge |\lambda_n| \quad \text{ and } \quad \lambda_1\cdots\lambda_k > 0.
$$
\item\label{gk_k_eigenvectors} Let $V$ be the linear span of all generalized eigenvectors of $g$ corresponding to the eigenvalues $\lambda_1, \dots, \lambda_k$. That is, $V$ is the unique $g$-invariant subspace such that $g$ restricted to $V$ has eigenvalues $\lambda_1, \dots, \lambda_k$. Then $V\in\Gr_{k,n}^{>0}$.
\end{enumerate}

\end{thm}

\begin{proof}
Consider $g$ acting on the exterior power $\bigwedge^{\hspace*{-1pt}k}(\mathbb{C}^n)$, which we regard as an $(\binom{n}{k}\times\binom{n}{k})$-matrix with entries $\Delta_{I,J}(g)$ for $I,J\in\binom{[n]}{k}$, and eigenvalues $\prod_{i\in I}\lambda_i$ for $I\in\binom{[n]}{k}$. By assumption, this matrix has positive entries, and so the result follows from \cref{perron}.
\end{proof}

\begin{cor}[Gantmakher and Krein {\cite[Theorems 10 and 13]{gantmakher_krein37}}]\label{gk}
Let $g\in\GL_n^{>0}$.
\begin{enumerate}[label=(\roman*), leftmargin=*, itemsep=2pt]
\item\label{gk_eigenvalues} The matrix $g$ has $n$ distinct positive real eigenvalues $\lambda_1 > \cdots > \lambda_n$.
\item\label{gk_eigenvectors} If we diagonalize $g$ as
$$
h^{-1}gh = \Diag{\lambda_1, \dots, \lambda_n}, \quad \text{ where } h\in\GL_n(\mathbb{C}),
$$
then the projection of $h$ to $\Fl_n(\mathbb{C})$ lies in $\Fl_n^{>0}$. That is, the complete flag generated by the eigenvectors of $g$, ordered by decreasing eigenvalue, is totally positive.
\end{enumerate}

\end{cor}

We will later state a converse to part \ref{gk_eigenvectors}; see \cref{complete_to_orbit}. It implies that for every $V\in\Fl_n^{>0}$, there exists $g\in\GL_n^{>0}$ such that $g\cdot V = V$.

\begin{eg}\label{eg_gk}
We illustrate \cref{gk} for the matrix
$$
g := \begin{bmatrix}
1 & 2 & 1 \\
1 & 3 & 2 \\
1 & 4 & 4
\end{bmatrix}\in\GL_3^{>0}.
$$
We diagonalize $g$ as follows:
$$
h^{-1}gh = \begin{bmatrix}
\frac{7+3\sqrt{5}}{2} & 0 & 0 \\[3pt]
0 & 1 & 0 \\[1pt]
0 & 0 & \frac{7-3\sqrt{5}}{2}
\end{bmatrix}, \quad \text{ where } h := \begin{bmatrix}
3 + \sqrt{5} & -2 & 3 - \sqrt{5} \\[2pt]
4 + 2\sqrt{5} & -1 & 4 - 2\sqrt{5} \\[2pt]
7 + 3\sqrt{5} & 2 & 7 - 3\sqrt{5}
\end{bmatrix}\in\GL_3(\mathbb{C}).
$$
We can verify (e.g.\ from \cref{fekete}) that the projection of $h$ to $\Fl_3(\mathbb{C})$ lies in $\Fl_3^{>0}$.
\end{eg}

We have the following analogue of \cref{gk} for $\GL_n^{\ge 0}$. Its statement is more subtle, because not all elements of $\GL_n^{\ge 0}$ are diagonalizable (such as \scalebox{0.8}{$\begin{bmatrix}1 & 1 \\ 0 & 1\end{bmatrix}$}).
\begin{cor}\label{gk_tnn}
Let $g\in\GL_n^{\ge 0}$.
\begin{enumerate}[label=(\roman*), leftmargin=*, itemsep=2pt]
\item\label{gk_tnn_eigenvalues} The matrix $g$ has $n$ nonnegative real eigenvalues $\lambda_1 \ge \cdots \ge \lambda_n$ (including multiplicities).
\item\label{gk_tnn_eigenvectors} Let $K := \{i \in [n-1] : \lambda_i > \lambda_{i+1}\}$, and take $h\in\GL_n(\mathbb{C})$ such that $h^{-1}gh$ is the Jordan form of $g$, with Jordan blocks ordered by decreasing eigenvalue. Then the projection of $h$ to $\PFl{K}{n}(\mathbb{C})$ lies in $\PFl{K}{n}^{\ge 0}$. That is, the flag in $\PFl{K}{n}(\mathbb{C})$ generated by the generalized eigenvectors of $g$, ordered by decreasing eigenvalue, is totally nonnegative.
\end{enumerate}

\end{cor}

\begin{proof}
This follows from \cref{gk}, using the fact that $\GL_n^{\ge 0} = \overline{\GL_n^{>0}}$.
\end{proof}

\subsection{The cell decomposition of \texorpdfstring{$\Fl_n^{\ge 0}$}{the totally nonnegative part of Fl(n)}}\label{sec_cell_decomposition}
We recall a decomposition of $\Fl_n^{\ge 0}$ introduced by Lusztig \cite{lusztig94}.
\begin{defn}[{\cite[Chapter 2]{bjorner_brenti05}}]\label{defn_symmetric_group}
For $0 \le k \le n$, we define the partial order $\le$ on $\binom{[n]}{k}$, called the {\itshape Gale order}, as follows:
$$
\{i_1 < \cdots < i_k\} \le \{j_1 < \cdots < j_k\} \quad \iff \quad i_1 \le j_1, \dots, i_k \le j_k.
$$
Given $n\in\mathbb{N}$, let $\mathfrak{S}_n$ denote the symmetric group of all permutations of $[n]$. We define the partial order $\le$ on $\mathfrak{S}_n$, called the {\itshape (strong) Bruhat order}, as follows:
$$
v \le w \quad \iff \quad v([k]) \le w([k])\; \text{ for } 1 \le k \le n-1.
$$
The Bruhat order on $\mathfrak{S}_n$ has the minimum $\id := (i \mapsto i)$ and the maximum $w_0 := (i \mapsto n+1-i)$, and is graded by the function $\ell : \mathfrak{S}\to\mathbb{N}$. For example, the Hasse diagram of $\mathfrak{S}_3$ is shown in \cref{figure_S3}.

For $w\in\mathfrak{S}_n$, we define the {\itshape (signed) permutation matrix} $\mathring{w}\in\GL_n(\mathbb{C})$ by $$
\mathring{w}_{i,j} := \begin{cases}
\pm 1, & \text{ if $i = w(j)$}; \\
0, & \text{ otherwise},
\end{cases}\quad \text{ for } 1 \le i,j \le n,
$$
where the signs are chosen so that all left-justified minors of $\mathring{w}$ are nonnegative. Note that
$$
\mathring{(w^{-1})} = \delta_n(\mathring{w})^{-1}\delta_n.
$$
We will also regard $\mathring{w}$ as an element of $\Fl_n(\mathbb{C})$.
\end{defn}

\begin{figure}[ht]
\begin{center}
$$
\begin{tikzpicture}[baseline=(current bounding box.center),scale=1.0]
\pgfmathsetmacro{\s}{1.0};
\pgfmathsetmacro{\hd}{1.20};
\pgfmathsetmacro{\vd}{0.96};
\pgfmathsetmacro{\is}{1.68};
\node[inner sep=\is](123)at(0,0){\scalebox{\s}{$123$}};
\node[inner sep=\is](213)at($(123)+(-\hd,\vd)$){\scalebox{\s}{$213$}};
\node[inner sep=\is](132)at($(123)+(\hd,\vd)$){\scalebox{\s}{$132$}};
\node[inner sep=\is](312)at($(213)+(0,\vd)$){\scalebox{\s}{$312$}};
\node[inner sep=\is](231)at($(132)+(0,\vd)$){\scalebox{\s}{$231$}};
\node[inner sep=\is](321)at($(312)+(\hd,\vd)$){\scalebox{\s}{$321$}};
\path[semithick](123)edge(213) edge(132) (213)edge(312) edge(231) (132)edge(312) edge(231) (312)edge(321) (231)edge(321);
\end{tikzpicture}
$$
\caption{The Hasse diagram of Bruhat order on $\mathfrak{S}_3$.}
\label{figure_S3}
\end{center}
\end{figure}
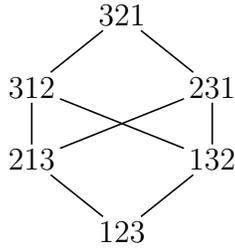

\begin{eg}\label{eg_permutation_matrix}
Let $w := 312 \in\mathfrak{S}_3$. Then $\ell(w) = 2$ and $\mathring{w} = \scalebox{0.8}{$\begin{bmatrix}0 & -1 & 0 \\ 0 & 0 & -1 \\ 1 & 0 & 0\end{bmatrix}$}$.
\end{eg}

\begin{defn}[{\cite[Section 8]{lusztig94}}]\label{defn_cell_decomposition}
Given $v,w\in\mathfrak{S}_n$ with $v \le w$, we define
$$
C_{v,w} := (\Bminus_n(\mathbb{C})\cdot\mathring{v}) \cap (\B_n(\mathbb{C})\cdot\mathring{w}) \cap \Fl_n^{\ge 0},
$$
which is the intersection inside $\Fl_n^{\ge 0}$ of the opposite Schubert cell indexed by $v$ and the Schubert cell indexed by $w$. Then we have the disjoint union
\begin{align}\label{cell_decomposition_equation}
\Fl_n^{\ge 0} = \bigsqcup_{v,w\in\mathfrak{S}_n,\, v \le w}C_{v,w},
\end{align}
and $C_{\id,w_0} = \Fl_n^{>0}$. We observe that each cell $C_{v,w}$ is preserved by the action of the totally positive part of the torus $\H_n^{>0}$ from \cref{torus_action}.
\end{defn}

Rietsch \cite{rietsch99} proved that \eqref{cell_decomposition_equation} is a cell decomposition:
\begin{thm}[Rietsch {\cite[Theorem 2.8]{rietsch99}}]\label{cell_decomposition_ball}
Let $v,w\in\mathfrak{S}_n$ with $v \le w$. Then $C_{v,w}$ is homeomorphic to an open ball of dimension $\ell(w) - \ell(v)$.
\end{thm}

\begin{eg}\label{eg_cell_decomposition}
We have $C_{132,312} = \left\{\scalebox{0.8}{$\begin{bmatrix}a & -1 & 0 \\ 0 & 0 & -1 \\ 1 & 0 & 0\end{bmatrix}$} : a > 0 \right\}\subseteq\Fl_3^{\ge 0}$, which is homeomorphic to a $1$-dimensional open ball.
\end{eg}

\section{The totally nonnegative part of \texorpdfstring{$\U_n$}{U(n)} and the twist map}\label{sec_tnn_U}

\noindent We  define the totally positive part and totally nonnegative part of $\U_n$, which we will be able to identify with $\Fl_n^{>0}$ and $\Fl_n^{\ge 0}$, respectively. We use this identification to introduce an involution $\twist$ on $\Fl_n^{\ge 0}$ which we call the {\itshape twist map}, motivated by similar maps of Berenstein, Fomin, and Zelevinsky (see \cref{twist_motivation}).

We will use the following consequence of the Gram--Schmidt orthonormalization process, or equivalently, the $QR$-decomposition. It is an instance of the {\itshape Iwasawa decomposition} for semisimple Lie groups. We remark that this decomposition has been studied in the context of totally positive matrices by Gasca and Pe\~{n}a \cite[Section 4]{gasca_pena93}.
\begin{prop}[{\cite[Section VI.4]{knapp02}}]\label{iwasawa_decomposition}
Let $n\in\mathbb{N}$.
\begin{enumerate}[label=(\roman*), leftmargin=*, itemsep=2pt]
\item\label{iwasawa_decomposition_map} The multiplication map
\begin{align}\label{eqn_iwasawa_decomposition}
\U_n \times \H_n^{>0} \times \N_n(\mathbb{C}) \to \GL_n(\mathbb{C}), \quad (g_1,g_2,g_3) \mapsto g_1g_2g_3
\end{align}
is a diffeomorphism.
\item\label{iwasawa_decomposition_isomorphism} We have $\U_n/\T_n = \GL_n(\mathbb{C})/\B_n(\mathbb{C}) = \Fl_n(\mathbb{C})$.
\end{enumerate}

\end{prop}

\begin{defn}\label{defn_iwasawa_projection}
Given $n\in\mathbb{N}$, we let $\Kterm : \GL_n(\mathbb{C}) \to \U_n$ denote the projection onto the first component of the inverse of \eqref{eqn_iwasawa_decomposition}. That is, for $(g_1, g_2, g_3) \in \U_n \times \H_n^{>0} \times \N_n(\mathbb{C})$, we have $\Kterm(g_1g_2g_3) = g_1$.

Similarly, we let $\kterm : \gl_n(\mathbb{C}) \to \uu_n$ denote the projection onto the first summand in the direct sum $\gl_n(\mathbb{C}) = \uu_n \oplus \h_n(\mathbb{R}) \oplus \n_n(\mathbb{C})$, which is linear over $\mathbb{R}$. That is, for $L\in\gl_n(\mathbb{C})$, the element $\kterm(L)$ is skew-Hermitian and $L - \kterm(L)$ is upper-triangular with real diagonal entries.
\end{defn}

\subsection{The totally nonnegative part of \texorpdfstring{$\U_n$}{U(n)}}\label{sec_tnn_U_defn}
We use the Iwasawa decomposition to define the totally positive and totally nonnegative parts of $\U_n$.
\begin{defn}\label{defn_tp_U}
Given $n\in\mathbb{N}$, we define the {\itshape totally positive part} of $\U_n$, denoted $\U_n^{>0}$, as the image of $\GL_n^{>0}$ under $\Kterm$. That is, an element $g_1\in\U_n$ is {\itshape totally positive} if and only if there exist $g_2\in\H_n^{>0}$ and $g_3\in\N_n(\mathbb{C})$ such that $g_1g_2g_3\in\GL_n^{>0}$. We define the {\itshape totally nonnegative part} of $\U_n$, denoted $\U_n^{\ge 0}$, as the closure of $\U_n^{>0}$ inside $\U_n$.
\end{defn}

We make several remarks about \cref{defn_tp_U}.

\begin{rmk}
We emphasize that although $\U_n$ is a subset of $\GL_n(\mathbb{C})$, the same does not hold for the respective totally positive or totally nonnegative parts unless $n=1$. For example, if $n\ge 2$ we have $\GL_n^{>0} \cap \U_n = \emptyset$, since every element in $\U_n$ has a matrix entry which is not a positive real number.
\end{rmk}

\begin{rmk}
It would make sense to denote $\U_n^{\ge 0}$ instead by any of $\O_n^{\ge 0}$, $\SU_n^{\ge 0}$, or $\SO_n^{\ge 0}$, since every element of $\U_n^{\ge 0}$ has real matrix entries and determinant $1$. We use the notation $\U_n^{\ge 0}$ since we wish to view this space as a subset of the real Lie group $\U_n$.
\end{rmk}

\begin{rmk}
We note that the projection of $\GL_n^{\ge 0}$ to $\U_n$ under $\Kterm$ is contained in $\U_n^{\ge 0}$, but is not equal to $\U_n^{\ge 0}$ unless $n=1$ (cf.\ \cref{proj_tnn_GL_to_Fl}).
\end{rmk}

\begin{eg}\label{eg_tnn_U_n=2}
Let us determine $\U_2^{>0}$ and $\U_2^{\ge 0}$, using the description of $\GL_2^{>0}$ from \cref{eg_tnn_GL_n=2}. For $a,b,c,d > 0$, we have the decomposition
$$
\begin{bmatrix}a & b \\[4pt] c & d + \frac{bc}{a} \\[2pt]\end{bmatrix} = \begin{bmatrix}
\frac{a}{\sqrt{a^2 + c^2}} & \frac{-c}{\sqrt{a^2 + c^2}} \\[4pt]
\frac{c}{\sqrt{a^2 + c^2}} & \frac{a}{\sqrt{a^2 + c^2}} \\[2pt]
\end{bmatrix}
\begin{bmatrix}
\sqrt{a^2 + c^2} & 0 \\[4pt]
0 & \frac{ad}{\sqrt{a^2 + c^2}} \\[2pt]
\end{bmatrix}
\begin{bmatrix}1 & \frac{a^2b + acd + bc^2}{a} \\[4pt]
0 & 1 \\[2pt]
\end{bmatrix}
$$
as in \eqref{eqn_iwasawa_decomposition}. Setting $\alpha := \arccos\big(\frac{a}{\sqrt{a^2 + c^2}}\big)$, we obtain
$$
\U_2^{>0} = \left\{\begin{bmatrix}\cos(\alpha) & -\hspace*{-1pt}\sin(\alpha) \\[1pt] \sin(\alpha) & \cos(\alpha)\end{bmatrix} : \alpha \in (0, \textstyle\frac{\pi}{2})\right\}.
$$
It follows that
\begin{gather*}
\U_2^{\ge 0} = \left\{\begin{bmatrix}\cos(\alpha) & -\hspace*{-1pt}\sin(\alpha) \\[1pt] \sin(\alpha) & \cos(\alpha)\end{bmatrix} : \alpha \in [0, \textstyle\frac{\pi}{2}]\right\}.\qedhere
\end{gather*}

\end{eg}

\begin{prop}\label{U_to_Fl}
Let $n\in\mathbb{N}$.
\begin{enumerate}[label=(\roman*), leftmargin=*, itemsep=2pt]
\item\label{U_to_Fl_isomorphism} Let $S$ be the open subset of $\O_n$ defined by the equations
\begin{align}\label{open_neighborhood}
\sum_{I\in\binom{[n]}{k}}\Delta_I > 0 \quad \text{ for } 1 \le k \le n.
\end{align}
Then $S$ contains $\U_n^{\ge 0}$. The projection map $S \to \Fl_n(\mathbb{R})$ is a diffeomorphism from $S$ onto its image, and further restricts to bijections
$$
\U_n^{>0}\xrightarrow{\cong}\Fl_n^{>0} \quad \text{ and } \quad \U_n^{\ge 0}\xrightarrow{\cong}\Fl_n^{\ge 0}.
$$
\item\label{U_to_Fl_tp} An element of $\U_n$ lies in $\U_n^{>0}$ if and only if all its left-justified minors are positive real numbers.
\item\label{U_to_Fl_tnn} An element of $\U_n$ lies in $\U_n^{\ge 0}$ if and only if all its left-justified minors are nonnegative real numbers.
\end{enumerate}

\end{prop}

\begin{proof}
First we prove the forward direction of part \ref{U_to_Fl_tp}, which implies the forward direction of part \ref{U_to_Fl_tnn}. Let $g_1\in\U_n^{>0}$, so that there exist $g_2\in\H_n^{>0}$ and $g_3\in\N_n(\mathbb{C})$ such that the element $g := g_1g_2g_3$ lies in $\GL_n^{>0}$. Write $g_2 = \Diag{\lambda_1, \dots, \lambda_n}$. Then for every $1 \le k \le n$ and $I\in\binom{[n]}{k}$, we have
$$
\Delta_I(g_1) = (\lambda_1\cdots\lambda_k)^{-1}\Delta_I(g) > 0.
$$

Now we prove part \ref{U_to_Fl_isomorphism}, whence the reverse directions of parts \ref{U_to_Fl_tp} and \ref{U_to_Fl_tnn} follow from \cref{tnn_Fl_converse}. Note that $S$ contains $\U_n^{\ge 0}$ by the forward direction of part \ref{U_to_Fl_tnn}. Let $\T_n(\mathbb{R}) := \T_n\cap\O_n$, which is a discrete group of size $2^n$. Since the $2^n$ orbits $S\cdot g$ for $g\in \T_n(\mathbb{R})$ are disjoint, the projection $S \to \O_n/\T_n(\mathbb{R}) = \Fl_n(\mathbb{R})$ is a diffeomorphism onto its image. This completes the proof.
\end{proof}

\begin{rmk}\label{consecutive_rows}
We observe that by \cref{U_to_Fl}\ref{U_to_Fl_isomorphism} and \cref{fekete}, an element of $\U_n$ lies in $\U_n^{>0}$ if and only if all its left-justified minors indexed by consecutive rows are positive real numbers.
\end{rmk}

\begin{rmk}
Recall from \cref{defn_symmetric_group} that to each permutation $w\in\mathfrak{S}_n$, we associate a signed permutation matrix $\mathring{w}\in\GL_n(\mathbb{R})$. The signs are determined by the fact that $\mathring{w}\in\U_n^{\ge 0}$.
\end{rmk}

\subsection{The reversal and duality maps}\label{sec_flip}
In this subsection, we introduce two involutions on $\U_n$: the reversal map $\rev$ and the duality map $\flip$. They correspond to reversing the order of either the rows or the columns of a matrix in $\U_n$, as well as changing the signs of certain entries so as to be compatible with total positivity. We recall from \cref{defn_symmetric_group} that we have the matrices
$$
\mathring{w}_0 = \scalebox{0.8}{$\begin{bmatrix}
\hspace*{2pt}0 & 0 & \cdots & (-1)^{n-1}\hspace*{2pt} \\
\hspace*{2pt}\vdots & \vdots & \reflectbox{$\ddots$} & \vdots\hspace*{2pt} \\
\hspace*{2pt}0 & -1 & \cdots & 0\hspace*{2pt} \\
\hspace*{2pt}1 & 0 & \cdots & 0\hspace*{2pt}
\end{bmatrix}$}\in\U_n^{\ge 0} \quad \text{ and } \quad \delta_n = \Diag{1, -1, 1, \dots, (-1)^{n-1}}\in\U_n,
$$
satisfying
$$
(\mathring{w}_0)^{-1} = \delta_n\mathring{w}_0\delta_n \quad \text{ and } \quad \delta_n^{-1} = \delta_n.
$$
\begin{defn}\label{defn_rev}
Given $n\in\mathbb{N}$, define the involution $\rev : \U_n \to \U_n$ by$$
\rev(g) := \mathring{w}_0\delta_ng\delta_n.
$$
For $K\subseteq [n-1]$, we obtain an involution $\rev : \PFl{K}{n}(\mathbb{C}) \to \PFl{K}{n}(\mathbb{C})$, given by
$$
\rev(V) = \mathring{w}_0\delta_n\cdot V \quad \text{ for all } V\in\PFl{K}{n}(\mathbb{C}).
$$
That is, $\rev(V)$ is obtained from $V$ by reversing the order of the ground set $[n]$. Note that by \cref{U_to_Fl} and \eqref{defn_tnn_Fl_surjections}, $\rev$ preserves $\U_n^{>0}$, $\U_n^{\ge 0}$, $\PFl{K}{n}^{>0}$, and $\PFl{K}{n}^{\ge 0}$.
\end{defn}

\begin{eg}
The reversal map $\rev$ sends
\begin{gather*}
\begin{bmatrix}
\frac{\sqrt{3}}{2} & -\frac{1}{2\sqrt{2}} & \frac{1}{2\sqrt{2}} \\[6pt]
\frac{\sqrt{3}}{4} & \frac{1}{4\sqrt{2}} & -\frac{5}{4\sqrt{2}} \\[6pt]
\frac{1}{4} & \frac{3\sqrt{3}}{4\sqrt{2}} & \frac{\sqrt{3}}{4\sqrt{2}}
\end{bmatrix}\hspace*{1pt} \text{ to } \hspace*{1pt}\begin{bmatrix}
\frac{1}{4} & -\frac{3\sqrt{3}}{4\sqrt{2}} & \frac{\sqrt{3}}{4\sqrt{2}} \\[6pt]
\frac{\sqrt{3}}{4} & -\frac{1}{4\sqrt{2}} & -\frac{5}{4\sqrt{2}} \\[6pt]
\frac{\sqrt{3}}{2} & \frac{1}{2\sqrt{2}} & \frac{1}{2\sqrt{2}}
\end{bmatrix}\hspace*{1pt} \text{ in } \U_3^{>0}.\qedhere
\end{gather*}

\end{eg}

\begin{defn}\label{defn_orthogonal_complement}
Given $n\in\mathbb{N}$, let $\langle\cdot,\cdot\rangle$ denote the inner product $\langle v,w\rangle := v_1\ccon{w_1} + \cdots + v_n\ccon{w_n}$ on $\mathbb{C}^n$. For a subspace $V$ of $\mathbb{C}^n$, we let $V^\perp$ denote the orthogonal complement of $V$.

Given $K = \{k_1 < \cdots < k_l\}\subseteq [n-1]$, define $K^\perp := \{n-k_l < \cdots < n-k_1\}\subseteq [n-1]$. For a flag $V = (V_{k_1}, \dots, V_{k_l}) \in \PFl{K}{n}(\mathbb{C})$, we define the orthogonally complementary flag $V^\perp \in\PFl{K\hspace*{-1pt}^\perp}{n}(\mathbb{C})$ by
\begin{gather*}
V^\perp := (W_{n-k_l}, \dots, W_{n-k_1}), \quad \text{ where } \quad W_{n-k_i} := (V_{k_i})^\perp \text{ for } 1 \le i \le l.
\end{gather*}

\end{defn}

\begin{defn}\label{defn_flip}
Given $n\in\mathbb{N}$, define the involution $\flip : \U_n \to \U_n$ by$$
\flip(g) := \delta_ng\delta_n\mathring{w}_0.
$$
In other words, $\flip(g)_{i,j} = (-1)^{n-i}g_{i,n+1-j}$ for $1 \le i,j \le n$.

Now let $K\subseteq [n-1]$. Note that $\flip(g)^{-1}\flip(gh) \in \P{K\hspace*{-1pt}^\perp}{n}(\mathbb{C})\cap\U_n$ for all $h\in\P{K}{n}(\mathbb{C})\cap\U_n$, so by \cref{iwasawa_decomposition}\ref{iwasawa_decomposition_isomorphism} we obtain an involution $\flip : \PFl{K}{n}(\mathbb{C}) \to \PFl{K\hspace*{-1pt}^\perp}{n}(\mathbb{C})$. In fact, we have
$$
\flip(V) = \delta_n\cdot V^\perp \quad \text{ for all } V\in\PFl{K}{n}(\mathbb{C}).
$$
We call $\flip(V)$ the {\itshape dual flag} of $V$. Note that the maps $\rev$ and $\flip$ on both $\U_n$ and $\PFl{K}{n}(\mathbb{C})$ commute.
\end{defn}

\begin{lem}\label{flip_action}
Let $n\in\mathbb{N}$, and let $K\subseteq [n-1]$.
\begin{enumerate}[label=(\roman*), leftmargin=*, itemsep=2pt]
\item\label{flip_action_formula_U} Let $g\in\U_n$. For $0 \le k \le n$, we have
$$
\Delta_{I,J}(\flip(g)) = (-1)^{\sumof{J} - \binom{k+1}{2}}\det(g)\ccon{\Delta_{[n]\setminus I, [n]\setminus J^\perp}(g)} \quad \text{ for all } I,J\in\textstyle\binom{[n]}{k}.
$$
\item\label{flip_action_formula_Fl} Let $V\in\PFl{K}{n}(\mathbb{C})$. We have
$$
\Delta_I(\flip(V)) = \ccon{\Delta_{[n]\setminus I}(V)} \quad \text{ for all } k\in K^\perp  \text{ and } I\in\textstyle\binom{[n]}{k}.
$$
\item\label{flip_action_tp} The involution $\flip$ preserves $\U_n^{>0}$ and $\U_n^{\ge 0}$. It takes $\PFl{K}{n}^{>0}$ onto $\PFl{K^\perp}{n}^{>0}$ and $\PFl{K}{n}^{\ge 0}$ onto $\PFl{K^\perp}{n}^{\ge 0}$.
\end{enumerate}

\end{lem}

\begin{proof}
\ref{flip_action_formula_U} For $0 \le k \le n$ and $I,J\in\binom{[n]}{k}$, we have
\begin{align*}
\Delta_{I,J}(\flip(g))
&= (-1)^{\sumof{I} + \sumof{J} + \sumof{J^\perp} - \binom{k+1}{2}}\Delta_{I,J^\perp}(g) \\
&= (-1)^{\sumof{I} + \sumof{J} + \sumof{J^\perp} - \binom{k+1}{2}}\ccon{\Delta_{J^\perp\hspace*{-0.5pt},I}(g^{-1})} \qquad \text{(since $g^{-1} = \adjoint{g}$)} \\
&= (-1)^{\sumof{J} - \binom{k+1}{2}}\ccon{\textstyle\frac{1}{\det(g)}\Delta_{[n]\setminus I, [n]\setminus J^\perp}(g)} \qquad \text{(by \eqref{jacobi})}.
\end{align*}

\ref{flip_action_formula_Fl} This follows from part \ref{flip_action_formula_U}, by representing any element of $\PFl{K}{n}(\mathbb{C})$ by an element of $\U_n$ and taking $J = [k]$.

\ref{flip_action_tp} By part \ref{flip_action_formula_U} and \cref{U_to_Fl}\ref{U_to_Fl_tp}, we see that $\flip$ preserves $\U_n^{>0}$. The result then follows by \cref{U_to_Fl}\ref{U_to_Fl_isomorphism}, \eqref{defn_tnn_Fl_surjections}, and taking closures.
\end{proof}

\begin{eg}\label{eg_flip}
We illustrate \cref{flip_action}\ref{flip_action_tp} in the case of $\U_2^{>0}$. By \cref{eg_tnn_U_n=2}, we can write any element $g\in\U_2^{>0}$ as $g = \scalebox{0.8}{$\begin{bmatrix}\cos(\alpha) & -\hspace*{-1pt}\sin(\alpha) \\[1pt] \sin(\alpha) & \cos(\alpha)\end{bmatrix}$}$, where $\alpha\in(0,\frac{\pi}{2})$. Then
\begin{gather*}
\flip(g) = \delta_2g\delta_2\mathring{w}_0 = \begin{bmatrix}\sin(\alpha) & -\hspace*{-1pt}\cos(\alpha) \\[1pt] \cos(\alpha) & \sin(\alpha)\end{bmatrix}\in\U_2^{>0}.\qedhere
\end{gather*}

\end{eg}

\subsection{The twist map}\label{sec_twist}
We now introduce the twist map $\twist$.
\begin{defn}\label{defn_positive_inverse}
Given $n\in\mathbb{N}$, define the involution $\iota : \GL_n(\mathbb{C}) \to \GL_n(\mathbb{C})$ by
$$
\iota(g) := \delta_n g^{-1}\delta_n.
$$
In other words, $\iota(g)_{i,j} = (-1)^{i+j}(g^{-1})_{i,j}$ for $1 \le i,j \le n$.
\end{defn}

\begin{eg}\label{eg_positive_inverse}
Let $g := \scalebox{0.8}{$\begin{bmatrix}a & b \\ c & d\end{bmatrix}$} \in \GL_2(\mathbb{C})$. Then $\iota(g) = \frac{1}{ad-bc}\scalebox{0.8}{$\begin{bmatrix}d & b \\ c & a\end{bmatrix}$}$.
\end{eg}

Fomin and Zelevinsky \cite[Section 2.1]{fomin_zelevinsky99} call $\iota$ the ``totally nonnegative version'' of the matrix inverse, since it preserves $\GL_n^{>0}$ and $\GL_n^{\ge 0}$; this follows from \eqref{jacobi}. We will focus on the more subtle analogue for $\U_n$:
\begin{prop}\label{positive_inverse}
Let $n\in\mathbb{N}$.
\begin{enumerate}[label=(\roman*), leftmargin=*, itemsep=2pt]
\item\label{positive_inverse_twist} Let $g\in\U_n$. We have
$$
\Delta_{[i,j]}(\iota(g)) = \sum_{I\in\binom{[j-i+2,n]}{i-1}}\!\Delta_I(g)\ccon{\Delta_{[j-i+1]\hspace*{0.5pt}\cup\hspace*{0.5pt} I}(g)} \quad \text{ for all } 1 \le i \le j \le n.
$$
\item\label{positive_inverse_action} The involution $\iota$ preserves $\U_n^{>0}$ and $\U_n^{\ge 0}$.
\end{enumerate}

\end{prop}

\begin{proof}
\ref{positive_inverse_twist} For $1 \le i \le j \le n$, we have
\begin{align*}
\Delta_{[i,j]}(\iota(g)) &= \frac{1}{\det(g)}\Delta_{[j-i+2,n],[i-1]\hspace*{0.5pt}\cup\hspace*{0.5pt} [j+1,n]}(g) \qquad \text{(by \eqref{jacobi})} \\[2pt]
&= \frac{1}{\det(g)}\sum_{I\in\binom{[j-i+2,n]}{i-1}}\!(-1)^{\sumof{[j-i+2,j]} + \sumof{I}}\Delta_{I,[i-1]}(g)\Delta_{[j-i+2,n]\setminus I,[j+1,n]}(g) \qquad \text{(by \eqref{laplace})} \\[2pt]
&= \sum_{I\in\binom{[j-i+2,n]}{i-1}}\!\Delta_I(g)\ccon{\Delta_{[j-i+1]\hspace*{0.5pt}\cup\hspace*{0.5pt} I}(g)} \qquad \text{(by \eqref{jacobi}, since $g^{-1} = \adjoint{g}$)}.
\end{align*}

\ref{positive_inverse_action} Recall the descriptions of $\U_n^{>0}$ from \cref{U_to_Fl}\ref{U_to_Fl_tp} and \cref{consecutive_rows}. By part \ref{positive_inverse_twist}, if $g\in\U_n^{>0}$, then $\Delta_{[i,j]}(\iota(g)) > 0$ for all $1 \le i \le j \le n$, whence $\iota(g)\in\U_n^{>0}$. Therefore $\iota$ preserves $\U_n^{>0}$, and also preserves the closure $\U_n^{\ge 0}$.
\end{proof}

\begin{rmk}\label{first_row_positive}
\cref{positive_inverse}\ref{positive_inverse_action} implies that if $g\in\U_n^{>0}$, then the entries of the first row of $g$ are nonzero and alternate in sign, i.e.,
\begin{align}\label{first_row_positive_equation}
(-1)^{j-1}g_{1,j} > 0 \quad \text{ for } 1 \le j \le n.
\end{align}
Conversely, if $V\in\Fl_n^{>0}$ and $g\in\U_n$ represents $V$, then $g\in\U_n^{>0}$ if and only if $g$ satisfies \eqref{first_row_positive_equation} (if and only if $g$ satisfies \eqref{open_neighborhood}). However, note that \eqref{first_row_positive_equation} (even after replacing `$>$' with `$\ge$') does not similarly characterize the closure $\U_n^{\ge 0}$, in contrast to \eqref{open_neighborhood}. For example, \eqref{first_row_positive_equation} does not determine the signs in the signed permutation matrix $\mathring{w}\in\U_n^{\ge 0}$. We will return to this distinction in \cref{BFR_map}.
\end{rmk}

\begin{defn}\label{defn_twist}
Let $n\in\mathbb{N}$. By \cref{positive_inverse}\ref{positive_inverse_action}, $\iota$ preserves $\U_n^{\ge 0}$. Hence the identification $\U_n^{\ge 0} \xrightarrow{\cong} \Fl_n^{\ge 0}$ of \cref{U_to_Fl}\ref{U_to_Fl_isomorphism} induces an involution
$$
\twist : \Fl_n^{\ge 0} \to \Fl_n^{\ge 0}, \quad g\in\U_n^{\ge 0} \mapsto \iota(g),
$$
which is a diffeomorphism on some open neighborhood of $\Fl_n^{\ge 0}$ inside $\Fl_n(\mathbb{R})$. (Explicitly, we can take this open neighborhood to be the image in $\Fl_n(\mathbb{R})$ of $S\cap\iota(S)\subseteq\O_n$, where $S$ is defined by \eqref{open_neighborhood}.) We call this involution the {\itshape (Iwasawa) twist map}.
\end{defn}

\begin{rmk}\label{twist_extension}
It is not clear how to extend $\twist$ to all of $\Fl_n(\mathbb{R})$ or $\Fl_n(\mathbb{C})$, since there is no canonical way to represent a complete flag by an element of $\U_n$. Similarly, it is not clear how to define a twist map on the totally nonnegative part of an arbitrary partial flag variety $\PFl{K}{n}^{\ge 0}$, since there is no canonical way to represent a totally nonnegative partial flag by an element of $\U_n^{\ge 0}$.
\end{rmk}

\begin{rmk}\label{twist_motivation}
The name {\itshape twist map} is motivated by the twist maps defined by Berenstein, Fomin, and Zelevinsky on $\N_n(\mathbb{C})$ \cite[Lemma 1.3]{berenstein_fomin_zelevinsky96} and by Fomin and Zelevinsky on $\GL_n(\mathbb{C})$ \cite[(4.10)]{fomin_zelevinsky99}. The key difference between these maps and our map $\twist$ is that the former are based on the Bruhat decomposition of $\GL_n(\mathbb{C})$, whereas $\twist$ is based on the Iwasawa decomposition.

Indeed, the map $\twist$ on $\Fl_n^{\ge 0}$ takes a complete flag represented as a matrix $g\in\U_n^{\ge 0}$, and acts as the map
$$
g \mapsto \delta_n\transpose{g}\delta_n.
$$
The map of \cite[Lemma 1.3]{berenstein_fomin_zelevinsky96} induces a rational map on $\Fl_n(\mathbb{C})$ defined in a similar way (up to an application of the map $\rev$ from \cref{defn_rev}), but where we instead represent a complete flag by a matrix of the form
$$
g = \scalebox{0.8}{$\begin{bmatrix}
\hspace*{2pt}\ast & \ast & \ast & \cdots & (-1)^{n-1}\hspace*{2pt} \\
\hspace*{2pt}\vdots & \vdots & \vdots & \reflectbox{$\ddots$} & \vdots\hspace*{2pt} \\
\hspace*{2pt}\ast & \ast & 1 & \cdots & 0\hspace*{2pt} \\
\hspace*{2pt}\ast & -1 & 0 & \cdots & 0\hspace*{2pt} \\
\hspace*{2pt}1 & 0 & 0 & \cdots & 0\hspace*{2pt}
\end{bmatrix}$}.
$$
For example, when $n=3$ we obtain the map on $\Fl_3^{>0}$
$$
\begin{bmatrix}
bc & -b & 1 \\
a+c & -1 & 0 \\
1 & 0 & 0
\end{bmatrix} \hspace*{2pt}\mapsto\hspace*{2pt} \begin{bmatrix}
bc & -(a+c) & 1 \\
b & -1 & 0 \\
1 & 0 & 0
\end{bmatrix} \quad (a,b,c > 0).
$$

Note that the latter map above is not defined on all of $\Fl_n^{\ge 0}$. Following \cite[(4.10)]{fomin_zelevinsky99}, one could attempt to extend the definition to all of $\Fl_n^{\ge 0}$, but we would expect the definition to be different for each cell $C_{v,w}$ in \eqref{cell_decomposition_equation}, and that the resulting map would not necessarily be continuous when passing between cells. One encounters a similar issue when attempting to extend $\twist$ to all of $\Fl_n(\mathbb{R})$ or $\Fl_n(\mathbb{C})$ (cf.\ \cref{twist_extension}), but the issue occurs away from the totally nonnegative part. The fact that $\twist$ is a diffeomorphism defined on a neighborhood of $\Fl_n^{\ge 0}$ (and not merely on $\Fl_n^{>0}$) will be essential for us, for example in \cref{sec_toda}.
\end{rmk}

\begin{rmk}\label{BFR_remark}
The twist map $\twist$ generalizes (in type $A$) a map of Bloch, Flaschka, and Ratiu \cite[Section 3]{bloch_flaschka_ratiu90} defined on the subset of tridiagonal matrices of $\Orbit_{\bflambda}^{\ge 0}$, known as an isospectral manifold of Jacobi matrices. We discuss this in more detail in \cref{BFR_map}, after introducing Jacobi matrices in \cref{sec_tridiagonal_orbit}. It is also closely related to a map on $\Fl_n(\mathbb{R})$ introduced by Mart\'{i}nez Torres and Tomei \cite[Proposition 1]{martinez_torres_tomei}. The main difference between the two maps is in the domain of definition. Indeed, our twist map is a diffeomorphism on $\Fl_n^{\ge 0}$, and is designed to be compatible with total positivity. On the other hand, the map of \cite{martinez_torres_tomei} is defined piecewise on each Bruhat cell of $\Fl_n(\mathbb{R})$, and is designed to be compatible with the asymptotic behavior of the symmetric Toda flow (see \cref{sorting_remark}); however, it is not compatible with total positivity, for the same reasons as discussed in \cref{twist_motivation}. We also mention that the twist map is also related to the dressing transformations of Poisson geometry \cite{semenov-tian-shansky85,lu_weinstein90}.
\end{rmk}

\begin{eg}\label{eg_twist}
We explicitly describe the twist map $\twist$ on $\Fl_n^{\ge 0}$ for $n = 1,2,3$. When $n=1$, $\Fl_1^{\ge 0}$ is a point, so $\twist$ is necessarily the identity. When $n=2$, we can verify from \cref{eg_tnn_U_n=2} that $\twist$ is again the identity.

We now consider the case $n=3$. Let $g\in\Fl_3^{\ge 0}$, and let $\Delta_I$ and $\Delta^\twist_I$ denote the Pl\"{u}cker coordinates of $g$ and $\twist(g)$, respectively, where the former are normalized so that
\begin{align}\label{twist_normalization}
\sum_{I\in\binom{[3]}{k}}\Delta_I^2 = 1 \quad \text{ for } k = 1,2,3.
\end{align}
Note that the Pl\"{u}cker coordinates satisfy the {\itshape Pl\"{u}cker relation} (cf.\ \cite[Section 9.1]{fulton97})
$$
\Delta_2\Delta_{13} = \Delta_1\Delta_{23} + \Delta_3\Delta_{12}.
$$
By \cref{positive_inverse}\ref{positive_inverse_twist}, we find
$$
\Delta^\twist_1 = \Delta_1,\quad \Delta^\twist_2 = \Delta_2\Delta_{12} + \Delta_3\Delta_{13},\quad \Delta^\twist_3 = \Delta_{23},\quad \Delta^\twist_{12} = \Delta_{12},\quad \Delta^\twist_{23} = \Delta_3.
$$
The remaining Pl\"{u}cker coordinate $\Delta^\twist_{13}$ can be obtained from the Pl\"{u}cker relation:
$$
\Delta^\twist_{13} = \frac{\Delta^\twist_1\Delta^\twist_{23} + \Delta^\twist_3\Delta^\twist_{12}}{\Delta^\twist_2} = \frac{\Delta_1\Delta_3 + \Delta_{23}\Delta_{12}}{\Delta_2\Delta_{12} + \Delta_3\Delta_{13}}.
$$
One can verify that the $\Delta^\twist$'s satisfy the same normalization condition \eqref{twist_normalization} and that $\Delta \mapsto \Delta^\twist$ defines an involution, though this is not obvious.
\end{eg}

The twist map acts on the cell decomposition \eqref{cell_decomposition_equation} of $\Fl_n^{\ge 0}$:
\begin{thm}\label{twist_action}
The twist map $\twist$ preserves $\Fl_n^{>0}$ and $\Fl_n^{\ge 0}$. For all $v,w\in\mathfrak{S}_n$ with $v\le w$, it restricts to a diffeomorphism $\twist : C_{v,w} \xrightarrow{\cong} C_{v^{-1},w^{-1}}$.
\end{thm}

\begin{proof}
The map $\twist$ preserves $\Fl_n^{>0}$ and $\Fl_n^{\ge 0}$ by \cref{positive_inverse}\ref{positive_inverse_action}. Since $\twist$ is an involution, it remains to prove the containment $\twist(C_{v,w})\subseteq C_{v^{-1},w^{-1}}$. We show that given
$$
g \in (\Bminus_n(\mathbb{C})\hspace*{1pt}\mathring{v}\B_n(\mathbb{C})) \cap (\B_n(\mathbb{C})\hspace*{1pt}\mathring{w}\B_n(\mathbb{C})) \cap \U_n,
$$
we have
$$
\iota(g) \in (\Bminus_n(\mathbb{C})\hspace*{1pt}\mathring{(v^{-1})}\B_n(\mathbb{C})) \cap (\B_n(\mathbb{C})\hspace*{1pt}\mathring{(w^{-1})}\B_n(\mathbb{C})).
$$
Indeed, we have
$$
\iota(g) = \delta_n\transpose{g}\delta_n \in (\delta_n\transpose{\B_n(\mathbb{C})}\delta_n)(\delta_n\transpose{\mathring{v}}\delta_n)(\delta_n\transpose{\Bminus_n(\mathbb{C})}\delta_n) = \Bminus_n(\mathbb{C})\hspace*{1pt}\mathring{(v^{-1})}\B_n(\mathbb{C})
$$
and
\begin{gather*}
\iota(g) = \delta_ng^{-1}\delta_n \in (\delta_n\B_n(\mathbb{C})^{-1}\delta_n)(\delta_n(\mathring{w})^{-1}\delta_n)(\delta_n\B_n(\mathbb{C})^{-1}\delta_n) = \B_n(\mathbb{C})\hspace*{1pt}\mathring{(w^{-1})}\B_n(\mathbb{C}).\qedhere
\end{gather*}

\end{proof}

\begin{eg}\label{eg_twist_action}
We illustrate \cref{twist_action} in the case $n := 3$, $v := 132$, and $w := 312$:
\begin{gather*}
C_{132,312} \ni \begin{bmatrix}
\cos(\alpha) & -\hspace*{-1pt}\sin(\alpha) & 0 \\[1pt]
0 & 0 & -1 \\[1pt]
\sin(\alpha) & \cos(\alpha) & 0
\end{bmatrix} \hspace*{2pt}\overset{\twist}{\mapsto}\hspace*{2pt}
\begin{bmatrix}
\cos(\alpha) & 0 & \sin(\alpha) \\[1pt]
\sin(\alpha) & 0 & -\hspace*{-1pt}\cos(\alpha) \\[1pt]
0 & 1 & 0
\end{bmatrix} \in C_{132,231} \quad \big(\alpha \in (0, \textstyle\frac{\pi}{2})\big).\qedhere
\end{gather*}
\end{eg}

We conclude this section by relating the three maps $\rev$, $\flip$, and $\twist$.
\begin{lem}\label{twist_flip_rev}
Let $n\in\mathbb{N}$.
\begin{enumerate}[label=(\roman*), leftmargin=*, itemsep=2pt]
\item\label{twist_flip_rev_U} We have $\iota\circ\rev\circ\iota = \flip$ on $\U_n$.
\item\label{twist_flip_rev_Fl} We have $\twist\circ\rev\circ\twist = \flip$ on $\Fl_n^{\ge 0}$.
\end{enumerate}
\end{lem}

\begin{proof}
We can verify part \ref{twist_flip_rev_U} from the definitions, whence part \ref{twist_flip_rev_Fl} follows.
\end{proof}

Like the twist map $\twist$, the maps $\rev$ and $\flip$ act on the cell decomposition \eqref{cell_decomposition_equation} of $\Fl_n^{\ge 0}$:
\begin{lem}\label{flip_rev_action}
Let $n\in\mathbb{N}$, and let $v,w\in\mathfrak{S}_n$ with $v\le w$.
\begin{enumerate}[label=(\roman*), leftmargin=*, itemsep=2pt]
\item\label{flip_rev_action_rev} The map $\rev : \Fl_n^{\ge 0} \to \Fl_n^{\ge 0}$ restricts to a diffeomorphism $C_{v,w} \xrightarrow{\cong} C_{w_0w,w_0v}$.
\item\label{flip_rev_action_flip} The map $\flip : \Fl_n^{\ge 0} \to \Fl_n^{\ge 0}$ restricts to a diffeomorphism $C_{v,w} \xrightarrow{\cong} C_{ww_0,vw_0}$.
\end{enumerate}
\end{lem}

\begin{proof}
We prove part \ref{flip_rev_action_rev}, whence part \ref{flip_rev_action_flip} follows from \cref{twist_flip_rev}\ref{twist_flip_rev_Fl} and \cref{twist_action}. Since $\rev$ is an involution, it suffices to prove the containment $\rev(C_{v,w})\subseteq C_{w_0w,w_0v}$. We show that given
$$
g \in (\Bminus_n(\mathbb{C})\hspace*{1pt}\mathring{v}\B_n(\mathbb{C})) \cap (\B_n(\mathbb{C})\hspace*{1pt}\mathring{w}\B_n(\mathbb{C})) \cap \U_n,
$$
we have
$$
\rev(g) \in (\Bminus_n(\mathbb{C})\hspace*{1pt}\mathring{(w_0w)}\B_n(\mathbb{C})) \cap (\B_n(\mathbb{C})\hspace*{1pt}\mathring{(w_0v)}\B_n(\mathbb{C})).
$$
Indeed, we have
\begin{align*}
\rev(g) = \mathring{w}_0\delta_ng\delta_n &\in (\mathring{w}_0\Bminus_n(\mathbb{C})\hspace*{1pt}\mathring{v}\B_n(\mathbb{C})) \cap (\mathring{w}_0\B_n(\mathbb{C})\hspace*{1pt}\mathring{w}\B_n(\mathbb{C})) \\
&= (\B_n(\mathbb{C})\hspace*{1pt}\mathring{(w_0v)}\B_n(\mathbb{C})) \cap (\Bminus_n(\mathbb{C})\hspace*{1pt}\mathring{(w_0w)}\B_n(\mathbb{C})).\qedhere
\end{align*}

\end{proof}

\section{The totally nonnegative part of an adjoint orbit}\label{sec_tnn_orbit}

\noindent In this section we introduce the totally positive and totally nonnegative parts of any adjoint orbit $\Orbit_{\bflambda}$ of $\uu_n$. We can identify $\Orbit_{\bflambda}$ with some partial flag variety $\PFl{K}{n}(\mathbb{C})$, and its totally positive and totally nonnegative parts are defined so as to agree with those for $\PFl{K}{n}(\mathbb{C})$. We then study this notion in more detail in three cases of particular interest: when the corresponding flag variety is the complete flag variety, when the corresponding flag variety is a Grassmannian, and for tridiagonal matrices.

\subsection{Adjoint orbits of \texorpdfstring{$\uu_n$}{u(n)}}\label{sec_tnn_u}
We introduce adjoint orbits of the Lie algebra $\uu_n$ of $\U_n$.
\begin{defn}\label{defn_orbit}
Let $\bflambda = (\lambda_1, \dots, \lambda_n)\in\mathbb{R}^n$ be weakly decreasing, i.e., $\lambda_1 \ge \cdots \ge \lambda_n$. We define the {\itshape adjoint orbit}
$$
\Orbit_{\bflambda} := \{g(\ii\Diag{\bflambda})g^{-1} : g\in\U_n\} \subseteq \uu_n.
$$
We define the totally positive and totally nonnegative parts of $\Orbit_{\bflambda}$ by 
$$
\Orbit_{\bflambda}^{>0} := \{g(\ii\Diag{\bflambda})g^{-1} : g\in\U_n^{>0}\}, \quad \Orbit_{\bflambda}^{\ge 0} := \overline{\Orbit_{\bflambda}^{>0}} = \{g(\ii\Diag{\bflambda})g^{-1} : g\in\U_n^{\ge 0}\},
$$
where the latter description of $\Orbit_{\bflambda}^{\ge 0}$ will follow from \cref{Fl_to_orbit}.
\end{defn}

\begin{rmk}
We note that every adjoint orbit of $\uu_n$ is of the form $\Orbit_{\bflambda}$ for some $\bflambda$. The assumption that $\bflambda$ is weakly decreasing is not an arbitrary convention; it is essential for defining $\Orbit_{\bflambda}^{>0}$ and $\Orbit_{\bflambda}^{\ge 0}$.
\end{rmk}

\begin{rmk}\label{left_right_action}
We have defined $\Orbit_{\bflambda}^{>0}$ and $\Orbit_{\bflambda}^{\ge 0}$ using the left action of $\U_n$ on $\uu_n$. If instead we use the right action, we obtain the same spaces conjugated by $\delta_n$:
$$
\{g^{-1}(\ii\Diag{\bflambda})g : g\in\U_n^{>0}\} = \delta_n\Orbit_{\bflambda}^{>0}\delta_n, \quad \{g^{-1}(\ii\Diag{\bflambda})g : g\in\U_n^{\ge 0}\} = \delta_n\Orbit_{\bflambda}^{\ge 0}\delta_n.
$$
This follows from \cref{positive_inverse}\ref{positive_inverse_action}.
\end{rmk}

\begin{eg}\label{eg_orbit}
Let $\bflambda := (\lambda_1, \lambda_2)\in\mathbb{R}^2$ with $\lambda_1 \ge \lambda_2$. Then by \cref{eg_tnn_U_n=2}, we have
\begin{gather*}
\Orbit_{\bflambda}^{>0} = \left\{\ii\begin{bmatrix}
\lambda_1\cos^2(\alpha) + \lambda_2\sin^2(\alpha) & (\lambda_1 - \lambda_2)\sin(\alpha)\cos(\alpha) \\[2pt]
(\lambda_1 - \lambda_2)\sin(\alpha)\cos(\alpha) & \lambda_1\sin^2(\alpha) + \lambda_2\cos^2(\alpha)
\end{bmatrix} : \alpha \in (0, \textstyle\frac{\pi}{2})\right\}.\qedhere
\end{gather*}

\end{eg}

\begin{lem}\label{Fl_to_orbit}
Let $\bflambda\in\mathbb{R}^n$ be weakly decreasing, and set $K := \{i \in [n-1] : \lambda_i > \lambda_{i+1}\}$. Then the map
\begin{align}\label{Fl_to_orbit_map}
\begin{gathered}
\PFl{K}{n}(\mathbb{C}) \to \Orbit_{\bflambda}, \\
g\in\U_n \mapsto g(\ii\Diag{\bflambda})g^{-1}
\end{gathered}
\end{align}
is a diffeomorphism which takes $\PFl{K}{n}^{>0}$ onto $\Orbit_{\bflambda}^{>0}$ and $\PFl{K}{n}^{\ge 0}$ onto $\Orbit_{\bflambda}^{\ge 0}$.
\end{lem}

\begin{proof}
By \cref{iwasawa_decomposition}, we have $\PFl{K}{n}(\mathbb{C}) = \U_n/(\P{K}{n}(\mathbb{C})\cap\U_n)$, and $\P{K}{n}(\mathbb{C})\cap\U_n$ is the centralizer of $\ii\Diag{\bflambda}$. Therefore \eqref{Fl_to_orbit_map} is well-defined and a diffeomorphism, and the remaining assertions follow from \cref{U_to_Fl}\ref{U_to_Fl_isomorphism}.
\end{proof}

\subsection{The complete flag variety and eventually totally positive matrices}\label{sec_Fl_orbit} We consider the case when $\Orbit_{\bflambda}\cong\Fl_n(\mathbb{C})$, i.e., when $\bflambda$ is strictly decreasing (or {\itshape generic}). After translating $\bflambda$ by a multiple of $(1, \dots, 1)$, we may additionally assume that all its components are positive. Then results of Gantmakher and Krein \cite{gantmakher_krein37} and Kushel \cite{kushel15} characterize $-\ii \Orbit_{\bflambda}^{>0}$ as a space of {\itshape eventually totally positive matrices}. For completeness, we provide a proof.
\begin{prop}[Gantmakher and Krein {\cite[Theorem 16]{gantmakher_krein37}}; Kushel {\cite[Theorem 7]{kushel15}}]\label{complete_to_orbit}
Let $\bflambda = (\lambda_1, \dots, \lambda_n)$, where $\lambda_1 > \cdots > \lambda_n > 0$, and let $\ii L\in\Orbit_{\bflambda}$. Then the following are equivalent:
\begin{enumerate}[label=(\roman*), leftmargin=*, itemsep=2pt]
\item\label{complete_to_orbit_positivity} $\ii L\in\Orbit_{\bflambda}^{>0}$;
\item\label{complete_to_orbit_some} $L^m\in\GL_n^{>0}$ for some $m\in\mathbb{Z}_{>0}$; and
\item\label{complete_to_orbit_all} $L^m\in\GL_n^{>0}$ for all sufficiently large $m\in\mathbb{Z}_{>0}$.
\end{enumerate}

\end{prop}

\begin{proof}
Note that \ref{complete_to_orbit_all} $\Rightarrow$ \ref{complete_to_orbit_some} holds, and \ref{complete_to_orbit_some} $\Rightarrow$ \ref{complete_to_orbit_positivity} follows from \cref{gk}\ref{gk_eigenvectors}. We now prove \ref{complete_to_orbit_positivity} $\Rightarrow$ \ref{complete_to_orbit_all}. Suppose that $\ii L\in\Orbit_{\bflambda}^{>0}$, so that $L = g\Diag{\bflambda}g^{-1}$ for some $g\in\U_n^{>0}$. Let $I,J\in\binom{[n]}{k}$, where $1 \le k \le n$. By \eqref{cauchy-binet}, we have
$$
\Delta_{I,J}(L^m) = \sum_{K\in\binom{[n]}{k}}\!\big(\!\textstyle\prod_{i\in K}\lambda_i\big)^m\Delta_{I,K}(g)\Delta_{J,K}(g) = (\lambda_1\cdots\lambda_k)^m(\Delta_I(g)\Delta_J(g) + o(1))
$$
as $m\to\infty$. Since $\Delta_I(g), \Delta_J(g) > 0$ by \cref{U_to_Fl}\ref{U_to_Fl_tp}, we see that $\Delta_{I,J}(L^m) > 0$ for all $m$ sufficiently large.
\end{proof}

\begin{rmk}\label{oscillatory_remark}
A matrix $g\in\GL_n(\mathbb{R})$ is called {\itshape oscillatory} \cite[Section 2]{gantmakher_krein37} if $g\in\GL_n^{\ge 0}$ and $g^m\in\GL_n^{>0}$ for some $m > 0$ (equivalently, for all $m \ge n-1$). Every eventually totally positive matrix is oscillatory, but the converse does not hold. For example, the matrix
$$
g = \begin{bmatrix}
11 & 3\sqrt{2} & -1 \\[2pt]
3\sqrt{2} & 10 & 3\sqrt{2} \\[2pt]
-1 & 3\sqrt{2} & 11
\end{bmatrix}
$$
is eventually totally positive, but it is not totally nonnegative.
\end{rmk}

\begin{rmk}\label{complete_to_orbit_tnn}
We observe that in \cref{complete_to_orbit}\ref{complete_to_orbit_some}, the required power $m\in\mathbb{Z}_{>0}$ may be arbitrarily large, even when $\bflambda$ is fixed. (This is in contrast to the situation for oscillatory matrices, where the required power $m$ is at most $n-1$.) To see this, take $\alpha\in (0,\frac{\pi}{2})$, and define
$$
\ii L := g(\ii \Diag{\lambda_1, \lambda_2, \lambda_3})g^{-1} \in\Orbit_{\bflambda}^{>0}, \quad \text{ where } g := \begin{bmatrix}
\frac{1}{\sqrt{2}}\sin(\alpha) & -\frac{1}{\sqrt{2}} & \frac{1}{\sqrt{2}}\cos(\alpha) \\[6pt]
\cos(\alpha) & 0 & -\hspace*{-1pt}\sin(\alpha) \\[6pt]
\frac{1}{\sqrt{2}}\sin(\alpha) & \frac{1}{\sqrt{2}} & \frac{1}{\sqrt{2}}\cos(\alpha) \\[2pt]
\end{bmatrix}\in\U_3^{>0}.
$$
Then
$$
(L^m)_{1,3} = \textstyle\frac{1}{2}\big(\!\sin^2(\alpha)(\lambda_1^m - \lambda_2^m) - \cos^2(\alpha)(\lambda_2^m - \lambda_3^m)\big).
$$
If $L^m\in\GL_3^{>0}$, then $(L^m)_{1,3} > 0$, which implies
$$
\frac{\lambda_1^m - \lambda_2^m}{\lambda_2^m - \lambda_3^m} > \frac{1}{\tan^2(\alpha)}.
$$
As $\alpha\to 0$, this requires $m\to\infty$.

We also observe that the analogue of \cref{complete_to_orbit} for $\Orbit_{\bflambda}^{\ge 0}$ fails to hold. To see this, take $\alpha := 0$ above, so that $\ii L\in\Orbit_{\bflambda}^{\ge 0}$. Then
$$
(L^m)_{1,3} = -\textstyle\frac{1}{2}(\lambda_2^m - \lambda_3^m) < 0,
$$
so $L^m\notin\GL_3^{\ge 0}$ for all $m\in\mathbb{Z}_{>0}$.
\end{rmk}

\subsection{The Grassmannian and projection matrices}\label{sec_Gr_orbit}
We consider the case when $\Orbit_{\bflambda}\cong\Gr_{k,n}(\mathbb{C})$, i.e., when $\lambda_1 = \cdots = \lambda_k > \lambda_{k+1} = \cdots = \lambda_n$. After translating $\bflambda$ by a scalar multiple of $(1, \dots, 1)$ and rescaling it by a positive constant, we may assume that $\bflambda = (1, \dots, 1, 0, \dots, 0)$.
\begin{defn}\label{defn_projection}
Given $0 \le k \le n$, we let $\bfomega{k} := (1, \dots, 1, 0, \dots, 0)$ denote the vector of $k$ ones followed by $n-k$ zeros. Then $-\ii \Orbit_{\bfomega{k}}$ is a space of projection matrices:
$$
\Orbit_{\bfomega{k}} = \{\ii P : P \in \gl_n(\mathbb{C}) \text{ with } P^2 = P = \adjoint{P} \text{ and } \tr(P) = k\}.
$$
(We may replace the condition $\tr(P) = k$ with $\rank(P) = k$.)

Given $V\in\Gr_{k,n}(\mathbb{C})$, let $\Proj{V}\in\gl_n(\mathbb{C})$ denote the orthogonal projection from $\mathbb{C}^n$ onto the subspace $V$. If we regard $V$ as an $n\times k$ matrix modulo column operations, then $\Proj{V} = V(\adjoint{V}V)^{-1}\adjoint{V}$.
\end{defn}

\begin{lem}\label{Gr_to_orbit}
Let $0 \le k \le n$. Then the map
$$
\Gr_{k,n}(\mathbb{C}) \to \Orbit_{\bfomega{k}}, \quad V \mapsto \ii\Proj{V}
$$
is a diffeomorphism which takes $\Gr_{k,n}^{>0}$ onto $\Orbit_{\bfomega{k}}^{>0}$ and $\Gr_{k,n}^{\ge 0}$ onto $\Orbit_{\bfomega{k}}^{\ge 0}$.
\end{lem}

\begin{proof}
This follows from \cref{Fl_to_orbit}, since the map $V \mapsto \ii\Proj{V}$ is precisely \eqref{Fl_to_orbit_map}.
\end{proof}

We explain how to recover the Pl\"{u}cker coordinates of $V$ from $\Proj{V}$. This will lead to explicit descriptions of $\Orbit_{\bfomega{k}}^{>0}$ and $\Orbit_{\bfomega{k}}^{\ge 0}$. We recall that $\inv{I}{J}$ denotes the number of pairs $(i,j)\in I\times J$ such that $i > j$.
\begin{lem}\label{projection_minors}
Let $V\in\Gr_{k,n}(\mathbb{C})$. Then for $1 \le l \le n$, we have
\begin{align}\label{projection_minors_equation}
\Delta_{I,J}(\Proj{V}) = \frac{\displaystyle\sum_{K\in\binom{[n]\setminus (I\cup J)}{k-l}}\!(-1)^{\inv{I}{K} + \inv{J}{K}}\Delta_{I\cup K}(V)\ccon{\Delta_{J\cup K}(V)}}{\displaystyle\sum_{K\in\binom{[n]}{k}}\!|\Delta_K(V)|^2} \quad \text{ for all } I,J\in\textstyle\binom{[n]}{l}.
\end{align}
In particular, we have
\begin{align}\label{projection_maximal_minors}
\Delta_{I,J}(\Proj{V}) = \frac{\Delta_I(V)\ccon{\Delta_J(V)}}{\displaystyle\sum_{K\in\binom{[n]}{k}}\!|\Delta_K(V)|^2} \quad \text{ for all } I,J\in\textstyle\binom{[n]}{k}.
\end{align}

\end{lem}

\begin{proof}
We regard $V$ as an $n\times k$ matrix, so that $\Proj{V} = V(\adjoint{V}V)^{-1}\adjoint{V}$. Then for $1 \le l \le n$ and $I,J\in\binom{[n]}{l}$, we have
\begin{align*}
&\Delta_{I,J}(\Proj{V}) = \sum_{I',J'\in\binom{[k]}{l}}\!\Delta_{I,I'}(V)\Delta_{I',J'}((\adjoint{V}V)^{-1})\Delta_{J',J}(\adjoint{V}) \qquad \text{(by \eqref{cauchy-binet})} \\
&= \sum_{I',J'\in\binom{[k]}{l}}\!\Delta_{I,I'}(V)\left(\frac{(-1)^{\sumof{I'} + \sumof{J'}}}{\displaystyle\det(\adjoint{V}V)}\Delta_{[k]\setminus J',[k]\setminus I'}(\adjoint{V}V)\right)\ccon{\Delta_{J,J'}(V)} \qquad \text{(by \eqref{jacobi})} \\
&= \sum_{I',J'\in\binom{[k]}{l}}\!\Delta_{I,I'}(V)\Bigg(\frac{(-1)^{\sumof{I'} + \sumof{J'}}}{\displaystyle\det(\adjoint{V}V)}\sum_{K\in\binom{[n]}{k-l}}\!\ccon{\Delta_{K,[k]\setminus J'}(V)}\Delta_{K,[k]\setminus I'}(V)\Bigg)\ccon{\Delta_{J,J'}(V)} \qquad \text{(by \eqref{cauchy-binet})} \\
&= \scalebox{0.92}{$\displaystyle\frac{1}{\displaystyle\det(\adjoint{V}V)}\sum_{K\in\binom{[n]}{k-l}}\raisebox{-4pt}{$\Bigg($}\sum_{I'\in\binom{[k]}{l}}(-1)^{\sumof{I'}}\Delta_{I,I'}(V)\Delta_{K,[k]\setminus I'}(V)\raisebox{-4pt}{$\Bigg)$}\raisebox{-4pt}{$\Bigg($}\sum_{J'\in\binom{[k]}{l}}(-1)^{\sumof{J'}}\ccon{\Delta_{J,J'}(V)\Delta_{K,[k]\setminus J'}(V)}\raisebox{-4pt}{$\Bigg)$}$},
\end{align*}
which simplifies to \eqref{projection_minors_equation} by \eqref{cauchy-binet} and \eqref{laplace}.
\end{proof}

\begin{cor}\label{projection_tnn_description}
Let $0 \le k \le n$.
\begin{enumerate}[label=(\roman*), leftmargin=*, itemsep=4pt]
\item\label{projection_tnn_description_tp} We have $\Orbit_{\bfomega{k}}^{>0} = \{\ii P\in\Orbit_{\bfomega{k}} : \textnormal{all $k\times k$ minors of $P$ are real and positive}\}$.
\item\label{projection_tnn_description_tnn} We have $\Orbit_{\bfomega{k}}^{\ge 0} = \{\ii P\in\Orbit_{\bfomega{k}} : \textnormal{all $k\times k$ minors of $P$ are real and nonnegative}\}$.
\end{enumerate}

\end{cor}

\begin{proof}
This follows from \eqref{projection_maximal_minors}, \cref{Gr_to_orbit}, and \cref{tnn_Fl_converse}.
\end{proof}

\begin{cor}\label{evenness}
Let $n\in\mathbb{N}$, and let $\ii P\in\Orbit_{\bfomega{k}}^{>0}$. Let $1 \le l \le k$, and suppose that $I,J\in\binom{[n]}{l}$ satisfy the {\itshape evenness condition}: between any two elements of $\mathbb{Z}\setminus (I\cup J)$, there are an even number of elements in the multiset union $I\cup J$.
\begin{enumerate}[label=(\roman*), leftmargin=*, itemsep=2pt]
\item\label{evenness_positive} If $|I\cap J| \ge k+l-n$, then $\Delta_{I,J}(P) > 0$.
\item\label{evenness_zero} If $|I\cap J| < k+l-n$, then $\Delta_{I,J}(P) = 0$.
\end{enumerate}

\end{cor}

\begin{proof}
The evenness condition implies that $\inv{I}{K} + \inv{J}{K}$ is even for all $K\subseteq \mathbb{Z}\setminus (I\cup J)$. Therefore the numerator of the right-hand side of \eqref{projection_minors_equation} is a sum of positive terms, by \cref{Gr_to_orbit} and \cref{tnn_Fl}\ref{tnn_Fl_tp}. The two cases correspond to whether the sum has at least one term or not.
\end{proof}

\begin{rmk}
\cref{evenness} implies that certain minors $\Delta_{I,J}(P)$ are positive or zero for all $\ii P\in\Orbit_{\bfomega{k}}^{>0}$. We can similarly argue that every other minor is either zero, negative, or can take any sign (we omit the details). For example, let $(k,n) := (2,4)$, and consider the $1\times 1$ minors (i.e.\ the entries) of $P$. We have $P_{1,1}, P_{2,2}, P_{3,3}, P_{4,4}, P_{1,2}, P_{2,3}, P_{3,4} > 0$ and $P_{1,4} < 0$. The remaining entries $P_{1,3}$ and $P_{2,4}$ can take any sign, as demonstrated by the matrices
\begin{gather*}
17P = \scalebox{0.9}{$\begin{bmatrix}
6 & 7 & 1 & -4 \\
7 & 11 & 4 & 1 \\
1 & 4 & 3 & 5 \\
-4 & 1 & 5 & 14
\end{bmatrix}$},\hspace*{6pt}
\scalebox{0.9}{$\begin{bmatrix}
3 & 4 & 1 & -5 \\
4 & 11 & 7 & -1 \\
1 & 7 & 6 & 4 \\
-5 & -1 & 4 & 14
\end{bmatrix}$},\hspace*{6pt}
\scalebox{0.9}{$\begin{bmatrix}
6 & 4 & -1 & -7 \\
4 & 14 & 5 & 1 \\
-1 & 5 & 3 & 4 \\
-7 & 1 & 4 & 11
\end{bmatrix}$},\hspace*{6pt}
\scalebox{0.9}{$\begin{bmatrix}
3 & 5 & -1 & -4 \\
5 & 14 & 4 & -1 \\
-1 & 4 & 6 & 7 \\
-4 & -1 & 7 & 11
\end{bmatrix}$}.
\end{gather*}

\end{rmk}

\subsection{Tridiagonal matrices}\label{sec_tridiagonal_orbit}
Tridiagonal matrices are often of particular interest in applications, and will play an important role throughout the paper. We give an explicit description of the tridiagonal parts of $\Orbit_{\bflambda}^{>0}$ and $\Orbit_{\bflambda}^{\ge 0}$.
\begin{defn}\label{defn_jacobi_matrix}
Let $\bflambda\in\mathbb{R}^n$ be weakly decreasing. We define the spaces of {\itshape Jacobi matrices}
$$
\Jac_{\bflambda}^{>0} := (\ii\gl_n^{>0})\cap\Orbit_{\bflambda} \quad \text{ and } \quad \Jac_{\bflambda}^{\ge 0} := (\ii\gl_n^{\ge 0})\cap\Orbit_{\bflambda}.
$$
That is, $\Jac_{\bflambda}^{>0}$ (respectively, $\Jac_{\bflambda}^{\ge 0}$) is the set of elements $\ii L\in\Orbit_{\bflambda}$ such that $L$ is a real tridiagonal matrix with positive (respectively, nonnegative) entries immediately above and below the diagonal.
\end{defn}

We will show that $\Jac_{\bflambda}^{>0}$ (respectively, $\Jac_{\bflambda}^{\ge 0}$) is precisely the subset of tridiagonal elements of $\Orbit_{\bflambda}^{>0}$ (respectively, $\Orbit_{\bflambda}^{\ge 0}$). We then give an explicit description of $\Jac_{\bflambda}^{>0}$ in terms of {\itshape Vandermonde flags}.
\begin{lem}\label{projection_sum}
Let $\bflambda\in\mathbb{R}^n$ be weakly decreasing, and set $K := \{i \in [n-1] : \lambda_i > \lambda_{i+1}\}$. Given $\ii L\in\Orbit_{\bflambda}$, let $V = (V_k)_{k\in K}\in\PFl{K}{n}(\mathbb{C})$ be the corresponding flag under the inverse map of \eqref{Fl_to_orbit_map}. Then
\begin{align}\label{projection_sum_formula}
L = \Big(\sum_{k\in K}(\lambda_k - \lambda_{k+1})P_k\Big) + \lambda_n\I_n, \quad \text{ where } P_k := \Proj{V_k} \text{ for } k\in K.
\end{align}

\end{lem}

\begin{proof}
This follows from \cref{Gr_to_orbit}, by writing $\bflambda = \big(\sum_{k\in K}(\lambda_k - \lambda_{k+1})\bfomega{k}\big) + \lambda_n\bfomega{n}$.
\end{proof}

\begin{lem}\label{positive_superdiagonal}
Suppose that $\bflambda\in\mathbb{R}^n$ is weakly decreasing and nonconstant.
\begin{enumerate}[label=(\roman*), leftmargin=*, itemsep=2pt]
\item\label{positive_superdiagonal_tp} If $\ii L\in\Orbit_{\bflambda}^{>0}$, then $L_{i,i+1} = L_{i+1,i} > 0$ for $1 \le i \le n-1$.
\item\label{positive_superdiagonal_tnn} If $\ii L\in\Orbit_{\bflambda}^{\ge 0}$, then $L_{i,i+1} = L_{i+1,i} \ge 0$ for $1 \le i \le n-1$.
\end{enumerate}

\end{lem}

We note that if $\bflambda$ is constant, then $\Orbit_{\bflambda}^{>0} = \Orbit_{\bflambda}^{\ge 0} = \Orbit_{\bflambda} = \{\ii\Diag{\bflambda}\hspace*{-1pt}\}$.
\begin{proof}
We prove part \ref{positive_superdiagonal_tp}, whence part \ref{positive_superdiagonal_tnn} follows since $\Orbit_{\bflambda}^{\ge 0} = \overline{\Orbit_{\bflambda}^{>0}}$. Set $K := \{i \in [n-1] : \lambda_i > \lambda_{i+1}\}$, which is nonempty by assumption. Let $\ii L\in\Orbit_{\bflambda}$, and let $V = (V_k)_{k\in K}\in\PFl{K}{n}^{>0}$ be the corresponding flag under the inverse map of \eqref{Fl_to_orbit_map}. Then by \eqref{projection_sum_formula}, we have
$$
L_{i,i+1} = \sum_{k\in K}(\lambda_k - \lambda_{k+1})(\Proj{V_k})_{i,i+1} \quad \text{ for } 1 \le i \le n-1.
$$
By \cref{Gr_to_orbit} and \cref{evenness}\ref{evenness_positive}, we have $(\Proj{V_k})_{i,i+1} > 0$ for all $k\in K$. Therefore $L_{i,i+1} > 0$.
\end{proof}

\begin{prop}\label{tridiagonal_orbit}
Let $\bflambda\in\mathbb{R}^n$ be weakly decreasing and nonconstant. Then
$$
\Jac_{\bflambda}^{>0} = \{L\in\Orbit_{\bflambda}^{>0} : L \textnormal{ is tridiagonal}\} \quad \text{ and } \quad \Jac_{\bflambda}^{\ge 0} = \{L\in\Orbit_{\bflambda}^{\ge 0} : L \textnormal{ is tridiagonal}\}.
$$
Moreover, if $\bflambda$ is not strictly decreasing, then $\Jac_{\bflambda}^{>0}$ is empty.
\end{prop}

\begin{proof}
The containments $\supseteq$ follow from \cref{positive_superdiagonal}. To prove the first $\subseteq$ containment, let $\ii L\in\Jac_{\bflambda}^{>0}$.  Then $L\in\gl_n^{>0}$, so $\exp(L)\in\GL_n^{>0}$. Applying \cref{gk}\ref{gk_eigenvectors} to $\exp(L)$ implies $\ii L\in\Orbit_{\bflambda}^{>0}$, as desired. Moreover, \cref{gk}\ref{gk_eigenvalues} implies that if such an $\ii L$ exists, then $\bflambda$ is strictly decreasing. The second $\subseteq$ containment follows from a similar argument, using \cref{gk_tnn}\ref{gk_tnn_eigenvectors}.
\end{proof}

When $\bflambda$ is strictly decreasing, the space $\Jac_{\bflambda}^{>0}$ is known as an {\itshape isospectral manifold of Jacobi matrices}. It was first considered by Moser \cite{moser75} in connection with the Toda lattice, based on work of Flaschka \cite{flaschka74}. We will discuss the Toda lattice further in \cref{sec_toda}. The topology of $\Jac_{\bflambda}^{>0}$ was studied by Tomei \cite{tomei84}, who showed in particular that its closure is $\Jac_{\bflambda}^{\ge 0}$. Bloch, Flaschka, and Ratiu \cite{bloch_flaschka_ratiu90} gave the following descriptions of $\Jac_{\bflambda}^{>0}$ and $\Jac_{\bflambda}^{\ge 0}$, which hold for any compact Lie algebra.
\begin{thm}[Bloch, Flaschka, and Ratiu {\cite[Theorem p.\ 60]{bloch_flaschka_ratiu90}}]\label{jacobi_manifold}
Let $\bflambda\in\mathbb{R}^n$ be strictly decreasing, and let $C\subseteq\mathbb{R}^n$ denote the convex hull of all permutations of $\bflambda$, which is the moment polytope of $\Orbit_{\bflambda}$. Then there is a diffeomorphism from $\Jac_{\bflambda}^{>0}$ to the interior of $C$ which extends to a homeomorphism $\Jac_{\bflambda}^{\ge 0} \xrightarrow{\cong} C$.
\end{thm}

See \cref{BFR_map} for further discussion. We plan to study the homeomorphism $\Jac_{\bflambda}^{\ge 0} \to C$ in more detail in future work.

We now describe $\Jac_{\bflambda}^{>0}$ as a subset of $\Fl_n^{>0}$ under the identification \eqref{Fl_to_orbit_map}. Remarkably, it is a twisted $\H_n^{>0}$-orbit. This is based on a well-known correspondence in numerical analysis between orthogonal tridiagonalization of a symmetric matrix (which we uncharacteristically take to be a diagonal matrix) and Krylov subspaces; we refer to \cite{parlett98,golub_van_loan13} for further details. This description is also related to Moser's spectral variables \cite[Section 3]{moser75} for the manifold $\Jac_{\bflambda}^{>0}$; see \cref{moser_variables}.
\begin{defn}\label{defn_krylov}
Let $\bflambda = (\lambda_1, \dots, \lambda_n)\in\mathbb{C}^n$ have distinct entries, and let $x\in\mathbb{P}^{n-1}(\mathbb{C})$ have no zero entries. Define the {\itshape Vandermonde flag} $\Kry{\bflambda}{x}\in\Fl_n(\mathbb{C})$ as the complete flag $(V_1, \dots, V_{n-1})$, where
$$
V_k := \spn(x, \Diag{\bflambda}x, \dots, \Diag{\bflambda}\hspace*{-1pt}^{k-1}x) \quad \text{ for } 1 \le k \le n-1.
$$
That is, $\Kry{\bflambda}{x}$ is represented by the rescaled Vandermonde matrix $(\lambda_i^{j-1}x_i)_{1 \le i,j \le n}$. The fact that $\Kry{\bflambda}{x}$ lies in $\Fl_n(\mathbb{C})$ follows from \eqref{vandermonde}. Moreover, if $\lambda_1, \dots, \lambda_n$ are strictly decreasing real numbers and $x\in\mathbb{P}^{n-1}_{>0}$, then $\Kry{\bflambda}{x}\in\Fl_n^{>0}$, by \eqref{vandermonde} and \cref{tnn_Fl_converse}. We also observe that $\bflambda$ and $x$ are uniquely determined by $\Kry{\bflambda}{x}$, modulo translating $\bflambda$ by a scalar multiple of $(1, \dots, 1)$ and rescaling it by a nonzero constant.
\end{defn}

\begin{eg}\label{eg_krylov}
When $n=3$, the flag $\Kry{\bflambda}{x}$ is represented by $\scalebox{0.8}{$\begin{bmatrix}
x_1 & \lambda_1x_1 & \lambda_1^2x_1 \\
x_2 & \lambda_2x_2 & \lambda_2^2x_2 \\
x_3 & \lambda_3x_3 & \lambda_3^2x_3
\end{bmatrix}$}$.
\end{eg}

\begin{rmk}\label{krylov_torus_orbit}
Let $\bflambda\in\mathbb{C}^n$ have distinct entries. Recall the torus action on $\Fl_n(\mathbb{C})$ from \cref{torus_action}. For $x\in\mathbb{P}^{n-1}(\mathbb{C})$ with no zero entries and $h\in\H_n(\mathbb{C})$, we have
$$
h\hspace*{1pt}\Kry{\bflambda}{x} = \Kry{\bflambda}{hx}.
$$
In particular, the subset
$$
\{\Kry{\bflambda}{x} : x\in\mathbb{P}^{n-1}(\mathbb{C}) \text{ has no zero entries}\}\subseteq\Fl_n(\mathbb{C})
$$
is a $\H_n(\mathbb{C})$-orbit. Similarly, if the entries of $\lambda$ are strictly decreasing real numbers, then
$$
\{\Kry{\bflambda}{x} : x\in\mathbb{P}^{n-1}_{>0}\}\subseteq\Fl_n^{>0}
$$
is a $\H_n^{>0}$-orbit.
\end{rmk}

\begin{lem}[cf.\ {\cite[Theorem 8.3.1]{golub_van_loan13}}]\label{tridiagonal_to_krylov}
Let $\bflambda\in\mathbb{C}^n$ have distinct entries, let $g\in\U_n$, let $L := g\Diag{\bflambda}g^{-1}\in\gl_n(\mathbb{C})$, and let $x$ be the first column of $\iota(g)$. Then the following are equivalent:
\begin{enumerate}[label=(\roman*), leftmargin=*, itemsep=2pt]
\item\label{tridiagonal_to_krylov_tridiagonality} $L$ is tridiagonal and $L_{i,i+1}\neq 0$ for $1 \le i \le n-1$; and
\item\label{tridiagonal_to_krylov_flag} all entries of $x$ are nonzero, and the projection of $\iota(g)$ to $\Fl_n(\mathbb{C})$ equals $\Kry{\bflambda}{x}$.
\end{enumerate}

\end{lem}

\begin{proof}
Our argument follows \cite[Theorem 8.3.1]{golub_van_loan13}, which proves the implication \ref{tridiagonal_to_krylov_tridiagonality} $\Rightarrow$ \ref{tridiagonal_to_krylov_flag} over the real numbers. Let $M$ denote the matrix $(\lambda_i^{j-1}x_i)_{1 \le i,j \le n}$, which represents the flag $\Kry{\bflambda}{x}$. Then part \ref{tridiagonal_to_krylov_flag} is equivalent to the statement $\iota(g)^{-1}M\in\B_n(\mathbb{C})$. On the other hand, column $j$ of $\iota(g)^{-1}M$ (for $1 \le j \le n)$ is
$$
\iota(g)^{-1}\Diag{\bflambda}\hspace*{-1pt}^{j-1}x = (\delta_ng\delta_n)\Diag{\bflambda}\hspace*{-1pt}^{j-1}(\delta_ng^{-1}\delta_ne_1) = \delta_nL^{j-1}e_1.
$$
Therefore part \ref{tridiagonal_to_krylov_tridiagonality} is also equivalent to the statement $\iota(g)^{-1}M\in\B_n(\mathbb{C})$.
\end{proof}

\begin{cor}\label{tridiagonal_flag}
Let $\bflambda\in\mathbb{R}^n$ be strictly decreasing. Then the inverse map of \eqref{Fl_to_orbit_map} identifies $\Jac_{\bflambda}^{>0}$ with a twisted totally positive torus orbit of Vandermonde flags inside $\Fl_n^{>0}$:
$$
\Jac_{\bflambda}^{>0} \hspace*{2pt}\xrightarrow{\cong}\hspace*{2pt} \twist\big(\{\Kry{\bflambda}{x} : x\in\mathbb{P}^{n-1}_{>0}\}\big) \subseteq \Fl_n^{>0}.
$$
\end{cor}

\begin{proof}
Let $\ii L$ denote an arbitrary element of $\Orbit_{\bflambda}^{>0}$, so that $L = g\Diag{\bflambda}g^{-1}$ for some $g\in\U_n^{>0}$. Then the inverse map of \eqref{Fl_to_orbit_map} sends $\ii L$ to $g\in\Fl_n^{>0}$. By \cref{tridiagonal_orbit}, it suffices to prove that the following two statements are equivalent:
\begin{enumerate}[label=(\roman*), leftmargin=*, itemsep=2pt]
\item $L$ is tridiagonal and $L_{i,i+1}\neq 0$ for $1 \le i \le n-1$; and
\item $\twist(g) = \Kry{\bflambda}{x}$ for some $x\in\mathbb{P}^{n-1}_{>0}$.
\end{enumerate}
Note that  the first column of $\iota(g)$ has positive entries, by \cref{positive_inverse}\ref{positive_inverse_action} and \cref{U_to_Fl}\ref{U_to_Fl_tp}. Therefore the result follows from \cref{tridiagonal_to_krylov}.
\end{proof}

\begin{eg}\label{eg_tridiagonal_flag}
We illustrate \cref{tridiagonal_flag} in the case $n := 3$. Let $\bflambda := (1, 0, -1)$, and let $x\in\mathbb{P}^2_{>0}$. Then $\Kry{\bflambda}{x}\in\Fl_3^{>0}$ is represented by the matrix
$$
\begin{bmatrix}
x_1 & x_1 & x_1 \\
x_2 & 0 & 0 \\
x_3 & -x_3 & x_3
\end{bmatrix}.
$$
We act on the right by $\B_3(\mathbb{C})$ to turn this matrix into an element of $\U_3^{>0}$:
$$
\begin{bmatrix}
\frac{x_1}{\sqrt{x_1^2 + x_2^2 + x_3^2}} & \frac{-x_1(x_2^2 + 2x_3^2)}{\sqrt{(x_1^2 + x_2^2 + x_3^2)(x_1^2x_2^2 + 4x_1^2x_3^2 + x_2^2x_3^2)}} & \frac{x_2x_3}{\sqrt{x_1^2x_2^2 + 4x_1^2x_3^2 + x_2^2x_3^2}} \\[10pt]
\frac{x_2}{\sqrt{x_1^2 + x_2^2 + x_3^2}} & \frac{x_2(x_1^2 - x_3^2)}{\sqrt{(x_1^2 + x_2^2 + x_3^2)(x_1^2x_2^2 + 4x_1^2x_3^2 + x_2^2x_3^2)}} & \frac{-2x_1x_3}{\sqrt{x_1^2x_2^2 + 4x_1^2x_3^2 + x_2^2x_3^2}} \\[10pt]
\frac{x_3}{\sqrt{x_1^2 + x_2^2 + x_3^2}} & \frac{x_3(2x_1^2 + x_2^2)}{\sqrt{(x_1^2 + x_2^2 + x_3^2)(x_1^2x_2^2 + 4x_1^2x_3^2 + x_2^2x_3^2)}} & \frac{x_1x_2}{\sqrt{x_1^2x_2^2 + 4x_1^2x_3^2 + x_2^2x_3^2}} \\[8pt]
\end{bmatrix} =: \iota(g).
$$
Setting $\ii L := g(\ii\Diag{\bflambda})g^{-1}\in\Orbit_{\bflambda}$, we find that
$$
L = \begin{bmatrix}
\frac{x_1^2 - x_3^2}{x_1^2 + x_2^2 + x_3^2} & \frac{\sqrt{x_1^2x_2^2 + 4x_1^2x_3^2 + x_2^2x_3^2}}{x_1^2 + x_2^2 + x_3^2} & 0 \\[8pt]
\frac{\sqrt{x_1^2x_2^2 + 4x_1^2x_3^2 + x_2^2x_3^2}}{x_1^2 + x_2^2 + x_3^2} & \frac{(x_1^2 - x_3^2)(x_2^4 - 4x_1^2x_3^2)}{(x_1^2 + x_2^2 + x_3^2)(x_1^2x_2^2 + 4x_1^2x_3^2 + x_2^2x_3^2)} & \frac{2x_1x_2x_3\sqrt{x_1^2 + x_2^2 + x_3^2}}{x_1^2x_2^2 + 4x_1^2x_3^2 + x_2^2x_3^2} \\[8pt]
0 & \frac{2x_1x_2x_3\sqrt{x_1^2 + x_2^2 + x_3^2}}{x_1^2x_2^2 + 4x_1^2x_3^2 + x_2^2x_3^2} & \frac{x_2^2(x_3^2 - x_1^2)}{x_1^2x_2^2 + 4x_1^2x_3^2 + x_2^2x_3^2} \\[4pt]
\end{bmatrix}.
$$
Note that $L$ indeed lies in $\Jac_{\bflambda}^{>0}$, i.e., it is tridiagonal and $L_{1,2}, L_{2,3} > 0$.
\end{eg}

\begin{rmk}
\cref{tridiagonal_flag} demonstrates that the twist map $\twist$ acts in an elegant way on Vandermonde flags. We can also describe the action of the maps $\rev$ and $\flip$ on Vandermonde flags. Namely, let $\bflambda\in\mathbb{R}^n$ have distinct entries, and let $x\in\mathbb{P}^{n-1}(\mathbb{C})$. Then
\begin{align}\label{krylov_rev}
\rev(\Kry{(\lambda_1, \dots, \lambda_n)}{(x_1 : \cdots : x_n)}) = \Kry{(\lambda_n, \dots, \lambda_1)}{(x_n : \cdots : x_1)},
\end{align}
and
\begin{align}\label{krylov_flip}
\flip(\Kry{\bflambda}{x}) = \Kry{\bflambda}{y}, \quad \text{ where } \quad \ccon{y_i} = \frac{(-1)^{i-1}}{x_i\prod_{j\neq i}(\lambda_i - \lambda_j)}\hspace*{2pt} \text{ for } 1 \le i \le n.
\end{align}
The statement \eqref{krylov_rev} follows from \cref{defn_krylov}. We can prove \eqref{krylov_flip} using a version of \cref{tridiagonal_to_krylov} which involves the last column of $\iota(g)$, rather than the first column (we omit the details).

For example, let $\bflambda := (1, 0, -1)$, as in \cref{eg_tridiagonal_flag}. Then
$$
\rev(\Kry{\bflambda}{x}) = \Kry{(-1,0,1)}{(x_3 : x_2 : x_1)} = \Kry{\bflambda}{(x_3 : x_2 : x_1)},
$$
and
\begin{gather*}
\flip(\Kry{\bflambda}{x}) = \Kry{\bflambda}{y}, \quad \text{ where } y = \big(\textstyle\frac{1}{2x_1} : \frac{1}{x_2} : \frac{1}{2x_3}\big).
\end{gather*}

\end{rmk}

\begin{rmk}\label{moser_variables}
\cref{tridiagonal_flag} gives an explicit parametrization of $\Jac_{\bflambda}^{>0}$ by $\mathbb{P}_{>0}^{n-1}$. This parametrization was first introduced by Moser \cite[Section 3]{moser75}. Specifically, Moser's variables $r_1, \dots, r_n$ (required to be positive and satisfy $r_1^2 + \cdots + r_n^2 = 1$) are obtained by normalizing our $x\in\mathbb{P}_{>0}^{n-1}$, i.e.,
$$
r_j = \frac{x_j}{\sqrt{x_1^2 + \cdots + x_n^2}} \quad \text{ for } 1 \le j \le n.
$$
Moser's motivation was to give an explicit description of the tridiagonal symmetric Toda lattice, as we discuss further in \cref{torus_remark}.

While it is relatively simple to describe how to go from a matrix in $\Jac_{\bflambda}^{>0}$ to its parameters $r_1, \dots, r_n$ (for example, they are the normalized first components of the eigenvectors), the reverse process is nontrivial. The procedure we give above in terms of the twist map is qualitatively different from Moser's, while another approach was described by Deift, Lund, and Trubowitz \cite[Theorem p.\ 178]{deift_lund_trubowitz80} (cf.\ \cite[Theorem 2]{deift_nanda_tomei83}). These procedures are all ultimately equivalent; the novelty in our approach is our use of the twist map, and in the connection to total positivity.

For example, let us verify that the calculation of $L_{1,2}$ in \cref{eg_tridiagonal_flag} is consistent with the procedure described in \cite[Theorem 2]{deift_nanda_tomei83}. The formula therein states that
$$
L_{1,2}^2 = \sum_{j=1}^3((\lambda_j-a_1)r_j)^2, \quad \text{ where } r_j = \frac{x_j}{\sqrt{x_1^2 + x_2^2 + x_3^2}} \text{ and } a_1 = L_{1,1} = r_1^2 - r_3^2.
$$
Using $r_1^2 + r_2^2 + r_3^2 = 1$, we obtain
\begin{align*}
L_{1,2}^2 &= (1-a_1)^2r_1^2 + (0-a_1)^2r_2^2 + (-1-a_1)^2r_3^2 \\
&= (r_2^2 + 2r_3^2)^2r_1^2 + (-r_1^2 + r_3^2)^2r_2^2 + (-2r_1^2 - r_2^2)^2r_3^2 \\
&= (r_1^2r_2^2 + 4r_1^2r_3^2 + r_2^2r_3^2)(r_1^2 + r_2^2 + r_3^2) \\
&= r_1^2r_2^2 + 4r_1^2r_3^2 + r_2^2r_3^2,
\end{align*}
which indeed agrees with \cref{eg_tridiagonal_flag}.
\end{rmk}

Finally, we introduce the space of all totally positive Vandermonde flags. It will play an important role in \cref{sec_amplituhedron}.
\begin{defn}\label{defn_vandermonde_flags}
Given $n\in\mathbb{N}$, let $\Vand_n^{>0}$ denote the subset of $\Fl_n^{>0}$ of all totally positive Vandermonde flags:
$$
\Vand_n^{>0} := \{\Kry{\bflambda}{x} : \bflambda\in\mathbb{R}^n \text{ is strictly decreasing and } x\in\mathbb{P}^{n-1}_{>0}\} \subseteq \Fl_n^{>0}.
$$
\end{defn}

\begin{cor}\label{all_tridiagonal_flags}
We have the following bijection between the space of all Jacobi matrices modulo translation by scalar multiples of $\I_n$ and rescaling by $\mathbb{R}_{>0}$, and the space of twisted totally positive Vandermonde flags:
$$
(\ii\gl_n^{>0})\cap\uu_n /{\sim} \xrightarrow{\cong} \twist(\Vand_n^{>0}), \quad g(\ii\Diag{\bflambda})g^{-1} \mapsto g.
$$
Above, two matrices $L,M$ are equivalent under ${\sim}$ if and only if $M = t(L + c\I_n)$ for some $t > 0$ and $c\in\mathbb{R}$.
\end{cor}

\begin{proof}
Recall that $\bflambda$ and $x$ are uniquely determined by $\Kry{\bflambda}{x}$, modulo translating $\bflambda$ by a scalar multiple of $(1, \dots, 1)$ and rescaling it by a nonzero constant. Also, by \cref{tridiagonal_orbit}, $(\ii\gl_n^{>0})\cap\uu_n$ is the disjoint union of $\Jac_{\bflambda}^{>0}$ over all strictly decreasing $\bflambda\in\mathbb{R}^n$. Therefore the result follows from \cref{tridiagonal_flag}.
\end{proof}

Recall from \cref{krylov_torus_orbit} that the totally positive part of the torus $\H_n^{>0}$ acts on $\Vand_n^{>0}$. Surprisingly, $\H_n^{>0}$ also acts on $\twist(\Vand_n^{>0})$:
\begin{lem}\label{all_tridiagonal_flags_torus_orbit}
The space of twisted totally positive Vandermonde flags $\twist(\Vand_n^{>0})$ is invariant under the action of the totally positive part of the torus $\H_n^{>0}$.
\end{lem}

\begin{proof}
Consider the action of $\H_n(\mathbb{C})$ on $\uu_n$ by conjugation. Note that $(\ii\gl_n^{>0})\cap\uu_n$ is invariant under $\H_n^{>0}$. The result then follows from \cref{all_tridiagonal_flags}.
\end{proof}

\begin{rmk}\label{vandermonde_flags_property}
A further property shared by $\Vand_n^{>0}$ and $\twist(\Vand_n^{>0})$ is that they are both naturally in bijection with $\PFl{\{1,2\}}{n}^{>0}$. In particular, the projection map $\Fl_n^{>0} \to \PFl{\{1,2\}}{n}^{>0}$ restricts to a bijection on both $\Vand_n^{>0}$ and $\twist(\Vand_n^{>0})$. In the case of $\Vand_n^{>0}$, this follows from \cref{defn_krylov} and \cref{tnn_Fl}\ref{tnn_Fl_tp}. In the case of $\twist(\Vand_n^{>0})$, this is not straightforward to prove; we will do so in \cref{twist_projection_bijection}.
\end{rmk}

\section{Gradient flows on adjoint orbits}\label{sec_gradient}

\noindent In this section, we study gradient flows on a partial flag variety, viewed as an adjoint orbit $\Orbit_{\bflambda}$ of $\uu_n$. We consider gradient flows for functions of the form $\kappa(\cdot,N)$ for fixed $N\in\uu_n$, where $\kappa$ is the Killing form of $\uu_n$, in three natural Riemannian metrics: the K\"{a}hler, normal, and induced metrics. We point out that when $\Orbit_{\bflambda}$ is isomorphic to a Grassmannian, then these three metrics coincide up to dilation, but otherwise they are distinct. Our goal will be to determine when such a flow preserves the totally nonnegative part $\Orbit_{\bflambda}^{\ge 0}$. In the case of the K\"{a}hler metric, we completely classify which gradient flows preserve positivity. In the case of the normal metric, we show that when $\Orbit_{\bflambda}$ is isomorphic to $\Fl_n(\mathbb{C})$ with $n\ge 3$, there are no nontrivial gradient flows which preserve positivity. In the case of the induced metric, we make some preliminary investigations which indicate that whether or not there exists a non-trivial gradient flow on $\Orbit_{\bflambda}$ which preserves positivity can depend on the spacing between the entries of $\bflambda$.

We refer to \cite[Section 4.1]{abraham_marsden_ratiu88}, \cite[Chapter 8]{besse87}, and \cite[Section 15.2]{bloch_morrison_ratiu13} for background. For a given flow under consideration, we let $L(t)$ (for $t\in\mathbb{R}$) denote the flow beginning at $L(0) = L_0$, and we let $\dot{L}(t)$ denote the derivative of $L(t)$ with respect to $t$. Since $\Orbit_{\bflambda}$ is compact, all flows we consider are complete, i.e., they are defined for all $t\in\mathbb{R}$ \cite[Corollary 4.1.20]{abraham_marsden_ratiu88}. If $\dot{L}(0) = 0$, we call $L_0$ an {\itshape equilibrium}.
\begin{defn}\label{defn_positivity_preserving}
Let $\bflambda\in\mathbb{R}^n$ be weakly decreasing. We say that a flow on $\Orbit_{\bflambda}$ {\itshape weakly preserves positivity} if
$$
L(t)\in\Orbit_{\bflambda}^{\ge 0} \quad \text{ for all } L_0\in\Orbit_{\bflambda}^{\ge 0} \text{ and } t \ge 0,
$$
and {\itshape strictly preserves positivity} if
$$
L(t)\in\Orbit_{\bflambda}^{>0} \quad \text{ for all } L_0\in\Orbit_{\bflambda}^{\ge 0} \text{ and } t > 0.
$$
(So, every flow which strictly preserves positivity also weakly preserves positivity.) We make the analogous definitions for $\PFl{K}{n}(\mathbb{C})$ and $\U_n$.
\end{defn}

For example, the constant flow on $\Orbit_{\bflambda}$ weakly preserves positivity, but it does not strictly preserve positivity unless $\bflambda$ is constant (in which case $\Orbit_{\bflambda}$ is a point). We emphasize that in \cref{defn_positivity_preserving}, we require that positivity is preserved for {\itshape all} initial choices $L_0\in\Orbit_{\bflambda}^{\ge 0}$. In general, it is possible that the flow $L(t)$ remains in $\Orbit_{\bflambda}^{\ge 0}$ for some choices of $L_0\in\Orbit_{\bflambda}^{\ge 0}$, but not for others; see \cref{sometimes_preserving} for an intriguing instance of this phenomenon.
\begin{defn}\label{defn_gradient_flow}
Let $\kappa$ denote the {\itshape Killing form} on $\gl_n(\mathbb{C})$, given by
$$
\kappa(L,M) := 2n\tr(LM) - 2\tr(L)\tr(M) \quad \text{ for all } L,M\in\gl_n(\mathbb{C}).
$$
Then $-\kappa(\cdot,\cdot)$ defines a $[\cdot,\cdot]$-invariant pairing (i.e.\ $\kappa(\ad_L(M),N) = -\kappa(M,\ad_L(N))$) which is positive semidefinite on $\uu_n$.

Now let $\bflambda\in\mathbb{R}^n$ be weakly decreasing, and fix a Riemannian metric on $\Orbit_{\bflambda}$. Given $N\in\uu_n$, we define the {\itshape gradient flow on $\Orbit_{\bflambda}$ with respect to $N$} (in the given metric) as the flow given by
\begin{align}\label{defn_gradient_flow_equation}
\dot{L}(t) = \grad(H)(L(t)), \quad \text{ where } H(M) := \kappa(M,N) \text{ for all } M\in\Orbit_{\bflambda}.
\end{align}
We emphasize that we use the steepest ascent sign convention for the gradient flow.
\end{defn}

\begin{rmk}\label{eg_constant_flow}
We are interested in gradient flows on $\Orbit_{\bflambda}$ which preserve positivity with respect to some $N\in\uu_n$ (in a given metric). We point out that a necessary condition on $N$ is that it is purely imaginary, i.e., $\ii N$ is a real symmetric matrix.
\end{rmk}

\subsection{Background}\label{sec_gradient_background}
We briefly review the definitions of the three metrics we will consider, following \cite[Section 15.2]{bloch_morrison_ratiu13}; also see \cite[Section 4]{atiyah82}.
\begin{defn}\label{defn_metrics}
Let $\bflambda\in\mathbb{R}^n$ be weakly decreasing, and let $L\in\Orbit_{\bflambda}$.
\begin{itemize}[leftmargin=24pt, itemsep=2pt]
\item For $X\in\uu_n$, define $X^L$ and $X_L$ by the (unique) decomposition
\begin{align}\label{ad_decomposition}
X = X^L + X_L, \quad \text{ where $X^L\in\im(\ad_L)$ and $X_L\in\ker(\ad_L)$}.
\end{align}
Then the {\itshape normal metric} (or {\itshape standard metric}) on $\Orbit_{\bflambda}$ is given at $L\in\Orbit_{\bflambda}$ by
$$
\langle [L,X],[L,Y]\rangle_{\textnormal{normal}} := -\kappa(X^L,Y^L)
$$
for all tangent vectors $[L,X]$ and $[L,Y]$ at $L$.
\item The {\itshape induced metric} on $\Orbit_{\bflambda}$ is given at $L\in\Orbit_{\bflambda}$ by
$$
\langle [L,X],[L,Y]\rangle_{\textnormal{induced}} := -\kappa([L,X],[L,Y]) = \langle -\hspace*{-2pt}\ad_L^2([L,X]),[L,Y]\rangle_{\textnormal{normal}}
$$
for all tangent vectors $[L,X]$ and $[L,Y]$ at $L$.

\item Let $\sqrt{-\hspace*{-2pt}\ad_L^2}$ denote the positive square root of the positive semidefinite operator $-\hspace*{-2pt}\ad_L^2$. Then the {\itshape K\"{a}hler metric} on $\Orbit_{\bflambda}$ is given at $L\in\Orbit_{\bflambda}$ by
$$
\langle [L,X],[L,Y]\rangle_{\textnormal{K\"{a}hler}} := \langle\textstyle\sqrt{-\hspace*{-2pt}\ad_L^2}([L,X]),[L,Y]\rangle_{\textnormal{normal}}
$$
for all tangent vectors $[L,X]$ and $[L,Y]$ at $L$.
\end{itemize}

\end{defn}

We remark that the K\"{a}hler metric depends only on the corresponding flag variety under the identification \eqref{Fl_to_orbit_map}, not on the specific values of $\bflambda$ (aside from their multiplicities). This is in contrast to the normal and induced metrics, which do depend on the specific values of $\bflambda$.

A notable special case is when $\Orbit_{\bflambda}$ is isomorphic to a Grassmannian, as in \cref{sec_Gr_orbit}. Then the three metrics coincide up to dilation (cf.\ \cite[Section 4.2]{bloch_flaschka_ratiu90}), as we prove below. Therefore the three metrics on such $\Orbit_{\bflambda}$ give rise to the same gradient flows, but their descriptions are not obviously equivalent (see the running example: \cref{running_example_kahler}, \cref{running_example_normal}, and \cref{running_example_induced}). When considering flows which preserve positivity on such $\Orbit_{\bflambda}$, it will be most convenient to work in the K\"{a}hler metric, while in \cref{sec_lyapunov} we will work in the normal metric.
\begin{prop}\label{grassmannian_metrics}
Let $\bflambda\in\mathbb{R}^n$ with $\lambda_1 = \cdots = \lambda_k > \lambda_{k+1} = \cdots = \lambda_n$, so that $\Orbit_{\bflambda}\cong\Gr_{k,n}(\mathbb{C})$. Then the K\"{a}hler, normal, and induced metrics on $\Orbit_{\bflambda}$ all coincide up to dilation.
\end{prop}

\begin{proof}
By \cref{defn_metrics}, it suffices to show that for any $L\in\Orbit_{\bflambda}$, the operator $-\hspace*{-2pt}\ad_L^2$ acts as a positive scalar multiple of the identity on $\im(\ad_L)$. Indeed, we claim that
$$
-[L,[L,[L,M]]] = (\lambda_1 - \lambda_n)^2\hspace*{1pt}[L,M] \quad \text{ for all } M\in\uu_n.
$$
We can verify this by writing $-\ii L = (\lambda_1 - \lambda_n)P + \lambda_n\I_n$ for some $P\in\gl_n(\mathbb{C})$ with $P^2 = P = \adjoint{P}$, as in \eqref{projection_sum_formula}.
\end{proof}

We will only need to work with \cref{defn_metrics} in the case of the induced metric; for the K\"{a}hler and normal metrics, we will instead use known descriptions for their gradient flows, which we introduce in the respective subsections. For the induced metric, we will use the following general computation from \cite{bloch_morrison_ratiu13}:
\begin{lem}[{Bloch, Morrison, and Ratiu \cite[(15.4)]{bloch_morrison_ratiu13}}]\label{gradient_computation}
Fix a weakly decreasing $\bflambda\in\mathbb{R}^n$, a metric on $\Orbit_{\bflambda}$, and $N\in\uu_n$. Let $L(t)\in\Orbit_{\bflambda}$ evolve according to \eqref{defn_gradient_flow_equation}, i.e., the gradient flow with respect to $N$. Suppose that $M(t)\in\uu_n$ satisfies
$$
\langle [L(t),X],[L(t),M(t)]\rangle_{\textnormal{metric}} = \kappa([L(t),X],N)
$$
for all $t$ and tangent vectors $[L(t),X]$ at $L(t)$. Then we can write \eqref{defn_gradient_flow_equation} as
\begin{align}\label{gradient_computation_equation}
\dot{L}(t) = [L(t),M(t)].
\end{align}

\end{lem}

Since \eqref{gradient_computation_equation} is in {\itshape Lax form} \cite{lax68}, we can easily translate it into a flow on $\U_n$. We make some general observations about such flows.
\begin{lem}\label{lax_flow}
Let $\bflambda\in\mathbb{R}^n$ be weakly decreasing, set $K := \{i \in [n-1] : \lambda_i > \lambda_{i+1}\}$, and let $M(t)\in\uu_n$.
\begin{enumerate}[label=(\roman*), leftmargin=*, itemsep=2pt]
\item\label{lax_flow_translation} \textnormal{(Lax \cite[p.\ 470]{lax68})} Consider the flow on $\U_n$
$$
\dot{g}(t) = -M(t)g(t).
$$
Letting $L(t)$ denote $g(t)(\ii\Diag{\bflambda})g(t)^{-1}\in\Orbit_{\bflambda}$, we have the evolution
$$
\dot{L}(t) = [L(t),M(t)].
$$
\item\label{lax_flow_minors} Further suppose that $M(t)\in\oo_n$, and that $L(t)$ weakly (respectively, strictly) preserves positivity. Then for all $g_0\in\U_n^{\ge 0}$, we have
$$
\Delta_I(g(t)) \ge 0 \quad (\text{respectively}, > 0) \quad \text{ for all $k\in K$, $I\in\textstyle\binom{[n]}{k}$, and $t > 0$}.
$$
(If $K = [n-1]$, this means precisely that $g(t)$ weakly (respectively, strictly) preserves positivity in $\U_n$.) In particular, for all $k\in K$ and $I\in\binom{[n]}{k}$,
\begin{align}\label{lax_flow_minors_inequality}
\text{if } \Delta_I(g_0) = 0 \quad \text{ then } \quad \textstyle\frac{d}{dt}\eval{t=0}\Delta_I(g(t)) \ge 0.
\end{align}
\end{enumerate}

\end{lem}

\begin{proof}
We can verify part \ref{lax_flow_translation} directly. Part \ref{lax_flow_minors} follows from \cref{Fl_to_orbit}, \cref{tnn_Fl}, \cref{U_to_Fl}, and continuity of $g(t)\in\O_n$.
\end{proof}

\begin{rmk}\label{multilinearity_remark}
By multilinearity, we may express the derivated determinant in \eqref{lax_flow_minors_inequality} as follows:
\begin{align}\label{determinant_derivative}
\textstyle\frac{d}{dt}\eval{t=0}\Delta_I(g(t)) = \displaystyle\sum_{j=1}^k\Delta_I(g_0 \text{ with column $j$ replaced by column $j$ of } \dot{g}(0)).
\end{align}

\end{rmk}

\subsection{The K\"{a}hler metric}\label{sec_gradient_kahler}
In this subsection, we classify which gradient flows on $\Orbit_{\bflambda}$ with respect to $N\in\uu_n$ in the K\"{a}hler metric weakly or strictly preserve positivity. Namely, if $\ii N\in\gl_n^{\ge 0}$ then positivity is weakly preserved, and if $\ii N\in\gl_n^{>0}$ then positivity is strictly preserved. If $\Orbit_{\bflambda}$ is not isomorphic to a Grassmannian, then the converses to these statements hold. By contrast, in the Grassmannian case, there are additional such $N$ for which positivity is preserved; see \cref{positivity_preserving_kahler_grassmannian} and \cref{positivity_preserving_kahler_full}.

While the definition of the K\"{a}hler metric on $\Orbit_{\bflambda}$ is difficult to work with directly, its gradient flows admit a beautiful explicit solution. This has appeared in the literature in several places; see the work of Duistermaat, Kolk, and Varadarajan \cite[Section 3]{duistermaat_kolk_varadarajan83} and of Guest and Ohnita \cite[Appendix]{guest_ohnita93}, and the references therein.
\begin{prop}[{\cite[Section 3]{duistermaat_kolk_varadarajan83}; \cite[Appendix]{guest_ohnita93}}]\label{gradient_flow_kahler}
Let $\bflambda\in\mathbb{R}^n$ be weakly decreasing, set $K := \{i \in [n-1] : \lambda_i > \lambda_{i+1}\}$, and let $N\in\uu_n$. Let $L(t)$ evolve according to the gradient flow on $\Orbit_{\bflambda}$ with respect to $N$ in the K\"{a}hler metric, and let $V(t)\in\PFl{K}{n}(\mathbb{C})$ be the corresponding partial flag under the inverse map of \eqref{Fl_to_orbit_map}. Then
\begin{align}\label{gradient_flow_kahler_equation}
V(t) = \exp(t\ii N)V_0 \quad \text{ for all } t.
\end{align}
Letting $g(t)\in\U_n$ be any representative of $V(t)$, we have $L(t) = g(t)(\ii\Diag{\bflambda})g(t)^{-1}$. Explicitly, we can take $g_0\in\U_n$ representing $V_0$, and then take (cf.\ \cref{defn_iwasawa_projection})
\begin{align}\label{gradient_flow_kahler_equation_iwasawa}
g(t) = \Kterm(\exp(t\ii N)g_0) \quad \text{ for all } t.
\end{align}

\end{prop}

We emphasize that $N\in\uu_n$, and $\ii N$ is Hermitian. The assumption that $\bflambda$ is weakly decreasing is not important (until we consider the totally nonnegative part); only its multiplicities are relevant. Also, \eqref{gradient_flow_kahler_equation} should be regarded only as a flow on $\PFl{K}{n}(\mathbb{C})$, not on $\U_n$; in order to obtain a flow on $\U_n$, we must apply the Iwasawa decomposition, as in \eqref{gradient_flow_kahler_equation_iwasawa}.
\begin{rmk}\label{kahler_flow_decomposition}
There is an alternative way to describe the solution $L(t)$ in \cref{gradient_flow_kahler}. As in \eqref{projection_sum_formula}, write
\begin{align}\label{kahler_flow_decomposition_equation}
-\ii L(t) = \Big(\sum_{k\in K}(\lambda_k - \lambda_{k+1})P_k(t)\Big) + \lambda_n\I_n,
\end{align}
where $P_k(t)$ is the orthogonal projection onto the subspace spanned by the eigenvectors of $-\ii L(t)$ corresponding to the eigenvalues $\lambda_1, \dots, \lambda_k$. Explicitly, let $V(t) = (V_k(t))_{k\in K}\in\PFl{K}{n}(\mathbb{C})$ be as in \eqref{gradient_flow_kahler_equation}, with $V_0 = ((V_0)_k)_{k\in K}$, so that $V_k(t) = \exp(t\ii N)(V_0)_k$ for $k\in K$. Regarding elements of $\Gr_{k,n}(\mathbb{C})$ as $n\times k$ matrices, we have
\begin{multline}\label{kahler_flow_decomposition_formula}
P_k(t) = \Proj{V_k(t)} = V_k(t)(V_k(t)^*V_k(t))^{-1}V_k(t)^* \\
= \exp(t\ii N)(V_0)_k((V_0)_k^*\exp(2t\ii N)(V_0)_k)^{-1}(V_0)_k^*\exp(t\ii N).
\end{multline}

Note that \eqref{kahler_flow_decomposition_equation}, via \eqref{kahler_flow_decomposition_formula}, gives an explicit expression for $L(t)$. It does not require computing an Iwasawa decomposition; we only need to know $(V_0)_k$ for all $k\in K$. Also, by \cref{gradient_flow_kahler}, each $\ii P_k(t)$ evolves according to the gradient flow on $\Orbit_{\bfomega{k}}$ with respect to $N$ in the K\"{a}hler metric. But since $\Orbit_{\bfomega{k}} \cong \Gr_{k,n}(\mathbb{C})$, by \cref{grassmannian_metrics}, the K\"{a}hler, normal, and induced metrics coincide (because the dilation factors are $1$). For example, if we work instead in the normal metric, we will find (see \cref{gradient_flow_normal}) that
\begin{gather*}
\ii\dot{P}_k(t) = [\ii P_k(t), [\ii P_k(t),N]].
\end{gather*}

\end{rmk}

As a consequence of \cref{gradient_flow_kahler}, when considering flows which preserve positivity in the K\"{a}hler metric, we need only work with \eqref{gradient_flow_kahler_equation}:
\begin{cor}\label{kahler_flow_positivity_preserving}
Let $K\subseteq [n-1]$, and let $N\in\uu_n$. Then for all weakly decreasing $\bflambda\in\mathbb{R}^n$ with $\{i \in [n-1] : \lambda_i > \lambda_{i+1}\} = K$, the gradient flow on $\Orbit_{\bflambda}$ with respect to $N$ in the K\"{a}hler metric weakly preserves positivity if and only if the flow \eqref{gradient_flow_kahler_equation} on $\PFl{K}{n}(\mathbb{C})$ weakly  preserves positivity. If so, then the gradient flow on $\Orbit_{\bflambda'}$ with respect to $N$ in the K\"{a}hler metric also weakly preserves positivity, for all $K'\subseteq K$ and weakly decreasing $\bflambda'$ with $\{i \in [n-1] : \lambda'_i > \lambda'_{i+1}\} = K'$. The same statements hold with ``weakly'' replaced by ``strictly''.
\end{cor}

\begin{proof}
This follows from \cref{gradient_flow_kahler} and \eqref{defn_tnn_Fl_surjections}.
\end{proof}

\begin{eg}\label{running_example_kahler}
Let us consider an example in the case $n=2$. Set
$$
L_0 := \ii\begin{bmatrix}a & b \\ b & -a\end{bmatrix} \quad \text{ and } \quad \ii N := \begin{bmatrix}p & q \\ q & -p\end{bmatrix},
$$
where $a,b,p,q\in\mathbb{R}$ such that $a$ or $b$ is nonzero. We assume that $b \ge 0$. We have $L_0\in\Orbit_{\bflambda}$, where
$$
\lambda_1 := \sqrt{a^2+b^2} \quad \text{ and } \quad \lambda_2 := -\sqrt{a^2+b^2}.
$$

Let $L(t)\in\Orbit_{\bflambda}$ evolve according to the gradient flow with respect to $N$ in the K\"{a}hler metric. We have
$$
L_0 = g_0(\ii\Diag{\bflambda})g_0^{-1}, \quad \text{ where } g_0 := \frac{1}{\sqrt{2\lambda_1}}\begin{bmatrix}
\sqrt{\lambda_1 + a} & -\sqrt{\lambda_1 - a} \\[4pt]
\sqrt{\lambda_1 - a} & \sqrt{\lambda_1 + a}
\end{bmatrix}\in\U_2.
$$
By \cref{gradient_flow_kahler}, we have the explicit solution
$$
L(t) = g(t)(\ii\Diag{\bflambda})g(t)^{-1}, \quad \text{ where } g(t) = \Kterm(\exp(t\ii N)g_0).
$$
However, this involves computing a matrix exponential and an Iwasawa decomposition, which is already cumbersome when $n=2$. Instead, for the purposes of illustration as well as comparison with the normal and induced metrics, let us calculate $\dot{L}(0)$.

For the remainder of this example we write `$\equiv$' to mean equality up to $O(t^2)$ as $t\to 0$. Let $V_0\in\Fl_2(\mathbb{C})$ be the flag represented by $g_0$, and let $V(t)$ be defined by \eqref{gradient_flow_kahler_equation}. Then
\begin{multline*}
V(t) = \exp(t\ii N)V_0 \equiv (\I_2 + t\ii N)V_0 \\
= \frac{1}{\sqrt{2\lambda_1}}\begin{bmatrix}
\sqrt{\lambda_1 + a} + t(p\sqrt{\lambda_1 + a} + q\sqrt{\lambda_1 - a}) & -\sqrt{\lambda_1 - a} + t(-p\sqrt{\lambda_1 - a} + q\sqrt{\lambda_1 + a}) \\[6pt]
\sqrt{\lambda_1 - a} + t(q\sqrt{\lambda_1 + a} - p\sqrt{\lambda_1 - a}) & \sqrt{\lambda_1 + a} + t(-q\sqrt{\lambda_1 - a} - p\sqrt{\lambda_1 + a})
\end{bmatrix}.
\end{multline*}
Applying the Iwasawa decomposition gives
$$
g(t) \equiv \frac{1}{\sqrt{2\lambda_1}}\begin{bmatrix}
\sqrt{\lambda_1 + a} - t(\frac{aq-bp}{\lambda_1})\sqrt{\lambda_1 - a} & -\sqrt{\lambda_1 - a} - t(\frac{aq-bp}{\lambda_1})\sqrt{\lambda_1 + a} \\[8pt]
\sqrt{\lambda_1 - a} + t(\frac{aq-bp}{\lambda_1})\sqrt{\lambda_1 + a} & \sqrt{\lambda_1 + a} - t(\frac{aq-bp}{\lambda_1})\sqrt{\lambda_1 - a}
\end{bmatrix},
$$
and so
$$
L(t) = g(t)(\ii\Diag{\bflambda})\transpose{g(t)} \equiv \ii\begin{bmatrix}
a - t(\frac{aq-bp}{\lambda_1})2b & b + t(\frac{aq-bp}{\lambda_1})2a \\[6pt]
b + t(\frac{aq-bp}{\lambda_1})2a & -a + t(\frac{aq-bp}{\lambda_1})2b
\end{bmatrix}.
$$
Since $L(t) \equiv L_0 + t\dot{L}(0)$, we obtain
\begin{gather*}
\dot{L}(0) = \frac{2(aq-bp)}{\lambda_1}\ii\begin{bmatrix}-b & a \\ a & b\end{bmatrix}.\qedhere
\end{gather*}

\end{eg}

\begin{eg}\label{eg_kahler_diagonal}
We consider the same setup as in \cref{running_example_kahler}, but take $\ii N$ to be diagonal:
$$
\ii N := \begin{bmatrix}p & 0 \\[2pt] 0 & -p\end{bmatrix}.
$$
Let us calculate the explicit solution $L(t) = g(t)(\ii\Diag{\bflambda})g(t)^{-1}$ to the gradient flow in the K\"{a}hler metric. We have
$$
V(t) = \exp(t\ii N)V_0 = \frac{1}{\sqrt{2\lambda_1}}\begin{bmatrix}
e^{tp}\sqrt{\lambda_1 + a} & -e^{tp}\sqrt{\lambda_1 - a} \\[4pt]
e^{-tp}\sqrt{\lambda_1 - a} & e^{-tp}\sqrt{\lambda_1 + a}
\end{bmatrix}.
$$
Applying the Iwasawa decomposition gives
$$
g(t) = \frac{1}{\sqrt{e^{2tp}(\lambda_1+a) + e^{-2pt}(\lambda_1-a)}}\begin{bmatrix}
e^{tp}\sqrt{\lambda_1 + a} & -e^{-tp}\sqrt{\lambda_1 - a} \\[4pt]
e^{-tp}\sqrt{\lambda_1 - a} & e^{tp}\sqrt{\lambda_1 + a}
\end{bmatrix} \in \U_2,
$$
and so $L(t)$ equals
$$
\frac{\lambda_1}{e^{2tp}(\lambda_1+a) + e^{-2tp}(\lambda_1-a)}\ii\begin{bmatrix}
e^{2tp}(\lambda_1+a) - e^{-2tp}(\lambda_1-a) & 2b \\[4pt]
2b & -e^{2tp}(\lambda_1+a) + e^{-2tp}(\lambda_1-a)
\end{bmatrix}.
$$

We can use the formula above to compute the limits of $L(t)$ as $t\to\pm\infty$. If $b=0$ (i.e.\ $\lambda_1 = \pm a$) or $p = 0$, then $L(t)$ is constant. Otherwise, we have $\lambda_1\pm a > 0$. If $p > 0$, we obtain
$$
\lim_{t\to\infty}L(t) = \ii\begin{bmatrix}\lambda_1 & 0 \\ 0 & -\lambda_1\end{bmatrix} \quad \text{ and } \quad \lim_{t\to -\infty}L(t) = \ii\begin{bmatrix}-\lambda_1 & 0 \\ 0 & \lambda_1\end{bmatrix}.
$$
If $p < 0$, the limits are exchanged.
\end{eg}

We recall from \cref{defn_tnn_intro} that $\gl_n^{\ge 0}$ and $\gl_n^{>0}$ are the infinitesimal parts of $\GL_n^{\ge 0}$ and $\GL_n^{>0}$, respectively. The following theorem is an analogue of this statement, where instead of considering all minors of an $n\times n$ matrix, we only consider minors of a fixed order $k$. It will be the key to classifying positivity-preserving gradient flows on $\Orbit_{\bflambda}$, when $\Orbit_{\bflambda}\cong\Gr_{k,n}(\mathbb{C})$. We will only apply it when $M$ is a symmetric matrix, but for completeness we state it for general real $M$. The result and its proof are natural extensions of \cite[Section 3.2]{galashin_karp_lam}, which considered the case of a specific matrix $M$ (the {\itshape cyclic shift matrix}).
\begin{thm}\label{infinitesimal_k}
Let $1 \le k \le n-1$, and let $M\in\gl_n(\mathbb{R})$. Then the following are equivalent:
\begin{enumerate}[label=(\roman*), leftmargin=*, itemsep=4pt]
\item\label{infinitesimal_k_nonnegative_entries} if $k=1$: we have $M_{i,j} \ge 0$ for all $i\neq j$ in $[n]$; \\[2pt]
if $k=n-1$: we have $(-1)^{i+j-1}M_{i,j} \ge 0$ for all $i\neq j$ in $[n]$; \\[2pt]
if $2 \le k \le n-2$: we have
$$
M_{1,2}, M_{2,1}, M_{2,3}, M_{3,2}, \dots, M_{n-1,n}, M_{n,n-1}, (-1)^{k-1}M_{n,1}, (-1)^{k-1}M_{1,n} \ge 0
$$
and
$$
M_{i,j} = 0 \quad \text{ for all $i,j\in [n]$ such that } i-j \not\equiv -1, 0, 1\modulo{n};
$$
\item\label{infinitesimal_k_nonnegative_minors} all $k\times k$ minors of $\exp(tM)$ are nonnegative, for all $t \ge 0$; and
\item\label{infinitesimal_k_nonnegative_preserving} $\exp(tM)V \in \Gr_{k,n}^{\ge 0}$ for all $V\in\Gr_{k,n}^{\ge 0}$ and $t \ge 0$.
\end{enumerate}

Now let $D$ be the directed graph on the vertex set $[n]$, where $i\to j$ (for $i\neq j$) is an edge of $D$ if and only if $M_{i,j}\neq 0$. Then analogously, the following are equivalent:
\begin{enumerate}[resume, label=(\roman*), leftmargin=*, itemsep=4pt]
\item\label{infinitesimal_k_positive_entries} condition \ref{infinitesimal_k_nonnegative_entries} holds, and in addition, $D$ is strongly connected (i.e.\ for any $i\neq j$ in $[n]$, there exists a directed path from $i$ to $j$);
\item\label{infinitesimal_k_positive_minors} all $k\times k$ minors of $\exp(tM)$ are positive, for all $t>0$; and
\item\label{infinitesimal_k_positive_preserving} $\exp(tM)V \in \Gr_{k,n}^{>0}$ for all $V\in\Gr_{k,n}^{\ge 0}$ and $t>0$.
\end{enumerate}
\end{thm}

We observe that for $2 \le k \le n-2$, conditions \ref{infinitesimal_k_nonnegative_entries} and \ref{infinitesimal_k_positive_entries} above depend only on the parity of $k$. Therefore the other four conditions also only depend on the parity of $k$, which is far from obvious. We also remark that the condition that $D$ is strongly connected arises naturally in the Perron--Frobenius theory of nonnegative matrices and the theory of Markov chains (see e.g.\ \cite[Chapter XIII]{gantmacher59}), where it is known as {\itshape irreducibility}.
\begin{proof}
\ref{infinitesimal_k_nonnegative_entries} $\Rightarrow$ \ref{infinitesimal_k_nonnegative_minors}: We adapt an argument of Br\"{a}nd\'{e}n \cite[Proposition 2.3]{branden}. Note that by \cref{trotter} and the Cauchy--Binet identity \eqref{cauchy-binet}, the sum of two matrices satisfying \ref{infinitesimal_k_nonnegative_minors} also satisfies \ref{infinitesimal_k_nonnegative_minors}. Therefore it suffices to consider the case when $M$ has a single nonzero entry, say entry $(i,j)$. If $i=j$, then $\exp(tM) = \Diag{1, \dots, 1, e^{tM_i}, 1, \dots, 1}$, and \ref{infinitesimal_k_nonnegative_minors} holds. Otherwise, we have $\exp(tM) = \I_n + tM$, and so every $k\times k$ minor of $\exp(tM)$ equals either $1$ or $(-1)^{l-1}M_{i,j}t$, for some $l\in [k]$ satisfying $l \le |i-j|$ and $k-l \le n-1-|i-j|$. Therefore if \ref{infinitesimal_k_nonnegative_entries} holds, then so does \ref{infinitesimal_k_nonnegative_minors}.

\ref{infinitesimal_k_nonnegative_minors} $\Rightarrow$ \ref{infinitesimal_k_nonnegative_entries}: Suppose that \ref{infinitesimal_k_nonnegative_minors} holds. Note that for all $I,J\in\binom{[n]}{k}$ with $I\neq J$, we have $\Delta_{I,J}(\exp(tM)) = 0$ at $t=0$. Therefore, by \cref{multilinearity_remark}, we have
$$
\textstyle\frac{d}{dt}\eval{t=0}\Delta_{I,J}(\exp(tM)) = \displaystyle\sum_{j\in J}\Delta_{I,J}(\I_n \text{ with column $j$ replaced by column $j$ of } M) \ge 0.
$$
Let us take $I := K\cup\{i\}$ and $J := K\cup\{j\}$, where $i,j\in [n]$ with $i\neq j$, and $K\in\binom{[n]\setminus\{i,j\}}{k-1}$. Then we get
\begin{align}\label{inv_inequalities}
(-1)^{\inv{i}{K} + \inv{j}{K}}M_{i,j} \ge 0,
\end{align}
where $\inv{i'}{K}$ denotes the number of $j'\in K$ with $i' > j'$. We can then verify that these inequalities reduce to those in \ref{infinitesimal_k_nonnegative_entries}.

\ref{infinitesimal_k_nonnegative_minors} $\Rightarrow$ \ref{infinitesimal_k_nonnegative_preserving} and \ref{infinitesimal_k_positive_minors} $\Rightarrow$ \ref{infinitesimal_k_positive_preserving}: Let $V\in\Gr_{k,n}^{\ge 0}$. By \cref{tnn_Fl}\ref{tnn_Fl_tnn}, we can regard $V$ as an $n\times k$ matrix whose $k\times k$ minors are nonnegative, where at least one of these minors is positive. Therefore the implications follow from the Cauchy--Binet identity \eqref{cauchy-binet} and \cref{tnn_Fl_converse}.

\ref{infinitesimal_k_nonnegative_preserving} $\Rightarrow$ \ref{infinitesimal_k_nonnegative_minors} and \ref{infinitesimal_k_positive_preserving} $\Rightarrow$ \ref{infinitesimal_k_positive_minors}: For $J\in\binom{[n]}{k}$, let $V_J$ be the $n\times k$ matrix which has an identity matrix in rows $J$ and zeros elsewhere. We regard $V_J$ as the element of $\Gr_{k,n}^{\ge 0}$ with $\Delta_I(V_J) = \delta_{I,J}$ for all $I\in\binom{[n]}{k}$. Then for all $I\in\binom{[n]}{k}$ and $t\in\mathbb{R}$, we have
$$
\Delta_{I,J}(\exp(tM)) = \Delta_I(\exp(tM)V_J)
$$
(however, we caution that the Pl\"{u}cker coordinates on the right-hand side are only well-defined modulo a global scalar).

Now suppose that \ref{infinitesimal_k_nonnegative_preserving} holds, and let $J\in\binom{[n]}{k}$. Then for every $t\ge 0$, either 
$$
\Delta_{I,J}(\exp(tM)) \ge 0 \text{ for all } I\in\textstyle\binom{[n]}{k} \quad \text{ or } \quad \Delta_{I,J}(\exp(tM)) \le 0 \text{ for all } I\in\textstyle\binom{[n]}{k}.
$$
In order to prove \ref{infinitesimal_k_nonnegative_minors}, it suffices to rule out the latter case. Note that the columns $J$ of $\exp(tM)$ are linearly independent, so $\Delta_{I,J}(\exp(tM)) \neq 0$ for some $I\in\binom{[n]}{k}$. Hence in either case, we have
$$
\sum_{I\in\binom{[n]}{k}}\Delta_{I,J}(\exp(tM)) \neq 0 \quad \text{ for all } t \ge 0.
$$
Since the left-hand side is positive when $t=0$, by continuity it is positive for all $t\ge 0$. This proves \ref{infinitesimal_k_nonnegative_minors}. We can similarly prove \ref{infinitesimal_k_positive_preserving} $\Rightarrow$ \ref{infinitesimal_k_positive_minors}.

\ref{infinitesimal_k_positive_entries} $\Rightarrow$ \ref{infinitesimal_k_positive_minors}: We adapt the proof of \cite[Lemma 3.5]{galashin_karp_lam}. Suppose that \ref{infinitesimal_k_positive_entries} holds. Form the directed graph $\widehat{D}$ on the vertex set $\binom{[n]}{k}$, where $I\to J$ (for $I\neq J$) is an edge of $\widehat{D}$ if and only if there exists an edge $i\to j$ of $D$ such that $J = (I\setminus\{i\})\cup\{j\}$. We claim that $\widehat{D}$ is strongly connected. Indeed, it suffices to show that given $I\in\binom{[n]}{k}$, $i\in I$, and $j\in [n]\setminus I$, there exists a directed path from $I$ to $(I\setminus\{i\})\cup\{j\}$. We prove this by induction on the length $l\ge 1$ of the shortest directed path from $i$ to $j$ in $D$ (which exists since $D$ is strongly connected), with no base case. Given $l\ge 1$, suppose that the result holds for strictly smaller values of $l$. Take a directed path $i = j_0 \to \cdots \to j_l = j$ from $i$ to $j$, and let $0 \le m \le l-1$ be maximal such that $j_0, \dots, j_m\in I$. Then $j_{m+1}\notin I$, so we have the directed path in $\widehat{D}$
$$
I \to (I\setminus\{j_m\})\cup\{j_{m+1}\} \to (I\setminus\{j_{m-1}\})\cup\{j_{m+1}\} \to \cdots \to (I\setminus\{i\})\cup\{j_{m+1}\}.
$$
If $m+1 = l$, we are done. Otherwise, by the induction hypothesis, there exists a directed path from $(I\setminus\{i\})\cup\{j_{m+1}\}$ to $(I\setminus\{i\})\cup\{j\}$. Therefore we get a directed path from $I$ to $(I\setminus\{i\})\cup\{j\}$, completing the induction.

Since \ref{infinitesimal_k_positive_entries} $\Rightarrow$ \ref{infinitesimal_k_nonnegative_entries} $\Rightarrow$ \ref{infinitesimal_k_nonnegative_minors}, we know that all $k\times k$ minors of $\exp(tM)$ are nonnegative for all $t > 0$; it remains to show that no such minor is zero. Suppose otherwise that there exist $s > 0$ and $I,J_0\in\binom{[n]}{k}$ such that $\Delta_{I,J_0}\exp(sM) = 0$. Since the rows $I$ of $\exp(sM)$ are linearly independent, there exists $J_1\in\binom{[n]}{k}$ with $\Delta_{I,J_1}(\exp(sM)) \neq 0$. Since $\widehat{D}$ is strongly connected, there is a directed path from $J_1$ to $J_0$; it passes through an edge $J'\to J$ with $\Delta_{I,J'}(\exp(sM)) \neq 0$ and $\Delta_{I,J}(\exp(sM)) = 0$. We may write $J = (J'\setminus\{i'\})\cup\{j'\}$, where $M_{i',j'}\neq 0$.

Recall that all $k\times k$ minors of $\exp((s+t)M)$ are nonnegative for $t > -s$. In particular, $\Delta_{I,J}(\exp((s+t)M))$ equals $0$ at $t=0$, and it is nonnegative near $t=0$. Therefore
$$
\textstyle\frac{d}{dt}\eval{t=0}\Delta_{I,J}(\exp((s+t)M)) = 0.
$$
By \cref{multilinearity_remark} and multilinearity of the determinant, the left-hand side above equals
\begin{align*}
&\sum_{j\in J}\Delta_{I,J}(\exp(sM) \text{ with column $j$ replaced by column $j$ of } \exp(sM)M) \\
=& \sum_{j\in J}\sum_{i\notin J}(-1)^{\inv{i}{J\setminus\{j\}}+\inv{j}{J\setminus\{j\}}}M_{i,j}\Delta_{I,(J\setminus\{j\})\cup\{i\}}(\exp(sM)).
\end{align*}
By assumption, each summand above is nonnegative (cf.\ \eqref{inv_inequalities}), and the summand with $j = j'$ and $i = i'$ is nonzero. Therefore the sum is nonzero, a contradiction.

\ref{infinitesimal_k_positive_minors} $\Rightarrow$ \ref{infinitesimal_k_positive_entries}: Suppose that \ref{infinitesimal_k_positive_minors} holds. Since \ref{infinitesimal_k_positive_minors} $\Rightarrow$ \ref{infinitesimal_k_nonnegative_minors} $\Rightarrow$ \ref{infinitesimal_k_nonnegative_entries}, it remains to show that $D$ is strongly connected. Suppose otherwise, so that there exist distinct $i_0,j_0\in [n]$ such that there is no directed path from $i_0$ to $j_0$. Let $I_0\subseteq [n]$ denote the set of $i\in [n]$ (including $i_0$) such that there exists a directed path from $i_0$ to $i$. Then there are no edges from $I_0$ to $[n]\setminus I_0$, and $j_0\notin I_0$. From the expression $\exp(tM) = \lim_{m\to\infty}(\I_n + \frac{t}{m}M)^m$, we see that $\exp(tM)_{I_0,[n]\setminus I_0} = 0$. Taking any $I,J\in\binom{[n]}{k}$ such that $|I\cap I_0|$ and $|J\cap ([n]\setminus I_0)|$ are maximized, we have $\Delta_{I,J}(\exp(tM)) = 0$, a contradiction.
\end{proof}

\begin{cor}\label{positivity_preserving_kahler_grassmannian}
Let $\bflambda\in\mathbb{R}^n$ with $\lambda_1 = \cdots = \lambda_k > \lambda_{k+1} = \cdots = \lambda_n$, so that $\Orbit_{\bflambda}\cong\Gr_{k,n}(\mathbb{C})$, and let $N\in\uu_n$. Then the gradient flow on $\Orbit_{\bflambda}$ with respect to $N$ in the K\"{a}hler metric (equivalently, by \cref{grassmannian_metrics}, in the normal or induced metrics) weakly preserves positivity if and only if the following condition holds, depending on the value of $k$:
\begin{enumerate}[label=(\roman*), leftmargin=*, itemsep=2pt]
\item\label{positivity_preserving_kahler_grassmannian_first} $k=1$:
$$
\ii N_{i,j} \ge 0 \quad \text{ for all $i\neq j$ in } [n];
$$
\item\label{positivity_preserving_kahler_grassmannian_last} $k=n-1$:
$$
(-1)^{i+j-1}\ii N_{i,j} \ge 0 \quad \text{ for all $i\neq j$ in } [n];
$$
\item\label{positivity_preserving_kahler_grassmannian_general} $2 \le k \le n-2$:
$$
\ii N_{1,2},\, \ii N_{2,3},\, \dots,\, \ii N_{n-1,n},\, (-1)^{k-1}\ii N_{n,1} \ge 0,
$$
and
$$
N_{i,j} = 0 \quad \text{ for all $i,j\in [n]$ such that } i-j \not\equiv -1, 0, 1\modulo{n}.
$$
\end{enumerate}
Moreover, let $\Gamma$ be the undirected graph on the vertex set $[n]$, where $\{i,j\}$ is an edge of $\Gamma$ if and only if $N_{i,j}\neq 0$. Then the gradient flow strictly preserves positivity if and only if, additionally, $\Gamma$ is connected. (For $2 \le k \le n-2$, this means that at least $n-1$ of the $n$ inequalities in the first line of \ref{positivity_preserving_kahler_grassmannian_general} hold strictly.)
\end{cor}

For example, for the choice of $N$ in \cref{running_example_kahler}, we are in both the cases \ref{positivity_preserving_kahler_grassmannian_first} and \ref{positivity_preserving_kahler_grassmannian_last} above. The gradient flow with respect to $N$ in the K\"{a}hler metric weakly preserves positivity if and only if $q \ge 0$, and it strictly preserves positivity if and only if $q > 0$.

\begin{proof}
This follows from \cref{kahler_flow_positivity_preserving} and \cref{infinitesimal_k} (with $M = \ii N$).
\end{proof}

We now consider the case when $\Orbit_{\bflambda}$ is not isomorphic to a Grassmannian, i.e., $\bflambda$ has at least three distinct entries. Our analysis will be based on \cref{infinitesimal_k} along with the following two technical results.
\begin{lem}\label{corner_entry}
Let $K\subseteq [n-1]$ such that $|K| \ge 2$, and suppose that $M\in\gl_n(\mathbb{R})$ such that
$$
\exp(tM)V \in \PFl{K}{n}^{\ge 0} \quad \text{ for all } V\in\PFl{K}{n}^{\ge 0} \text{ and } t \ge 0.
$$
Then $M_{n,1} = M_{1,n} = 0$.
\end{lem}

\begin{proof}
By symmetry (specifically, using the map $\rev$ from \cref{defn_rev}), it suffices to show that $M_{n,1} = 0$. Take distinct elements $k < l$ of $K$, and let $w\in\mathfrak{S}_n$ be the cycle $(1 \hspace*{8pt} 2 \hspace*{4pt} \cdots \hspace*{4pt} k)$, so that
$$
\mathring{w} = \begin{bmatrix}
0 & (-1)^{k-1} & 0 \\
\I_{k-1} & 0 & 0 \\
0 & 0 & \I_{n-k}
\end{bmatrix} \in \U_n^{\ge 0}.
$$
Let $V\in\PFl{K}{n}^{\ge 0}$ be represented by $\mathring{w}\in\U_n^{\ge 0}$. Recall from \cref{torus_action} that $\H_n^{>0}$ acts on $\PFl{K}{n}^{\ge 0}$. In particular,
$$
W := \lim_{t\to 0,\, t > 0}\Diag{1, \dots, 1, t^{-1}}\exp(tM)V
$$
lies in $\PFl{K}{n}^{\ge 0}$, if the limit exists.

To calculate the limit, we replace $V$ with $\mathring{w}$ and work in the space of matrices, ignoring the last column. Since $\exp(tM) = \I_n + tM + O(t^2)$ as $t \to 0$, and row $n$ of $\mathring{w}_{[n],[n-1]}$ is zero, we obtain
\begin{multline*}
\lim_{t\to 0,\, t > 0}\Diag{1, \dots, 1, t^{-1}}\exp(tM)\mathring{w}_{[n],[n-1]} = \mathring{w}_{[n],[n-1]} + \Diag{0,\dots,0,1}M\mathring{w}_{[n],[n-1]} \\
= \begin{bmatrix}
0 & (-1)^{k-1} & 0 \\
\I_{k-1} & 0 & 0 \\
0 & 0 & \I_{n-k-1} \\
\ast & (-1)^{k-1}M_{n,1} & \ast
\end{bmatrix}.
\end{multline*}
(The entries $\ast$ will turn out to be unimportant.) This shows that the limit defining $W$ exists. Since $W\in\PFl{K}{n}^{\ge 0}$, it extends to a complete flag $(W_1, \dots, W_{n-1}) \in \Fl_n^{\ge 0}$. Observe that $e_1 + M_{n,1}e_n\in W_k$, so by \cref{tnn_counterexample}, we have $e_1\in W_{k+1} \subseteq W_l$. Because $W_l$ is spanned by the first $l$ columns of the matrix above, we see that $M_{n,1} = 0$.
\end{proof}

\begin{lem}\label{kahler_special_case}
Let $K := \{1, n-1\}$, and suppose that $M\in\gl_n(\mathbb{R})$ such that
$$
\exp(tM)V \in \PFl{K}{n}^{\ge 0} \quad \text{ for all } V\in\PFl{K}{n}^{\ge 0} \text{ and } t \ge 0.
$$
Then $M_{i,j} = 0$ for all $i,j\in [n]$ such that $|i-j| \ge 2$.
\end{lem}

\begin{proof}
We use a similar argument as in the proof of \cref{corner_entry}. By symmetry, it suffices to show that $M_{i,j} = 0$ for $i,j\in [n]$ with $i-j \ge 2$. Let
$$
w := (1 \hspace*{8pt} 2 \hspace*{4pt} \cdots \hspace*{4pt} j)^{-1}(i \hspace*{8pt} i+1 \hspace*{4pt} \cdots \hspace*{4pt} n)\in\mathfrak{S}_n,
$$
so that
$$
\mathring{w} = \begin{bmatrix}
0 & -\I_{j-1} & 0 & 0 & 0 \\
1 & 0 & 0 & 0 & 0 \\
0 & 0 & \I_{i-j-1} & 0 & 0 \\
0 & 0 & 0 & 0 & (-1)^{n-i} \\
0 & 0 & 0 & \I_{n-i} & 0
\end{bmatrix} \in \U_n^{\ge 0}.
$$
Let $V\in\PFl{K}{n}^{\ge 0}$ be represented by $\mathring{w}\in\U_n^{\ge 0}$. For $t > 0$, let $h(t)\in\H_n^{>0}$ be obtained from $\I_n$ by replacing the $(i,i)$-entry with $t^{-1}$. Assuming the limit exists, define
$$
W := \lim_{t\to 0,\, t > 0}h(t)\exp(tM)V \in\PFl{K}{n}^{\ge 0}.
$$

To calculate the limit, we replace $V$ with $\mathring{w}$ and work in the space of matrices, ignoring the last column. Since $\exp(tM) = \I_n + tM + O(t^2)$ as $t \to 0$, and row $i$ of $\mathring{w}_{[n],[n-1]}$ is zero, we obtain
$$
\lim_{t\to 0,\, t > 0}h(t)\exp(tM)\mathring{w}_{[n],[n-1]} = \begin{bmatrix}
0 & -\I_{j-1} & 0 & 0 \\
1 & 0 & 0 & 0 \\
0 & 0 & \I_{i-j-1} & 0 \\
M_{i,j} & \ast & \ast & \ast \\
0 & 0 & 0 & \I_{n-i}
\end{bmatrix}.
$$
(The entries $\ast$ will turn out to be unimportant.) This shows that the limit defining $W$ exists. Since $W\in\PFl{K}{n}^{\ge 0}$, it extends to a complete flag $(W_1, \dots, W_{n-1}) \in \Fl_n^{\ge 0}$.

Let $\mathbb{C}^{[j,i]}$ denote the span of $e_k$ for $j \le k \le i$, which has dimension at least $3$. For $1 \le k \le n-1$, let $d_k$ denote the dimension of $W_k\cap\mathbb{C}^{[j,i]}$, so that $W_k\cap\mathbb{C}^{[j,i]}\in\Gr_{d_k,i-j+1}^{\ge 0}$. Observe that the sequence $d_1, \dots, d_{n-1}$ increases by either $0$ or $1$ at each step. Since $d_2 \le 2 \le d_{n-1}$, we have $d_k = 2$ for some $2 \le k \le n-1$. Applying \cref{tnn_counterexample} to $W_1\cap\mathbb{C}^{[j,i]}$ and $W_k\cap\mathbb{C}^{[j,i]}$, we get that $e_j\in W_k\cap\mathbb{C}^{[j,i]} \subseteq W_{n-1}$. Because $W_{n-1}$ is spanned by the columns of the matrix above, we see that $M_{i,j} = 0$.
\end{proof}

We have the following analogue of \cref{infinitesimal_k} for an arbitrary partial flag variety which is not a Grassmannian:
\begin{thm}\label{infinitesimal_full}
Let $K\subseteq [n-1]$ such that $|K| \ge 2$, and let $M\in\gl_n(\mathbb{R})$.
\begin{enumerate}[label=(\roman*), leftmargin=*, itemsep=2pt]
\item\label{infinitesimal_full_nonnegative} We have $M\in\gl_n^{\ge 0}$ if and only if
\begin{align}\label{infinitesimal_full_nonnegative_equation}
\exp(tM)V \in \PFl{K}{n}^{\ge 0} \quad \text{ for all } V\in\PFl{K}{n}^{\ge 0} \text{ and } t \ge 0.
\end{align}
\item\label{infinitesimal_full_positive} We have $M\in\gl_n^{>0}$ if and only if
\begin{align}\label{infinitesimal_full_positive_equation}
\exp(tM)V \in \PFl{K}{n}^{>0} \quad \text{ for all } V\in\PFl{K}{n}^{\ge 0} \text{ and } t \ge 0.
\end{align}
\end{enumerate}
\end{thm}

\begin{proof}
The forward directions of parts \ref{infinitesimal_full_nonnegative} and \ref{infinitesimal_full_positive} follow from \cref{matrices_acting_on_flags}. To prove the reverse directions, suppose that \eqref{infinitesimal_full_nonnegative_equation} holds. Then for every $k\in K$, \eqref{infinitesimal_full_nonnegative_equation} also holds with $K$ replaced by $\{k\}$, so the conditions of \cref{infinitesimal_k}\ref{infinitesimal_k_nonnegative_entries} hold. These conditions, along with \cref{corner_entry} and \cref{kahler_special_case}, imply that $M\in\gl_n^{\ge 0}$. This proves the reverse direction of part \ref{infinitesimal_full_nonnegative}. Now suppose that in addition, \eqref{infinitesimal_full_positive_equation} holds. Then taking any $k\in K$, we have that \eqref{infinitesimal_full_positive_equation} holds with $K$ replaced by $\{k\}$, so the condition of \cref{infinitesimal_k}\ref{infinitesimal_k_positive_entries} holds. Since $M\in\gl_n^{\ge 0}$, this implies that $M\in\gl_n^{>0}$. This proves the reverse direction of part \ref{infinitesimal_full_positive}.
\end{proof}

\begin{cor}\label{positivity_preserving_kahler_full}
Let $\bflambda\in\mathbb{R}^n$ be weakly decreasing with at least three distinct entries (so that $\Orbit_{\bflambda}$ is {\itshape not} isomorphic to a Grassmannian), and let $N\in\uu_n$. Then the gradient flow on $\Orbit_{\bflambda}$ with respect to $N$ in the K\"{a}hler metric weakly preserves positivity if and only if $\ii N\in\gl_n^{\ge 0}$, and it strictly preserves positivity if and only if $\ii N\in\gl_n^{>0}$.
\end{cor}

\begin{proof}
This follows from \cref{kahler_flow_positivity_preserving} and \cref{infinitesimal_full} (with $M = \ii N$).
\end{proof}

\begin{rmk}\label{plucker_gradient_flows}
Recall the notion Pl\"{u}cker positivity introduced in \cref{defn_plucker_positive}. In analogy with \cref{defn_positivity_preserving}, for any weakly decreasing $\bflambda\in\mathbb{R}^n$, we can consider flows on $\Orbit_{\bflambda}$ which weakly or strictly preserve Pl\"{u}cker positivity. Note that a flow which preserves Pl\"{u}cker positivity does not necessarily preserve positivity, and vice-versa. However, we expect the two notions to be closely related. Here we discuss the case of the gradient flow with respect to $N\in\uu_n$ in the K\"{a}hler metric, and consider weak preservation (we have an entirely analogous analysis for strict preservation).

Let $K := \{i\in [n-1] : \lambda_i > \lambda_{i+1}\}$. For simplicity, we assume that $K\neq\{1,n-1\}$. By \cref{plucker_positive_projection} and \cref{gradient_flow_kahler} (cf.\ \cref{kahler_flow_decomposition}), the gradient flow on $\Orbit_{\bflambda}$ with respect to $N$ weakly preserves Pl\"{u}cker positivity if and only if the gradient flow on $\Orbit_{\bfomega{k}}$ with respect to $N$ weakly preserves Pl\"{u}cker positivity, for all $k\in K$; and this holds if and only if each $k\in K$ satisfies the condition of \cref{positivity_preserving_kahler_grassmannian}. Comparing this with \cref{positivity_preserving_kahler_full}, we see that if the gradient flow on $\Orbit_{\bflambda}$ weakly preserves positivity, then it weakly preserves Pl\"{u}cker positivity. The converse holds for all $N$ if and only if $K$ is a singleton or contains both an even and an odd number. Indeed, suppose that $|K|\ge 2$ and that all elements of $K$ have the same parity. Then in order for positivity to be weakly preserved with respect to $N$, we must have $\ii N_{n,1} = \ii N_{1,n} = 0$. However, Pl\"{u}cker positivity is preserved as long as $\ii N_{n,1} = \ii N_{1,n}$ has fixed sign (depending on the parity of the elements of $K$).
\end{rmk}

\subsection{The normal metric}\label{sec_gradient_normal}
In this subsection, we show that when $\Orbit_{\bflambda}$ is isomorphic to the complete flag variety $\Fl_n(\mathbb{C})$ with $n\ge 3$, the only gradient flow in the normal metric which weakly preserves positivity is the constant flow (\cref{positivity_preserving_normal}). This is in stark contrast to the case that $\Orbit_{\bflambda}$ is isomorphic to the Grassmannian $\Gr_{k,n}(\mathbb{C})$, whence the normal metric coincides with the K\"{a}hler metric up to dilation (see \cref{grassmannian_metrics}), and the gradient flows which preserve positivity are classified by \cref{positivity_preserving_kahler_grassmannian}. We do not consider here the remaining cases (i.e.\ when $\Orbit_{\bflambda}$ is isomorphic to neither $\Fl_n(\mathbb{C})$ nor $\Gr_{k,n}(\mathbb{C})$); we leave this to future work.

We use an explicit description of the gradient flow as a {\itshape double-bracket flow}, which was first observed by Brockett \cite{brockett91}. It can be verified from \cref{gradient_computation} (we omit the derivation).
\begin{prop}[{Brockett \cite{brockett91}; Bloch, Brockett, and Ratiu \cite[Proposition 1.4]{bloch_brockett_ratiu92}}]\label{gradient_flow_normal}
Let $\bflambda\in\mathbb{R}^n$ and let $N\in\uu_n$. Then the gradient flow on $\Orbit_{\bflambda}$ with respect to $N$ in the normal metric is given by
\begin{align}\label{gradient_flow_normal_equation}
\dot{L}(t) = [L(t),[L(t),N]].
\end{align}

\end{prop}

\begin{eg}\label{running_example_normal}
Let us consider the same setup as in \cref{running_example_kahler}, but let $L(t)$ evolve in the normal metric rather than the K\"{a}hler metric. By \cref{grassmannian_metrics}, these two evolutions must agree up to a dilation in $t$:
$$
L_{\textnormal{normal}}(t) = L_{\textnormal{K\"{a}hler}}((\lambda_1 - \lambda_2)t).
$$
Indeed, using the result of \cref{running_example_kahler} and \cref{gradient_flow_normal}, we can verify that this holds for $\dot{L}(0)$:
\begin{gather*}
\dot{L}_{\textnormal{normal}}(0) = [L_0, [L_0, N]] = 4(aq-bp)\ii\begin{bmatrix}-b & a \\ a & b\end{bmatrix} = (\lambda_1 - \lambda_2)\dot{L}_{\textnormal{K\"{a}hler}}(0).\qedhere
\end{gather*}

\end{eg}

\begin{lem}\label{normal_tridiagonal_necessary}
Let $\bflambda\in\mathbb{R}^n$ be strictly decreasing, and suppose that the gradient flow \eqref{gradient_flow_normal_equation} on $\Orbit_{\bflambda}$ with respect to $N\in\uu_n$ in the normal metric weakly preserves positivity. Then $\ii N\in\gl_n^{\ge 0}$.
\end{lem}

\begin{proof}
We assume that $\ii N$ is real. We must show that
$$
\ii N_{i,j} = 0 \text{ for all } i\ge j+2 \quad \text{ and } \quad \ii N_{j+1,j} \ge 0 \text{ for all } j.
$$
To this end, set $g_0 := \I_n\in\U_n^{\ge 0}$, and let $g(t)\in\U_n$ and $L(t)\in\Orbit_{\bflambda}$ evolve as in \cref{lax_flow}\ref{lax_flow_translation}, with $M(t) := [L(t),N]$. By \eqref{lax_flow_minors_inequality}, we have
$$
\textstyle\frac{d}{dt}\eval{t=0}\Delta_I(g(t)) \ge 0 \quad \text{ for all } I\subseteq [n] \text{ such that } I\neq [1], \dots, [n].
$$
Note that
$$
\dot{g}(0) = -[L_0,N]g_0 = -[\ii\Diag{\bflambda},N], \quad \text{ so } \quad \dot{g}(0)_{i,j} = \ii(\lambda_j - \lambda_i)N_{i,j} \text{ for } 1 \le i,j \le n.
$$
Using \eqref{determinant_derivative}, for $i \ge j+1$ we calculate
$$
\textstyle\frac{d}{dt}\eval{t=0}\Delta_{[j-1]\cup\{i\}}(g(t)) = \ii(\lambda_j - \lambda_i)N_{i,j}, \quad \text{ so } \quad \ii N_{i,j} \ge 0.
$$
Similarly, for $i \ge j+2$ we calculate
\begin{gather*}
\textstyle\frac{d}{dt}\eval{t=0}\Delta_{[j-1]\cup\{j+1,i\}}(g(t)) = -\ii(\lambda_j - \lambda_i)N_{i,j}, \quad \text{ so } \quad \ii N_{i,j} \le 0.\qedhere
\end{gather*}

\end{proof}

\begin{rmk}\label{tridiagonal_necessary_extension}
We observe that \cref{normal_tridiagonal_necessary} and its proof extend to the case that $\bflambda$ is weakly decreasing. Rather than obtaining that $\ii N$ lies in $\gl_n^{\ge 0}$, the conclusion is that $\ii N$ is a block Jacobi matrix, where the block sizes are determined by the multiplicities of $\bflambda$.
\end{rmk}

\begin{thm}\label{positivity_preserving_normal}
Let $\bflambda\in\mathbb{R}^n$ be strictly decreasing, and let $N\in\uu_n$. Then the gradient flow \eqref{gradient_flow_normal_equation} on $\Orbit_{\bflambda}$ with respect to $N$ in the normal metric does not strictly preserves positivity, and it weakly preserves positivity if and only if $N$ is a scalar multiple of $\I_n$ (i.e.\ the flow is constant).
\end{thm}

\begin{proof}
Suppose that the gradient flow \eqref{gradient_flow_normal_equation} with respect to $N$ weakly preserves positivity. We must show that $N$ is a scalar multiple of $\I_n$. By \cref{normal_tridiagonal_necessary}, we have $\ii N\in\gl_n^{\ge 0}$. It suffices to show that for all $1 \le j \le n-2$, the principal submatrix of $N$ using rows and columns $\{j,j+1,j+2\}$ is a scalar multiple of $\I_3$.

To this end, we first consider the case $n=3$. Let $g_0\in\U_3^{\ge 0}$, and let $g(t)\in\U_3$ and $L(t)\in\Orbit_{\bflambda}$ evolve as in \cref{lax_flow}\ref{lax_flow_translation}, with $M(t) := [L(t),N]$. For various choices of $g_0$ and $I$ such that $\Delta_I(g_0) = 0$, we apply \eqref{lax_flow_minors_inequality} and obtain $\textstyle\frac{d}{dt}\eval{t=0}\Delta_I(g(t)) \ge 0$.

We have
$$
g_0 = \begin{bmatrix}
1 & 0 & 0 \\[2pt]
0 & \frac{1}{\sqrt{2}} & -\frac{1}{\sqrt{2}} \\[6pt]
0 & \frac{1}{\sqrt{2}} & \frac{1}{\sqrt{2}} \\[2pt]
\end{bmatrix},\; I = \{3\} \quad\Longrightarrow\quad -\displaystyle\frac{\lambda_2 - \lambda_3}{2}\ii N_{2,1} \ge 0,
$$
and
$$
g_0 = \begin{bmatrix}
0 & -\frac{1}{\sqrt{2}} & \frac{1}{\sqrt{2}} \\[6pt]
0 & -\frac{1}{\sqrt{2}} & -\frac{1}{\sqrt{2}} \\[6pt]
1 & 0 & 0
\end{bmatrix},\; I = \{1\} \quad\Longrightarrow\quad -\displaystyle\frac{\lambda_2 - \lambda_3}{2}\ii N_{2,3} \ge 0.
$$
Since $\ii N_{2,1}\ge 0$ and $\ii N_{2,3} \ge 0$, we get $N_{2,1} = N_{2,3} = 0$. Therefore $N$ is diagonal.

Now we have
\begin{gather*}
g_0 = \begin{bmatrix}
\frac{1}{\sqrt{2}} & -\frac{1}{2} & \frac{1}{2} \\[6pt]
\frac{1}{\sqrt{2}} & \frac{1}{2} & -\frac{1}{2} \\[6pt]
0 & \frac{1}{\sqrt{2}} & \frac{1}{\sqrt{2}} \\[2pt]
\end{bmatrix},\; I = \{3\} \quad\Longrightarrow\quad \displaystyle\frac{\lambda_2 - \lambda_3}{4}\ii(N_{1,1} - N_{2,2}) \ge 0, \\[4pt]
g_0 = \begin{bmatrix}
0 & -\frac{1}{\sqrt{2}} & \frac{1}{\sqrt{2}} \\[6pt]
\frac{1}{\sqrt{2}} & -\frac{1}{2} & -\frac{1}{2} \\[6pt]
\frac{1}{\sqrt{2}} & \frac{1}{2} & \frac{1}{2} \\[2pt]
\end{bmatrix},\; I = \{1\} \quad\Longrightarrow\quad \displaystyle\frac{\lambda_2 - \lambda_3}{4}\ii(N_{3,3} - N_{2,2}) \ge 0, \\[4pt]
g_0 = \begin{bmatrix}
\frac{1}{2} & -\frac{1}{2} & \frac{1}{\sqrt{2}} \\[6pt]
\frac{1}{2} & -\frac{1}{2} & -\frac{1}{\sqrt{2}} \\[6pt]
\frac{1}{\sqrt{2}} & \frac{1}{\sqrt{2}} & 0 \\[2pt]
\end{bmatrix},\; I = \{1,2\} \quad\Longrightarrow\quad \displaystyle\frac{\lambda_1 - \lambda_2}{4}\ii(N_{2,2} - N_{1,1}) \ge 0,
\end{gather*}
and
$$
g_0 = \begin{bmatrix}
\frac{1}{\sqrt{2}} & -\frac{1}{\sqrt{2}} & 0 \\[6pt]
\frac{1}{2} & \frac{1}{2} & -\frac{1}{\sqrt{2}} \\[6pt]
\frac{1}{2} & \frac{1}{2} & \frac{1}{\sqrt{2}} \\[6pt]
\end{bmatrix},\; I = \{2,3\} \quad\Longrightarrow\quad \displaystyle\frac{\lambda_1 - \lambda_2}{4}\ii(N_{2,2} - N_{3,3}) \ge 0.
$$
(We note that these four choices are related by applying the maps $\rev$ and $\flip$; cf.\ \cref{flip_rev_action}.) Therefore $N_{1,1} = N_{2,2} = N_{3,3}$, so $N$ is a scalar multiple of $\I_3$, as desired.

Now we consider the case of general $n\ge 3$. Let $\tilde{N}$ denote the principal submatrix of $N$ using rows and columns $\{j,j+1,j+2\}$, where $1 \le j \le n-2$. We prove by induction on $j$ (with no base case) that $\tilde{N}$ is a scalar multiple of $\I_3$. Given $\tilde{g}_0\in\U_3^{\ge 0}$, define
$$
g_0 := \begin{bmatrix}\I_{j-1} & 0 & 0 \\ 0 & \tilde{g}_0 & 0 \\ 0 & 0 & \I_{n-j-2}\end{bmatrix}\in\U_n^{\ge 0}.
$$
Let $g(t)\in\U_n$ and $L(t)\in\Orbit_{\bflambda}$ evolve as in \cref{lax_flow}\ref{lax_flow_translation}, with $M(t) := [L(t),N]$. Let $\tilde{g}(t)\in\U_3$ and $\tilde{L}(t)\in\Orbit_{(\lambda_j, \lambda_{j+1}, \lambda_{j+2})}$ evolve similarly, with $\tilde{M}(t) := [\tilde{L}(t),\tilde{N}]$. By induction, we may assume that
$$
N = \begin{bmatrix}c\I_{j-1} & 0 & 0 \\ 0 & \tilde{N} & \ast \\ 0 & \ast & \ast\end{bmatrix} \text{ for some scalar } c, \text{ so that } \dot{g}(0) = \begin{bmatrix}0 & 0 & 0 \\ 0 & \dot{\tilde{g}}(0) & \ast \\ 0 & \ast & \ast\end{bmatrix}.
$$
Now for any $\tilde{I}\subseteq [3]$, define $I\subseteq [j+2]$ by $I := [j-1]\cup\{j-1+i : i\in\tilde{I}\}$. Then using \eqref{determinant_derivative}, we find
$$
\Delta_I(g_0) = \Delta_{\tilde{I}}(\tilde{g}_0) \quad \text{ and } \quad \textstyle\frac{d}{dt}\eval{t=0}\Delta_I(g(t)) = \textstyle\frac{d}{dt}\eval{t=0}\Delta_{\tilde{I}}(\tilde{g}(t)).
$$
Therefore by \eqref{lax_flow_minors_inequality}, choosing $\tilde{g}_0$ and $\tilde{I}$ as in the case $n=3$ above, we find that $\tilde{N}$ is a scalar multiple of $\I_3$. This completes the induction.
\end{proof}

\begin{rmk}\label{sometimes_preserving}
Let $\bflambda\in\mathbb{R}^n$ be strictly decreasing, where $n \ge 3$. We note that while the constant flow on $\Orbit_{\bflambda}$ is the only gradient flow in the normal metric which weakly preserves positivity, there do exist nonconstant gradient flows which preserve the tridiagonal subset $\Jac_{\bflambda}^{\ge 0}$ of $\Orbit_{\bflambda}^{\ge 0}$. Indeed, the gradient flow with respect to $N := -\ii\Diag{n-1, \dots, 1, 0} \in \uu_n$ preserves $\Jac_{\bflambda}^{\ge 0}$ in both time directions, by \cref{tridiagonal_symmetric_toda_gradient} and \cref{full_symmetric_toda_gradient}\ref{full_symmetric_toda_gradient_positivity}. This is the Toda lattice flow, which we study in detail in \cref{sec_toda}. It would be interesting to know if there are other natural subsets of $\Orbit_{\bflambda}^{\ge 0}$ which are preserved by some nonconstant gradient flow.
\end{rmk}

\subsection{The induced metric}\label{sec_gradient_induced}
In this subsection, we consider the gradient flow on $\Orbit_{\bflambda}$ with respect to $N$ in the induced metric, when $\Orbit_{\bflambda}$ is isomorphic to the complete flag variety $\Fl_n(\mathbb{C})$. We will show (see \cref{induced_tridiagonal_necessary}) that a necessary condition for positivity to be preserved is that $\ii N\in\gl_n^{\ge 0}$. We will also give an example (see \cref{induced_extended_example} and \cref{induced_no_preserving}) showing that the condition $\ii N\in\gl_n^{\ge 0}$ is not sufficient. While we are not able to determine necessary and sufficient conditions in general, our investigations indicate that such conditions likely depend in an intricate way on both $N$ and $\bflambda$. This is in contrast to gradient flows on $\Orbit_{\bflambda}$ in the other two metrics, where the conditions do not depend on $\bflambda$. In the case of the K\"{a}hler metric, this is because by definition, the metric does not depend on the choice of $\bflambda$. In the case of the normal metric, this is not obvious beforehand, but it follows from \cref{positivity_preserving_normal}.

We begin by giving explicit descriptions for gradient flows in the induced metric. We begin by considering any weakly decreasing $\bflambda$, and will later specialize to the case that $\bflambda$ is strictly decreasing. We recall the decomposition \eqref{ad_decomposition}.
\begin{prop}\label{gradient_flow_induced}
Let $\bflambda\in\mathbb{R}^n$ be weakly decreasing, and let $N\in\uu_n$. Then the gradient flow on $\Orbit_{\bflambda}$ with respect to $N$ in the induced metric is given by
\begin{align}\label{gradient_flow_induced_equation}
\dot{L}(t) = -N^{L(t)}.
\end{align}

\end{prop}

\begin{proof}
Take $M(t)\in\uu_n$ such that $[L(t),M(t)] = -N^{L(t)}$. Using \cref{defn_metrics} and \cref{gradient_computation}, we must show that
$$
\kappa([L(t),X],N^{L(t)}) = \kappa([L(t),X],N)
$$
for all $t$ and tangent vectors $[L(t),X]$ at $L(t)$. Indeed, since $\kappa$ is $[\cdot,\cdot]$-invariant, we have
\begin{gather*}
\kappa([L(t),X],N) = -\kappa(X,[L(t),N]) = -\kappa(X,[L(t),N^{L(t)}]) = \kappa([L(t),X],N^{L(t)}).\qedhere
\end{gather*}

\end{proof}

\begin{eg}\label{running_example_induced}
Let us consider the same setup as in \cref{running_example_kahler} and \cref{running_example_normal}, but let $L(t)$ evolve in the induced metric. By \cref{grassmannian_metrics}, we must have
$$
L_{\textnormal{induced}}(t) = L_{\textnormal{K\"{a}hler}}((\lambda_1 - \lambda_2)^{-1}t) = L_{\textnormal{normal}}((\lambda_1 - \lambda_2)^{-2}t).
$$
Let us verify that this holds for $\dot{L}(0)$. We have the decomposition
$$
N = N^{L_0} + N_{L_0} = -\frac{aq - bp}{a^2 + b^2}\ii\begin{bmatrix}-b & a \\ a & b\end{bmatrix} - \frac{ap + bq}{a^2 + b^2}\ii\begin{bmatrix}a & b \\ b & -a\end{bmatrix}.
$$
By \cref{gradient_flow_induced}, we obtain
\begin{gather*}
\dot{L}_{\textnormal{induced}}(0) = -N^{L_0} = \frac{aq - bp}{a^2 + b^2}\ii\begin{bmatrix}-b & a \\ a & b\end{bmatrix} = \frac{1}{\lambda_1 - \lambda_2}\dot{L}_{\textnormal{K\"{a}hler}}(0) = \frac{1}{(\lambda_1 - \lambda_2)^2}\dot{L}_{\textnormal{normal}}(0).\qedhere
\end{gather*}

\end{eg}

We now use \cref{lax_flow} to translate \eqref{gradient_flow_induced_equation} into a flow on $\PFl{K}{n}(\mathbb{C})$, by defining for all $L,N\in\uu_n$ an element $M\in\uu_n$ such that $[L,M] = - N^L$. While such an $M$ is only uniquely defined modulo $\ker(\ad_L)$, we fix a specific choice of $M$, which we denote by $\adinverse{L}(-N)$.
\begin{defn}\label{defn_ad_inverse}
Let $\bflambda\in\mathbb{R}^n$ be weakly decreasing. Define the linear operator $\adinverse{\ii\hspace*{-0.5pt}\Diag{\bflambda}}$ on $\uu_n$ by
$$
(\adinverse{\ii\hspace*{-0.5pt}\Diag{\bflambda}}(M))_{i,j} := \begin{cases}
0, & \text{ if $\lambda_i = \lambda_j$}; \\
\frac{\ii}{\lambda_j - \lambda_i}M_{i,j}, & \text{ otherwise},
\end{cases} \quad \text{ for } 1 \le i,j \le n.
$$
Then given $L\in\Orbit_{\bflambda}$, write $L = g(\ii\Diag{\bflambda})g^{-1}$ for some $g\in\U_n$, and define the linear operator $\adinverse{L}$ on $\uu_n$ by
\begin{align}\label{defn_ad_inverse_expansion}
\adinverse{L}(M) := g\adinverse{\ii\hspace*{-0.5pt}\Diag{\bflambda}}(g^{-1}Mg)g^{-1}.
\end{align}
We can verify that the definition of $\adinverse{L}$ depends only on $L$, not on the choice of $g$. In particular, $\ad_L^{-1}(M)$ is a smooth function of $L\in\Orbit_{\bflambda}$ and $M\in\uu_n$.
\end{defn}

\begin{lem}\label{ad_inverse_property}
Let $\bflambda\in\mathbb{R}^n$ be weakly decreasing, and let $L\in\Orbit_{\bflambda}$. Then
$$
[L,\adinverse{L}(M)] = M^L \quad \text{ for all } M\in\uu_n.
$$
\end{lem}

\begin{proof}
First we consider the case $L = \ii\Diag{\bflambda}$. The desired equality follows directly using
$$
(M^{\ii\hspace*{-0.5pt}\Diag{\bflambda}})_{i,j} = \begin{cases}
0, & \text{ if $\lambda_i = \lambda_j$}; \\
M_{i,j}, & \text{ otherwise},
\end{cases} \quad \text{ for } 1 \le i,j \le n.
$$

Now we consider the case of general $L$. Write $L = g(\ii\Diag{\bflambda})g^{-1}$ for some $g\in\U_n$. Note that $M^L = g(g^{-1}Mg)^{\ii\hspace*{-0.5pt}\Diag{\bflambda}}g^{-1}$. Therefore, taking the desired equality $[L,\adinverse{L}(M)] = M^L$ and conjugating it by $g^{-1}$, we obtain
$$
[\ii\Diag{\bflambda}, \adinverse{\ii\hspace*{-0.5pt}\Diag{\bflambda}}(g^{-1}Mg)] = (g^{-1}Mg)^{\ii\hspace*{-0.5pt}\Diag{\bflambda}},
$$
which we have verified above.
\end{proof}

\begin{lem}\label{gradient_flow_induced_flag}
Let $\bflambda\in\mathbb{R}^n$ be weakly decreasing, and let $N\in\uu_n$. Let $g(t)\in\U_n$ evolve according to
\begin{align}\label{gradient_flow_induced_flag_equation}
\dot{g}(t) = \adinverse{L(t)}(N)g(t), \quad \text{ where } L(t) = g(t)(\ii\Diag{\bflambda})g(t)^{-1},
\end{align}
beginning at $g_0\in\U_n$. Then
$$
\dot{L}(t) = -[L(t),\adinverse{L(t)}(N)],
$$
and $L(t)$ is the gradient flow \eqref{gradient_flow_induced_equation} on $\Orbit_{\bflambda}$ with respect to $N$ in the induced metric, beginning at $L_0 = g_0(\ii\Diag{\bflambda})g_0^{-1}\in\Orbit_{\bflambda}$.
\end{lem}

\begin{proof}
This follows from \cref{lax_flow}\ref{lax_flow_translation}, using \cref{gradient_flow_induced} and \cref{ad_inverse_property}.
\end{proof}

\begin{eg}\label{eg_ad_inverse}
Let us consider the same setup as in \cref{running_example_induced}, i.e.,
$$
L_0 := \ii\begin{bmatrix}a & b \\ b & -a\end{bmatrix} \quad \text{ and } \quad N := -\ii\begin{bmatrix}p & q \\ q & -p\end{bmatrix}.
$$
As in \cref{running_example_kahler}, we have $L_0\in\Orbit_{\bflambda}$, where $\lambda_1 = \sqrt{a^2+b^2} = -\lambda_2$. Also,
$$
L_0 = g_0(\ii\Diag{\bflambda})g_0^{-1}, \quad \text{ where } g_0 := \frac{1}{\sqrt{2\lambda_1}}\begin{bmatrix}
\sqrt{\lambda_1 + a} & -\sqrt{\lambda_1 - a} \\[4pt]
\sqrt{\lambda_1 - a} & \sqrt{\lambda_1 + a}
\end{bmatrix}\in\U_2.
$$
By \eqref{defn_ad_inverse_expansion}, we have
$$
\adinverse{L_0}(N) = g_0\adinverse{\ii\hspace*{-0.5pt}\Diag{\bflambda}}(g_0^{-1}Ng_0)g_0^{-1} = g_0\hspace*{1pt}\frac{aq-bp}{2(a^2 + b^2)}\begin{bmatrix}
0 & -1 \\[1pt]
1 & 0
\end{bmatrix}g_0^{-1} = \frac{aq-bp}{2(a^2 + b^2)}\begin{bmatrix}
0 & -1 \\[1pt]
1 & 0
\end{bmatrix}.
$$
Therefore by \cref{gradient_flow_induced_flag}, we have
$$
\dot{L}_{\textnormal{induced}}(0) = -[L_0,\adinverse{L_0}(N)] = \frac{aq - bp}{a^2 + b^2}\ii\begin{bmatrix}-b & a \\ a & b\end{bmatrix},
$$
in agreement with \cref{running_example_induced}.
\end{eg}

In the remainder of this subsection, we focus on the case that $\bflambda$ is strictly decreasing, i.e., $\Orbit_{\bflambda} \cong \Fl_n(\mathbb{C})$. The following result and its proof are analogous to \cref{normal_tridiagonal_necessary}, with the normal metric replaced by the induced metric; \cref{tridiagonal_necessary_extension} also applies.
\begin{prop}\label{induced_tridiagonal_necessary}
Let $\bflambda\in\mathbb{R}^n$ be strictly decreasing, and suppose that the gradient flow \eqref{gradient_flow_induced_equation} on $\Orbit_{\bflambda}$ with respect to $N\in\uu_n$ in the induced metric weakly preserves positivity. Then $\ii N\in\gl_n^{\ge 0}$.
\end{prop}

\begin{proof}
We assume that $\ii N$ is real. We must show that
$$
\ii N_{i,j} = 0 \text{ for all } i\ge j+2 \quad \text{ and } \quad \ii N_{j+1,j} \ge 0 \text{ for all } j.
$$
To this end, set $g_0 := \I_n\in\U_n^{\ge 0}$, and let $g(t)\in\U_n$ evolve as in \eqref{gradient_flow_induced_flag_equation}, with $L(t) = g(t)(\ii\Diag{\bflambda})g(t)^{-1}\in\Orbit_{\bflambda}$. By \eqref{lax_flow_minors_inequality}, we have
$$
\textstyle\frac{d}{dt}\eval{t=0}\Delta_I(g(t)) \ge 0 \quad \text{ for all } I\subseteq [n] \text{ such that } I\neq [1], \dots, [n].
$$
Note that
$$
\dot{g}(0) = \adinverse{L_0}(N)g_0 = \adinverse{\ii\hspace*{-0.5pt}\Diag{\bflambda}}(N), \quad \text{ so } \quad \dot{g}(0)_{i,j} = \begin{cases}
0, & \text{ if $i=j$}; \\
\frac{\ii}{\lambda_j - \lambda_i}N_{i,j}, & \text{ otherwise}.
\end{cases}
$$
Using \eqref{determinant_derivative}, for $i \ge j+1$ we calculate
$$
\textstyle\frac{d}{dt}\eval{t=0}\Delta_{[j-1]\cup\{i\}}(g(t)) = \frac{\ii}{\lambda_j - \lambda_i}N_{i,j}, \quad \text{ so } \quad \ii N_{i,j} \ge 0.
$$
Similarly, for $i \ge j+2$ we calculate
\begin{gather*}
\textstyle\frac{d}{dt}\eval{t=0}\Delta_{[j-1]\cup\{j+1,i\}}(g(t)) = \frac{-\ii}{\lambda_j - \lambda_i}N_{i,j}, \quad \text{ so } \quad \ii N_{i,j} \le 0.\qedhere
\end{gather*}

\end{proof}

We now further consider the flow \eqref{gradient_flow_induced_flag_equation}. Using \eqref{defn_ad_inverse_expansion}, we can rewrite \eqref{gradient_flow_induced_flag_equation} as
\begin{align}\label{gradient_flow_induced_flag_equation_rewritten}
\dot{g}(t) = g(t)\adinverse{\ii\hspace*{-0.5pt}\Diag{\bflambda}}(g(t)^{-1}Ng(t)).
\end{align}
When $\bflambda$ is strictly decreasing, we wish to view \eqref{gradient_flow_induced_flag_equation_rewritten} as a flow on $\Fl_n(\mathbb{C})$, and it will be more convenient to have $g(t)$ acted upon on the left, rather than the right. To achieve this, we apply the twist map from \cref{sec_twist}. Since the twist map preserves total positivity and total nonnegativity (see \cref{twist_action}), we may work with the twisted flow when considering which flows \eqref{gradient_flow_induced_flag_equation} preserve positivity. This in turn is equivalent to working with \eqref{gradient_flow_induced_equation}, by \cref{lax_flow}\ref{lax_flow_minors}. We summarize these observations in the following result:
\begin{lem}\label{gradient_flow_induced_flag_twisted}
Let $\bflambda\in\mathbb{R}^n$ be strictly decreasing, and let $N\in\uu_n$. Let $g(t)\in\U_n$ evolve according to \eqref{gradient_flow_induced_flag_equation}, and set $h(t) := \iota(t) = \delta_ng(t)^{-1}\delta_n \in \U_n$. Then $h(t)$ evolves according to
\begin{align}\label{gradient_flow_induced_flag_twisted_equation}
\dot{h}(t) = -\hspace*{-1pt}\adinverse{\ii\hspace*{-0.5pt}\Diag{\bflambda}}(h(t)\delta_nN\delta_nh(t)^{-1})h(t).
\end{align}
Furthermore, the gradient flow on $\Orbit_{\bflambda}$ with respect to $N$ in the induced metric weakly (respectively, strictly) preserves positivity if and only if the flow \eqref{gradient_flow_induced_flag_twisted_equation} on $\U_n$ weakly (respectively, strictly) preserves positivity.
\end{lem}

\begin{proof}
This follows from the preceding discussion, where we obtain \eqref{gradient_flow_induced_flag_twisted_equation} from \eqref{gradient_flow_induced_flag_equation_rewritten}.
\end{proof}

We emphasize that since we are employing the twist map, \cref{gradient_flow_induced_flag_twisted} only applies when $\bflambda$ is strictly decreasing. We also observe that the technique of applying the twist map can be employed to flows much more generally, and we will do so again for the symmetric Toda flow in \cref{sec_toda_normal}.

We believe it may be possible to use \eqref{gradient_flow_induced_flag_twisted_equation} to classify which gradient flows on $\Orbit_{\bflambda}$ (when $\bflambda$ is strictly decreasing) in the induced metric preserve positivity. As a first step in this direction, we investigate the case $n=3$. We will find that, curiously, whether or not positivity is preserved appears to depend on the choice of $\bflambda$ (though we are unable to prove this); see \eqref{induced_extended_example_edges_symmetrized} and \cref{induced_no_preserving}.

\begin{eg}\label{induced_extended_example}
Let $n := 3$, let $\bflambda\in\mathbb{R}^3$ be strictly decreasing, and let $N\in\uu_3$. We wish to determine when the gradient flow on $\Orbit_{\bflambda}$ with respect to $N$ in the induced metric weakly preserves positivity. By \cref{induced_tridiagonal_necessary}, it suffices to consider the case when $\ii N\in\gl_3^{\ge 0}$. Also, after translating $N$ by a scalar multiple of $\I_3$ (which does not change the gradient flow), we may assume that $N_{2,2} = 0$. That is,
\begin{align}\label{induced_extended_example_convention}
N = -\ii\begin{bmatrix}p & u & 0 \\ u & 0 & v \\ 0 & v & q\end{bmatrix} \quad \text{ for some } p,q\in\mathbb{R} \text{ and } u,v \ge 0.
\end{align}
For convenience, we also set
$$
c := \lambda_1 - \lambda_2 > 0 \quad \text{ and } \quad d := \lambda_2 - \lambda_3 > 0.
$$

Let $g(t)$ evolve according to \eqref{gradient_flow_induced_flag_twisted_equation}, beginning at an arbitrary $g_0\in\U_3^{\ge 0}$. In particular, we have
$$
\dot{g}(0) = -\hspace*{-1pt}\adinverse{\ii\hspace*{-0.5pt}\Diag{\bflambda}}\Bigg(g_0\ii\begin{bmatrix}-p & u & 0 \\ u & 0 & v \\ 0 & v & -q\end{bmatrix}g_0^{-1}\Bigg)g_0.
$$
We will determine when the inequalities in \eqref{lax_flow_minors_inequality} hold:
\begin{align}\label{induced_extended_example_inequality}
\textstyle\frac{d}{dt}\eval{t=0}\Delta_I(g(t)) \ge 0 \quad \text{ for all } I\subseteq [3] \text{ such that } \Delta_I(g_0) = 0.
\end{align}
We can express the left-hand side above as follows:
$$
\textstyle\frac{d}{dt}\eval{t=0}\Delta_{\{i\}}(g(t)) = \dot{g}(0)_{i,1} \quad \text{ and } \quad \textstyle\frac{d}{dt}\eval{t=0}\Delta_{[3]\setminus\{i\}}(g(t)) = (-1)^{i-1}\dot{g}(0)_{i,3}
$$
for all $i\in [3]$, where the second equality follows from \eqref{jacobi}. We emphasize that our approach based on \eqref{induced_extended_example_inequality} gives a necessary condition for positivity to be preserved, but not necessarily a sufficient condition, because \eqref{lax_flow_minors_inequality} only considers $g(t)$ to first order in $t$.

We consider several cases, depending on which cell $C_{v,w}$ contains $g_0$ in the cell decomposition \eqref{cell_decomposition_equation} of $\Fl_3^{\ge 0}$. Here, $v$ and $w$ are permutations in $\mathfrak{S}_3$ such that $v \le w$ (cf.\ \cref{figure_S3}). We observe that by symmetry, some cases are redundant. Namely, recall the involutions $\rev$ and $\flip$ defined on $\U_n$ from \cref{sec_flip}, which act on the cell decomposition \eqref{cell_decomposition_equation} according to \cref{flip_rev_action}. Therefore we only need to consider one cell among the orbit
$$
C_{v,w},\hspace*{1pt} C_{w_0w,w_0v},\hspace*{1pt} C_{ww_0,vw_0},\hspace*{1pt} C_{w_0vw_0,w_0ww_0},
$$
where $w_0 = 321$. On the other hand, $\rev$ and $\flip$ are compatible with \eqref{gradient_flow_induced_flag_twisted_equation}: the latter is invariant under the transformations
$$
h \leftrightarrow \rev(h) = \mathring{w}_0\delta_3h\delta_3, \quad N \leftrightarrow -\delta_3N\delta_3, \quad (\lambda_1, \lambda_2, \lambda_3) \leftrightarrow (-\lambda_3, -\lambda_2, -\lambda_1);
$$
and
$$
h \leftrightarrow \flip(h) = \delta_3h\delta_3\mathring{w}_0, \quad N \leftrightarrow \mathring{w}_0\delta_3N\mathring{w}_0\delta_3.
$$
In terms of the data $(c, d, p, q, u, v)$, these transformations correspond to, respectively,
\begin{align}\label{induced_extended_example_symmetries}
c \leftrightarrow d, \quad p \leftrightarrow -p, \quad q \leftrightarrow -q; \quad \text{ and } \quad p \leftrightarrow q, \quad u \leftrightarrow v.
\end{align}

Also observe that when $(v,w) = (123,321)$, we have $C_{v,w} = \Fl_3^{>0}$, so that $\Delta_I(g_0) \neq 0$ for all $I\subseteq [3]$. Therefore \eqref{induced_extended_example_inequality} is vacuously satisfied in this case, and so we do not need to consider it below. We note that the discussion above for $n=3$ can be easily generalized to any $n$.

We now consider the six possible cases. Below, we let $\alpha$ and $\beta$ denote arbitrary numbers in the interval $(0, \frac{\pi}{2})$.\vspace*{2pt}

{\itshape Case 1: $(v,w)$ equals $(123,123)$ or $(321,321)$.} We assume that $(v,w) = (123,123)$. Then
$$
g_0 = \begin{bmatrix}
1 & 0 & 0 \\
0 & 1 & 0 \\
0 & 0 & 1
\end{bmatrix} \quad \text{ and } \quad \dot{g}(0) = \begin{bmatrix}
0 & -\frac{u}{c} & 0 \\[4pt]
\frac{u}{c} & 0 & -\frac{v}{d} \\[4pt]
0 & \frac{v}{d} & 0
\end{bmatrix}.
$$
We must check \eqref{induced_extended_example_inequality} when $I = \{2\}, \{3\}, \{1,3\}, \{2,3\}$:
$$
\frac{u}{c} \ge 0, \qquad 0 \ge 0, \qquad \frac{v}{d} \ge 0, \qquad 0 \ge 0.
$$
These inequalities are always satisfied.\vspace*{2pt}

{\itshape Case 2: $(v,w)$ equals $(132,132)$, $(312,312)$, $(231,231)$, or $(213,213)$.} We assume that $(v,w) = (132,132)$. Then
$$
g_0 = \begin{bmatrix}
1 & 0 & 0 \\
0 & 0 & -1 \\
0 & 1 & 0
\end{bmatrix} \quad \text{ and } \quad \dot{g}(0) = \begin{bmatrix}
0 & -\frac{u}{c+d} & 0 \\[4pt]
0 & \frac{v}{d} & 0 \\[4pt]
\frac{u}{c+d} & 0 & \frac{v}{d}
\end{bmatrix}.
$$
We must check \eqref{induced_extended_example_inequality} when $I = \{2\}, \{3\}, \{1,2\}, \{2,3\}$:
$$
0 \ge 0, \qquad \frac{u}{c+d} \ge 0, \qquad \frac{v}{d} \ge 0, \qquad 0 \ge 0.
$$
These inequalities are always satisfied.\vspace*{2pt}

{\itshape Case 3: $(v,w)$ equals $(123,132)$, $(312,321)$, $(231,321)$, or $(123,213)$.} We assume that $(v,w) = (123,132)$. Then
$$
g_0 = \begin{bmatrix}
1 & 0 & 0 \\[1pt]
0 & \cos(\alpha) & -\hspace*{-1pt}\sin(\alpha) \\[1pt]
0 & \sin(\alpha) & \cos(\alpha)
\end{bmatrix} \quad \text{ and } \quad \dot{g}(0) = \begin{bmatrix}
\ast & \ast & \frac{du\sin(2\alpha)}{2c(c+d)} \\[6pt]
\frac{u\cos(\alpha)}{c} & \ast & \ast \\[6pt]
\frac{u\sin(\alpha)}{c+d} & \ast & \ast
\end{bmatrix},
$$
where the entries $\ast$ are unimportant. We must check \eqref{induced_extended_example_inequality} when $I = \{2\}, \{3\}, \{2,3\}$:
$$
\frac{u\cos(\alpha)}{c} \ge 0, \qquad \frac{u\sin(\alpha)}{c+d} \ge 0, \qquad \frac{du\sin(2\alpha)}{2c(c+d)} \ge 0.
$$
These inequalities are always satisfied.\vspace*{2pt}

{\itshape Case 4: $(v,w)$ equals $(213,231)$ or $(132,312)$.} We assume that $(v,w) = (213,231)$. Then
$$
g_0 = \begin{bmatrix}
0 & -\hspace*{-1pt}\cos(\alpha) & \sin(\alpha) \\[1pt]
1 & 0 & 0 \\[1pt]
0 & \sin(\alpha) & \cos(\alpha)
\end{bmatrix} \quad \text{ and } \quad \dot{g}(0) = \begin{bmatrix}
\frac{u\cos(\alpha)}{c} & \ast & \ast \\[1pt]
\ast & \ast & -\frac{(c+d)u\sin(2\alpha)}{2cd} \\[1pt]
\frac{u\sin(\alpha)}{d} & \ast & \ast
\end{bmatrix},
$$
where the entries $\ast$ are unimportant. We must check \eqref{induced_extended_example_inequality} when $I = \{1\}, \{3\}, \{1,3\}$:
$$
\frac{u\cos(\alpha)}{c} \ge 0, \qquad \frac{u\sin(\alpha)}{d} \ge 0, \qquad \frac{(c+d)u\sin(2\alpha)}{2cd} \ge 0.
$$
These inequalities are always satisfied.\vspace*{2pt}

{\itshape Case 5: $(v,w)$ equals $(132,231)$ or $(213,312)$.} We assume that $(v,w) = (132,231)$. Then
\begin{multline*}
g_0 = \begin{bmatrix}
\cos(\alpha) & 0 & \sin(\alpha) \\[1pt]
\sin(\alpha) & 0 & -\hspace*{-1pt}\cos(\alpha) \\[1pt]
0 & 1 & 0
\end{bmatrix} \quad \text{ and} \\
\dot{g}(0) = \begin{bmatrix}
\ast & \ast & \ast \\
\ast & \ast & \ast \\
\frac{cu(1 - \cos(2\alpha)) - cv\sin(2\alpha) + 2du}{2(c+d)d} & \ast & \frac{cv(1 + \cos(2\alpha)) - cu\sin(2\alpha) + 2dv}{2(c+d)d}
\end{bmatrix},
\end{multline*}
where the entries $\ast$ are unimportant. We must check \eqref{induced_extended_example_inequality} when $I = \{3\}, \{1,2\}$:
$$
\frac{cu(1 - \cos(2\alpha)) - cv\sin(2\alpha) + 2du}{2(c+d)d} \ge 0, \qquad \frac{cv(1 + \cos(2\alpha)) - cu\sin(2\alpha) + 2dv}{2(c+d)d} \ge 0.
$$
The left-hand side of the first inequality above is minimized (as a function of $\alpha$) when $\tan(2\alpha) = \frac{v}{u}$, and the left-hand side of the second inequality is minimized when $\tan(2\alpha) = -\frac{u}{v}$. Therefore these inequalities are equivalent to
\begin{align}\label{induced_extended_example_edges}
c(u - \sqrt{u^2 + v^2}) + 2du \ge 0, \qquad c(v - \sqrt{u^2 + v^2}) + 2dv \ge 0.
\end{align}
Symmetrizing according to \eqref{induced_extended_example_symmetries}, we conclude that \eqref{induced_extended_example_inequality} holds in this case if and only if
\begin{align}\label{induced_extended_example_edges_symmetrized}
u = v = 0 \quad \text{ or } \quad \min\bigg(\frac{u}{\sqrt{u^2 + v^2}}, \frac{v}{\sqrt{u^2 + v^2}}\bigg) \ge \max\bigg(\frac{c}{c + 2d}, \frac{d}{2c + d}\bigg).
\end{align}

{\itshape Case 6: $(v,w)$ equals $(123,231)$, $(213,321)$, $(132,321)$, or $(123,312)$.} We assume that $(v,w) = (123,231)$. Then
$$
g_0 = \begin{bmatrix}
\cos(\alpha) & -\hspace*{-1pt}\sin(\alpha)\cos(\beta) & \sin(\alpha)\sin(\beta) \\[1pt]
\sin(\alpha) & \cos(\alpha)\cos(\beta) & -\hspace*{-1pt}\cos(\alpha)\sin(\beta) \\[1pt]
0 & \sin(\beta) & \cos(\beta)
\end{bmatrix}.
$$
We must check \eqref{induced_extended_example_inequality} when $I = \{3\}$:
$$
\dot{g}(0)_{3,1} = \scalebox{0.9}{$\displaystyle\frac{cq\sin(2\alpha)\sin(2\beta) + 2cu(1 - \cos(2\alpha))\sin(\beta) + 2cv\sin(2\alpha)\cos(2\beta) + 4du\sin(\beta)}{4(c+d)d}$} \ge 0.
$$
Multiplying by $\frac{2(c+d)d}{\sin(\beta)}$, we obtain the equivalent inequality
\begin{align}\label{induced_extended_example_faces}
cq\sin(2\alpha)\cos(\beta) + cu(1 - \cos(2\alpha)) + cv\sin(2\alpha)\frac{\cos(2\beta)}{\sin(\beta)} + 2du \ge 0.
\end{align}
Note that if $q \ge 0$, then the left-hand side above is a weakly decreasing function of $\beta$, whence it is minimized as $\beta \to \frac{\pi}{2}$. The inequality then becomes
$$
cu(1 - \cos(2\alpha)) - cv\sin(2\alpha) + 2du \ge 0,
$$
which we considered in Case 5. In particular, if $p = q = 0$, then after symmetrizing according to \eqref{induced_extended_example_symmetries}, we find that \eqref{induced_extended_example_inequality} holds if and only if \eqref{induced_extended_example_edges_symmetrized} holds. In the general case when $p$ or $q$ is nonzero, \eqref{induced_extended_example_faces} (and its images under \eqref{induced_extended_example_symmetries}) will yield stronger conditions than \eqref{induced_extended_example_edges_symmetrized}.

In conclusion, \eqref{induced_extended_example_inequality} is equivalent to the inequality \eqref{induced_extended_example_faces} along with its images under \eqref{induced_extended_example_symmetries}. These inequalities imply \eqref{induced_extended_example_edges_symmetrized}, and they are equivalent to \eqref{induced_extended_example_edges_symmetrized} in the case that $p = q = 0$.

In particular, when $\bflambda$ is fixed (i.e.\ $c,d$ are fixed), there exists a nonzero $N$ (i.e.\ there exist $p,q,u,v$ not all zero) satisfying \eqref{induced_extended_example_inequality} if and only if
\begin{align}\label{induced_extended_example_interval}
\max\Big(\frac{c}{d}, \frac{d}{c}\Big) \le 2 + 2\sqrt{2}.
\end{align}
To see this, note that if \eqref{induced_extended_example_interval} holds, then we may take $(p,q,u,v) := (0,0,1,1)$. Conversely, suppose that \eqref{induced_extended_example_interval} does not hold; we must show that $p = q = u = v = 0$. First we consider the inequality in \eqref{induced_extended_example_edges_symmetrized}. If $u$ or $v$ is nonzero, then the left-hand side is at most $\frac{1}{\sqrt{2}}$, while by assumption, the right-hand side is greater than $\frac{1}{\sqrt{2}}$. Therefore $u = v = 0$. Then the inequality \eqref{induced_extended_example_faces} becomes $cq\sin(2\alpha)\cos(\beta) \ge 0$, which implies $q \ge 0$. Symmetrizing according to \eqref{induced_extended_example_symmetries} gives the inequalities $q \le 0$, $p \ge 0$, and $p \le 0$, so $p = q = 0$, as desired.
\end{eg}

Based on \cref{induced_extended_example}, we make the following observation:
\begin{prop}\label{induced_no_preserving}
Let $\bflambda\in\mathbb{R}^3$ be strictly decreasing such that $\frac{\lambda_1 - \lambda_2}{\lambda_2 - \lambda_3}$ lies outside the interval $\big[\frac{1}{2+2\sqrt{2}},\hspace*{1pt} 2+2\sqrt{2}\big]$, and let $N\in\uu_3$. Then the gradient flow \eqref{gradient_flow_induced_equation} on $\Orbit_{\bflambda}$ with respect to $N$ in the induced metric does not strictly preserves positivity, and it weakly preserves positivity if and only if $N$ is a scalar multiple of $\I_3$ (i.e.\ the flow is constant).
\end{prop}

\begin{proof}
This follows from the last paragraph of \cref{induced_extended_example}.
\end{proof}

\section{Lyapunov function and homeomorphism onto a closed ball}\label{sec_ball}

\noindent Galashin, Karp, and Lam \cite{galashin_karp_lam, galashin_karp_lam19} recently employed the notion of a {\itshape contractive flow} in order to show that the totally nonnegative part of any partial flag variety $G/P$ (as well as several other spaces appearing in algebraic combinatorics) is homeomorphic to a closed ball. In this section we rephrase this argument in the case that $G/P = \PFl{K}{n}(\mathbb{C})$ in terms of the orbit language. The key point is that by \cref{gradient_flow_kahler}, the flows on $\PFl{K}{n}(\mathbb{C})$ considered in \cite{galashin_karp_lam, galashin_karp_lam19} (which were defined by the explicit formula \eqref{gradient_flow_kahler_equation}) are in fact gradient flows in the K\"{a}hler metric. Therefore there is a natural candidate for a {\itshape Lyapunov function}, which we can then substitute for the role of the metric which was used in \cite{galashin_karp_lam, galashin_karp_lam19}.

\subsection{Stable manifold}\label{sec_stable_manifold}
In this subsection, we describe the stable manifold inside $\Orbit_{\bflambda}$ of the unique global attractor for a gradient flow in the K\"{a}hler metric.
\begin{defn}\label{defn_stable_manifold}
Let $-N\in\Orbit_{\bfmu}$, and set $K := \{i\in [n-1] : \mu_i > \mu_{i+1}\}$. (The reason that we are letting $\bfmu$ index the orbit of $-N$, rather than the orbit of $N$, is that we wish to consider the eigenvalues of $\ii N$ in decreasing order.) As in \eqref{projection_sum_formula}, let us write
\begin{align}\label{defn_stable_manifold_decomposition}
\ii N = \Big(\sum_{k\in K}(\mu_k - \mu_{k+1})P_k\Big) + \mu_n\I_n,
\end{align}
where $P_k$ is orthogonal projection from $\mathbb{C}^n$ onto the subspace spanned by the eigenvectors of $\ii N$ corresponding to the eigenvalues $\mu_1, \dots, \mu_k$. We define $\limitproj{N}{k} := P_k$ for all $k\in K$.

Now let $\bflambda\in\mathbb{R}^n$ be weakly decreasing such that $K' := \{i\in [n-1] : \lambda_i > \lambda_{i+1}\}$ is contained in $K$. Then we define
\begin{align}\label{limit_decomposition}
\limitorbit{N}{\bflambda} := \Big(\sum_{k\in K'}(\lambda_k - \lambda_{k+1})\ii\limitproj{N}{k}\Big) + \lambda_n\ii\I_n\in\Orbit_{\bflambda}.
\end{align}
We define the {\itshape stable manifold} (of $\limitorbit{N}{\bflambda}$ under the gradient flow with respect to $N$ in the K\"{a}hler metric) as
$$
\stableorbit{N}{\bflambda} := \{L_0\in\Orbit_{\bflambda} : L(t) \to \limitorbit{N}{\bflambda} \text{ as } t \to\infty\},
$$
where $L(t)$ evolves as in \cref{gradient_flow_kahler}.
\end{defn}

\begin{eg}\label{eg_stable_manifold}
We set $n := 2$, and consider (cf.\ \cref{running_example_kahler})
$$
\ii N := \begin{bmatrix}p & q \\ \overline{q} & -p\end{bmatrix}, \quad \text{ where $p\in\mathbb{R}$, $q\in\mathbb{C}$, and $p$ and $q$ are not both zero}.
$$
We have $-N\in\Orbit_{\bfmu}$, where
$$
\mu_1 := \sqrt{p^2 + |q|^2} \quad \text{ and } \quad \mu_2 := -\sqrt{p^2 + |q|^2}.
$$
Since $p$ and $q$ are not both zero, we have $\mu_1 > \mu_2$, and $\limitproj{N}{1}$ is orthogonal projection onto the eigenspace of $\mu_1$. Therefore the expansion \eqref{defn_stable_manifold_decomposition} is
$$
\ii N = (\mu_1 - \mu_2)\limitproj{N}{1} + \mu_2\I_2, \quad \text{ where } \limitproj{N}{1} = \frac{1}{2\mu_1}\begin{bmatrix}
\mu_1+p & q \\
\overline{q} & \mu_1-p
\end{bmatrix}.
$$ 

Now let $\lambda_1 \ge \lambda_2$. Then
$$
\limitorbit{N}{\bflambda} = (\lambda_1 - \lambda_2)\ii\limitproj{N}{1} + \lambda_2\ii\I_2 = \frac{\lambda_1 - \lambda_2}{2\mu_1}\ii\begin{bmatrix}p & q \\ \overline{q} & -p\end{bmatrix} + \frac{\lambda_1 + \lambda_2}{2}\ii \I_2 \in \Orbit_{\bflambda}.
$$
If $\lambda_1 = \lambda_2$, then $\Orbit_{\bflambda} = \{\limitorbit{N}{\bflambda}\}$ is a point. Otherwise, it will follow from \cref{stable_manifold_orbit}\ref{stable_manifold_orbit_condition} that the stable manifold $\stableorbit{N}{\bflambda}$ equals $\Orbit_{\bflambda}$ minus the single point
$$
(\lambda_1 - \lambda_2)\ii Q + \lambda_2\ii\I_2 = \frac{-\lambda_1 + \lambda_2}{2\mu_1}\ii\begin{bmatrix}p & q \\ \overline{q} & -p\end{bmatrix} + \frac{\lambda_1 + \lambda_2}{2}\ii \I_2.
$$
Here $Q = \I_2 - \limitproj{N}{1}$ is orthogonal projection onto the eigenspace of $\mu_2 = -\mu_1$.
\end{eg}

We show that \cref{defn_stable_manifold} is compatible with positivity:
\begin{lem}\label{eigenvalue_condition_satisfied}
Let $\bflambda\in\mathbb{R}^n$ be weakly decreasing, and set $K := \{i \in [n-1] : \lambda_i > \lambda_{i+1}\}$. Suppose that $-N\in\Orbit_{\bfmu}$ such that the gradient flow on $\Orbit_{\bflambda}$ with respect to $N$ in the K\"{a}hler metric strictly preserves positivity. Then for all $k\in K$, we have $\mu_k > \mu_{k+1}$ and $\ii\limitproj{N}{k}\in\Orbit_{\bfomega{k}}^{>0}$.
\end{lem}

\begin{proof}
By \cref{kahler_flow_positivity_preserving}, for $k\in K$, the flow \eqref{gradient_flow_kahler_equation} on $\Gr_{k,n}(\mathbb{C})$ strictly preserves positivity. Hence by the implication \ref{infinitesimal_k_positive_preserving} $\Rightarrow$ \ref{infinitesimal_k_positive_minors} of \cref{infinitesimal_k}, all $k\times k$ minors of $\exp(\ii N)$ (which has eigenvalues $e^{\mu_1} \ge \cdots \ge e^{\mu_n}$) are positive. Then \cref{gk_k}\ref{gk_k_eigenvalues} implies $\mu_k > \mu_{k+1}$, and \cref{gk_k}\ref{gk_k_eigenvectors} and \cref{Gr_to_orbit} imply $\ii\limitproj{N}{k}\in\Orbit_{\bfomega{k}}^{>0}$.
\end{proof}

We begin by describing the stable manifold of $\Orbit_{\bflambda}$ in the Grassmannian case (cf.\ \cref{defn_projection}), adapting the proof of \cite[Proposition 3.4]{galashin_karp_lam}.
\begin{lem}\label{stable_manifold_grassmannian}
Let $1 \le k \le n-1$, and let $-N\in\Orbit_{\bfmu}$ such that $\mu_k > \mu_{k+1}$.
\begin{enumerate}[label=(\roman*), leftmargin=*, itemsep=2pt]
\item\label{stable_manifold_grassmannian_condition} We have $\stableorbit{N}{\bfomega{k}} = \{\ii P\in\Orbit_{\bfomega{k}} : \rank(\limitproj{N}{k}P) = k\}$.
\item\label{stable_manifold_grassmannian_positivity} If $\ii\limitproj{N}{k}\in\Orbit_{\bfomega{k}}^{>0}$, then the stable manifold $\stableorbit{N}{\bfomega{k}}$ contains $\Orbit_{\bfomega{k}}^{\ge 0}$.
\end{enumerate}

\end{lem}

\begin{proof}
\ref{stable_manifold_grassmannian_condition} Let $\ii P\in\Orbit_{\bfomega{k}}$, and let $\ii P(t)$ evolve according to the gradient flow on $\Orbit_{\bfomega{k}}$ with respect to $N$ in the K\"{a}hler metric (with $P(0) = P$). We must show that $\lim_{t\to\infty}P(t) = \limitproj{N}{k}$ if and only if $\rank(\limitproj{N}{k}P) = k$.

For the forward direction, note that $t\mapsto\rank(\limitproj{N}{k}P(t))$ is a continuous function of $t$, and hence it is constant. If $\lim_{t\to\infty}P(t) = \limitproj{N}{k}$, then this function is identically equal to $k$; taking $t=0$ gives $\rank(\limitproj{N}{k}P) = k$.

Conversely, suppose that $\rank(\limitproj{N}{k}P) = k$. Let us work in an orthonormal basis of eigenvectors of $\ii N$ corresponding to the eigenvalues $\mu_1 \ge \dots \ge \mu_n$, so that $\ii N = \Diag{\bfmu}$ and $\limitproj{N}{k} = \scalebox{0.8}{$\begin{bmatrix}\I_k & 0 \\ 0 & 0\end{bmatrix}$}$.
By \cref{Gr_to_orbit}, we can write $P = \Proj{V}$ for some $V\in\Gr_{k,n}(\mathbb{C})$, which we regard as an $n\times k$ matrix. Write
$$
V = \begin{bmatrix}X \\ Y\end{bmatrix}, \quad \text{ where $X$ is $k\times k$ and $Y$ is $(n-k)\times k$}.
$$
For the moment, suppose that the columns of $V$ are orthonormal. Then
$$
P = VV^* = \begin{bmatrix}XX^* & XY^* \\[2pt] YX^* & YY^*\end{bmatrix}.
$$
Since $\rank(\limitproj{N}{k}P) = k$, we have $\rank(X) = k$. After multiplying $V$ on the right by $X^{-1}$, we may assume that $X = \I_k$.

By \eqref{gradient_flow_kahler_equation}, we have $P(t) = \Proj{V(t)} = V(t)(\adjoint{V(t)}V(t))^{-1}\adjoint{V(t)}$, where $V(t) := \exp(t\ii N)V$. Note that $\exp(t\ii N) = \Diag{e^{t\mu_1}, \dots, e^{t\mu_n}}$, so we may regard $V(t)$ as the $n\times k$ matrix
$$
V(t) = \begin{bmatrix}
\I_k \\
\Diag{e^{t\mu_{k+1}}, \dots, e^{t\mu_n}}Y\Diag{e^{-t\mu_1}, \dots, e^{-t\mu_k}}
\end{bmatrix}.
$$
Since $\mu_k > \mu_{k+1}$, we have $\Diag{e^{t\mu_{k+1}}, \dots, e^{t\mu_n}}Y\Diag{e^{-t\mu_1}, \dots, e^{-t\mu_k}}\to 0$ as $t\to\infty$. Therefore
$$
\lim_{t\to\infty}P(t) = \begin{bmatrix}\I_k & 0 \\ 0 & 0\end{bmatrix} = \limitproj{N}{k}.
$$

\ref{stable_manifold_grassmannian_positivity} Suppose that $\ii\limitproj{N}{k}\in\Orbit_{\bfomega{k}}^{>0}$. Given $\ii P\in\Orbit_{\bfomega{k}}^{\ge 0}$, we must show that $\ii P\in\stableorbit{N}{\bfomega{k}}$. By part \ref{stable_manifold_grassmannian_condition}, it is equivalent to show that $\rank(\limitproj{N}{k}P) = k$. Recall that $\limitproj{N}{k}$ and $P$ have rank $k$, so $\limitproj{N}{k}P$ has rank at most $k$. Conversely, by \cref{projection_tnn_description}, all $k\times k$ minors of $\limitproj{N}{k}$ are real and positive; also, all $k\times k$ minors of $P$ are real and nonnegative, and at least one such minor is positive. Therefore by the Cauchy--Binet identity \eqref{cauchy-binet}, $\limitproj{N}{k}P$ has a positive $k\times k$ minor, so its rank is at least $k$.
\end{proof}

\begin{rmk}\label{stable_manifold_grassmannian_remark}
We observe that in \cref{stable_manifold_grassmannian}, if $\ii\limitproj{N}{k}\in\Orbit_{\bfomega{k}}^{\ge 0}$, then the stable manifold $\stableorbit{N}{\bfomega{k}}$ contains $\Orbit_{\bfomega{k}}^{>0}$; the proof is similar to that of part \ref{stable_manifold_grassmannian_positivity}. Furthermore, if $\ii\limitproj{N}{k}\in\Orbit_{\bfomega{k}}^{\ge 0}\setminus\Orbit_{\bfomega{k}}^{>0}$, then there exists a point in $\Orbit_{\bfomega{k}}^{\ge 0}\setminus\Orbit_{\bfomega{k}}^{>0}$ which is not in the stable manifold $\stableorbit{N}{\bfomega{k}}$. Namely, by \cref{projection_tnn_description} and \eqref{projection_maximal_minors}, there exists $J\in\binom{[n]}{k}$ such that $\Delta_{I,J}(\limitproj{N}{k}) = 0$ for all $I\in\binom{[n]}{k}$. Then take $P$ to be orthogonal projection onto the span of $e_i$ for $i\in J$. The only nonzero minor of $P$ is $\Delta_{J,J}(P) = 1$, so $\ii P\in\Orbit_{\bfomega{k}}^{\ge 0}\setminus\Orbit_{\bfomega{k}}^{>0}$ by \cref{projection_tnn_description}. Also, by the Cauchy--Binet identity \eqref{cauchy-binet}, all $k\times k$ minors of $\limitproj{N}{k}P$ are zero. Hence $\ii P\notin\stableorbit{N}{\bfomega{k}}$ by \cref{stable_manifold_grassmannian}\ref{stable_manifold_grassmannian_condition}.
\end{rmk}

\begin{prop}\label{stable_manifold_orbit}
Let $\bflambda\in\mathbb{R}^n$ be weakly decreasing, set $K := \{i \in [n-1] : \lambda_i > \lambda_{i+1}\}$, and let $-N\in\Orbit_{\bfmu}$ such that $\mu_k > \mu_{k+1}$ for all $k\in K$.
\begin{enumerate}[label=(\roman*), leftmargin=*, itemsep=2pt]
\item\label{stable_manifold_orbit_condition} Let $L\in\Orbit_{\bflambda}$, and write $-\ii L = (\sum_{k\in K}(\lambda_k - \lambda_{k+1})P_k) + \lambda_n\I_n$ as in \eqref{projection_sum_formula}. Then
$$
L\in\stableorbit{N}{\bflambda} \quad \iff \quad \rank(\limitproj{N}{k}P_k) = k \; \text{ for all } k\in K.
$$
\item\label{stable_manifold_orbit_positivity} If $\ii\limitproj{N}{k}\in\Orbit_{\bfomega{k}}^{>0}$ for all $k\in K$, then the stable manifold $\stableorbit{N}{\bflambda}$ contains $\Orbit_{\bflambda}^{\ge 0}$.
\end{enumerate}

\end{prop}

\begin{proof}
By the observations of \cref{kahler_flow_decomposition}, we see that
$$
L\in\stableorbit{N}{\bflambda} \quad \iff \quad \ii P_k \in \stableorbit{N}{\bfomega{k}} \text{ for all } k \in K.
$$
Also, by \eqref{defn_tnn_Fl_surjections}, if $L\in\Orbit_{\bflambda}^{\ge 0}$, then $\ii P_k\in\Orbit_{\bfomega{k}}^{\ge 0}$ for all $k\in K$. The results then follow from \cref{stable_manifold_grassmannian}.
\end{proof}

\begin{rmk}\label{lasalle}
In \cref{stable_manifold_orbit}\ref{stable_manifold_orbit_condition}, we have given an explicit description of the stable manifold $\stableorbit{N}{\bflambda}$. If we only wish to know that $\stableorbit{N}{\bflambda}$ contains $\Orbit_{\bflambda}^{\ge 0}$ when the gradient flow on $\Orbit_{\bflambda}$ with respect to $N$ in the K\"{a}hler metric strictly preserves positivity, then the following alternative proof suffices. Let $S\subseteq\Orbit_{\bflambda}$ denote the complement of the set of equilibrium points other than $\limitorbit{N}{\bflambda}$. By \cref{positivity_preserving_kahler_grassmannian}, we can argue (e.g.\ using Perron--Frobenius theory) that $S$ contains $\Orbit_{\bflambda}^{\ge 0}$. Then LaSalle's invariance principle \cite[Section 9.2]{hirsch_smale_devaney13} along with \cref{kahler_lyapunov} imply that $\stableorbit{N}{\bflambda}$ contains $\Orbit_{\bflambda}^{\ge 0}$.
\end{rmk}

\subsection{Lyapunov function}\label{sec_lyapunov}
In this subsection, we show that $-\kappa(\cdot,N)$ is a {\itshape Lyapunov function} for $\limitorbit{N}{\bflambda}$, for the gradient flow on $\Orbit_{\bflambda}$ with respect to $N$ in the K\"{a}hler metric. While this essentially follows from the fact that the flow is the gradient flow of the function $\kappa(\cdot,N)$, we also give an elementary direct proof using the explicit description of the flow in \cref{kahler_flow_decomposition}. We refer to \cite[Section 9]{hirsch_smale_devaney13} and \cite[Section 4.3]{abraham_marsden_ratiu88} for further background on Lyapunov stability theory.
\begin{defn}\label{defn_lyapunov}
Consider a flow defined on a differentiable manifold $R$, and let $M\in R$ be an equilibrium point. A {\itshape strict Lyapunov function} for $M$ is a differentiable function $V : S \to \mathbb{R}$, where $S\subseteq R$ is an open subset containing $M$, satisfying the following two properties:
\begin{enumerate}[label={(L\arabic*)}, leftmargin=36pt, itemsep=2pt]
\item\label{defn_lyapunov_minimum} $V(L) > V(M)$ for all $L\neq M$ in $S$; and
\item\label{defn_lyapunov_decreasing} $\frac{d}{dt}\eval{t=0}V(L(t)) < 0$ for all $L_0\neq M$ in $S$, where $L(t)$ denotes the flow beginning at $L_0$.
\end{enumerate}
The existence of a strict Lyapunov function for the equilibrium point $M$ implies that it is {\itshape asymptotically stable} \cite[Section 9.2]{hirsch_smale_devaney13}.
\end{defn}

We observe that if $S$ has a Riemannian metric $\langle\cdot,\cdot\rangle_{\textnormal{metric}}$, then
$$
\textstyle\frac{d}{dt}\eval{t=0}V(L(t)) = \langle\grad(V)(L),\dot{L}(0)\rangle_{\textnormal{metric}}.
$$
In particular, for the gradient flow of the function $-V$, i.e.,
$$
\dot{L}(t) = \grad(-V)(L(t)),
$$
\ref{defn_lyapunov_decreasing} is always satisfied for non-equilibrium points $L_0$. Therefore $V$ is a strict Lyapunov function for $M$ on the stable manifold of $M$ (cf.\ \cite[Section 9.3]{hirsch_smale_devaney13}).

We now prove a slightly stronger statement in the case of gradient flows on $\Orbit_{\bflambda}$ in the K\"{a}hler metric. Our proof of \ref{defn_lyapunov_decreasing} will use the explicit description of the flows, rather than the fact it is gradient. We adapt an argument of Bloch, Brockett, and Ratiu \cite[p.\ 70]{bloch_brockett_ratiu92} for double-bracket flows (i.e.\ gradient flows in the normal metric).
\begin{prop}\label{kahler_lyapunov}
Let $\bflambda\in\mathbb{R}^n$ be weakly decreasing, set $K := \{i \in [n-1] : \lambda_i > \lambda_{i+1}\}$, and let $-N\in\Orbit_{\bfmu}$ such that $\mu_k > \mu_{k+1}$ for all $k\in K$. Consider the gradient flow on $\Orbit_{\bflambda}$ with respect to $N$ in the K\"{a}hler metric, and let $S\subseteq\Orbit_{\bflambda}$ be the complement of the set of equilibrium points other than $\limitorbit{N}{\bflambda}$. (In particular, $S$ contains the stable manifold $\stableorbit{N}{\bflambda}$.) Then
$$
V : S \to \mathbb{R}, \quad L \mapsto -\kappa(L,N)
$$
is a strict Lyapunov function for $\limitorbit{N}{\bflambda}$ on $S$.
\end{prop}

\begin{proof}
We must verify the two conditions of \cref{defn_lyapunov}. First we consider \ref{defn_lyapunov_minimum}. We claim that in fact $V(L) > V(\limitorbit{N}{\bflambda})$ for all $L\neq\limitorbit{N}{\bflambda}$ in $\Orbit_{\bflambda}$. This essentially follows from a theorem of Schur \cite{schur23} (one direction of the {\itshape Schur--Horn theorem}); we give a detailed argument below.

As in \eqref{projection_sum_formula}, let us write
$$
-\ii L = \Big(\sum_{k\in K}(\lambda_k - \lambda_{k+1})P_k\Big) + \lambda_n\I_n,
$$
where $P_k^2 = P_k = P_k^*$ and $\tr(P_k) = k$. Recall the analogous expansion \eqref{limit_decomposition} of $\limitorbit{N}{\bflambda}$. We begin by proving that
\begin{align}\label{kahler_lyapunov_trace_inequality}
\tr(\limitproj{N}{k}\ii N) \ge \tr(P_k\ii N) \quad \text{ for all } k \in K.
\end{align}
Let us work in an orthonormal basis of eigenvectors of $\ii N$ corresponding to the eigenvalues $\mu_1 \ge \dots \ge \mu_n$, so that $\ii N = \Diag{\bfmu}$ and $\limitproj{N}{k} = \scalebox{0.8}{$\begin{bmatrix}\I_k & 0 \\ 0 & 0\end{bmatrix}$}$. Then \eqref{kahler_lyapunov_trace_inequality} becomes
$$
\mu_1 + \cdots + \mu_k \ge (P_k)_{1,1}\mu_1 + \cdots + (P_k)_{n,n}\mu_n.
$$
By assumption, the diagonal entries of $P_k$ lie in the interval $[0,1]$ and sum to $k$. Therefore we obtain \eqref{kahler_lyapunov_trace_inequality}. Moreover, since $\mu_k > \mu_{k+1}$, the inequality is strict if $P_k \neq \limitproj{N}{k}$; and the latter condition holds for some $k\in K$, because $L\neq\limitorbit{N}{\bflambda}$. Multiplying \eqref{kahler_lyapunov_trace_inequality} by $\lambda_k - \lambda_{k+1}$ and summing over $k$, we obtain $\tr(\limitorbit{N}{\bflambda}N) > \tr(LN)$, which is equivalent to the desired inequality $V(L) > V(\limitorbit{N}{\bflambda})$.

We now prove that \ref{defn_lyapunov_decreasing} holds for $L\neq \limitorbit{N}{\bflambda}$ in $S$. Let us expand $-\ii L(t)$ as in \eqref{kahler_flow_decomposition_equation}. Set $P_k := P_k(0)$ for $k\in K$. Then by \cref{kahler_flow_decomposition}, we have
\begin{multline*}
\textstyle\frac{d}{dt}\eval{t=0}V(L(t)) = -\kappa(\dot{L}(0),N) = -\displaystyle\sum_{k\in K}(\lambda_k - \lambda_{k+1})\kappa(\ii \dot{P}_k(0),N) \\
= -\sum_{k\in K}(\lambda_k - \lambda_{k+1})\kappa([\ii P_k,[\ii P_k,N]],N) = \sum_{k\in K}(\lambda_k - \lambda_{k+1})\kappa([\ii P_k, N], [\ii P_k, N]),
\end{multline*}
where in the last step we used the fact that $\kappa$ is $[\cdot,\cdot]$-invariant.

Since $-\kappa$ is positive semidefinite, we have $\kappa([\ii P_k, N], [\ii P_k, N]) \le 0$ for all $k\in K$. Moreover, since $\dot{L}(0) \neq 0$, we have $\ii \dot{P}_k(0) \neq 0$ for some $k\in K$; then $[\ii P_k, N] \neq 0$, and so $\kappa([\ii P_k, N], [\ii P_k, N]) < 0$. Therefore $\textstyle\frac{d}{dt}\eval{t=0}V(L(t)) < 0$.
\end{proof}

We will need the following consequence of \cref{kahler_lyapunov} in \cref{sec_homeomorphism}:
\begin{cor}\label{boundary_time}
Adopt the notation and assumptions of \cref{kahler_lyapunov}. Let $S_0$ be a compact subset of $S$. Then for any gradient flow $L(t)$ in $\Orbit_{\bflambda}$ which is not the constant flow at $\limitorbit{N}{\bflambda}$, we have $L(t)\notin S_0$ for some $t \le 0$.
\end{cor}

\begin{proof}
By \cref{kahler_lyapunov}, $V : S \to \mathbb{R}$ is a strict Lyapunov function for $\limitorbit{N}{\bflambda}$. Let $S_1 := V^{-1}([V(L_0),\infty))\cap S_0$, which is compact since $V$ is continuous. By \ref{defn_lyapunov_decreasing}, for $t \le 0$ we have $V(L(t)) \ge V(L_0)$, so if $L(t)\in S_0$ then $L(t)\in S_1$. Hence it suffices to show that $L(t)\notin S_1$ for some $t \le 0$.

We proceed by contradiction and suppose that $L(t)\in S_1$ for all $t \le 0$. For $M\in\Orbit_{\bflambda}$, let $M(t)\in\Orbit_{\bflambda}$ denote the gradient flow beginning at $M$. Define $c\in\mathbb{R}$ to be the minimum of $-\frac{d}{dt}\eval{t=0}V(M(t))$ over all $M$ in the compact set $S_1$. By \ref{defn_lyapunov_minimum} we have $\limitorbit{N}{\bflambda}\notin S_1$, so \ref{defn_lyapunov_decreasing} implies that $c > 0$. By \ref{defn_lyapunov_decreasing}, we obtain
$$
V(L(t)) \ge V(L_0) - ct \quad \text{ for all } t \le 0.
$$
Therefore $V$ is unbounded on the compact set $S_1$, a contradiction.
\end{proof}

\subsection{Homeomorphism onto a closed ball}\label{sec_homeomorphism}
We now use gradient flows to show that the totally nonnegative part $\Orbit_{\bflambda}^{\ge 0}$ of an adjoint orbit is homeomorphic to a closed ball. As we have mentioned, this result was proved by Galashin, Karp, and Lam \cite{galashin_karp_lam19} in general Lie type, which we rephrase in type $A$ in the orbit language. We adopt the framework of {\itshape contractive flows} developed in \cite[Section 2]{galashin_karp_lam}; the main modification is that we use a Lyapunov function in place of the Euclidean norm employed in \cite{galashin_karp_lam}. We deduce the result about $\Orbit_{\bflambda}^{\ge 0}$ as a consequence of the more general \cref{compact_invariant_ball}, which we will also use to show that the Pl\"{u}cker-nonnegative part of a partial flag variety is homeomorphic to a closed ball (see \cref{plucker_ball}), and to study the topology of amplituhedra (see \cref{sec_amplituhedra_ball}).

We will need a continuity result for gradient flows on $\Orbit_{\bflambda}$, which follows from general principles. In the case relevant to us, namely for the K\"{a}hler metric, it also follows from the explicit formula \eqref{gradient_flow_kahler_equation}.
\begin{lem}[{\cite[Proposition 4.1.17(iii)]{abraham_marsden_ratiu88}}]\label{flow_continuity}
Consider a gradient flow \eqref{defn_gradient_flow_equation} on $\Orbit_{\bflambda}$. For $L\in\Orbit_{\bflambda}$, let $L(t)\in\Orbit_{\bflambda}$ denote the gradient flow beginning at $L$. Then the function
$$
\mathbb{R}\times\Orbit_{\bflambda} \to \Orbit_{\bflambda}, \quad (t,L) \mapsto L(t)
$$
is continuous.
\end{lem}

\begin{thm}\label{compact_invariant_ball}
Let $\bflambda\in\mathbb{R}^n$ be weakly decreasing, set $K := \{i \in [n-1] : \lambda_i > \lambda_{i+1}\}$, and let $-N\in\Orbit_{\bfmu}$ such that $\mu_k > \mu_{k+1}$ for all $k\in K$. Consider the gradient flow on $\Orbit_{\bflambda}$ with respect to $N$ in the K\"{a}hler metric. Let $S$ be a nonempty compact subset of the stable manifold $\stableorbit{N}{\bflambda}$, and let $\interior{S}$ denote the interior of $S$ inside $\Orbit_{\bflambda}$. Suppose that any flow beginning in $S$ lies in $\interior{S}$ for all positive time. Then $S$ is homeomorphic to a closed ball, $\interior{S}$ is homeomorphic to an open ball, and its boundary $S\setminus\interior{S}$ is homeomorphic to a sphere.
\end{thm}

\begin{proof}
We closely follow the proof of \cite[Lemma 2.3]{galashin_karp_lam}. Let $V : \stableorbit{N}{\bflambda} \to \mathbb{R}$ denote the strict Lyapunov function for $\limitorbit{N}{\bflambda}$ from \cref{kahler_lyapunov}. Define the function
$$
\nu : \stableorbit{N}{\bflambda} \to \mathbb{R}, \quad L \mapsto V(L) - V(\limitorbit{N}{\bflambda}).
$$
In particular, $\nu$ is nonnegative and equals zero precisely at $\limitorbit{N}{\bflambda}$. If $L(t)$ is the gradient flow beginning at any point of $\stableorbit{N}{\bflambda}$ other than $\limitorbit{N}{\bflambda}$, then $\nu(L(t))$ is strictly decreasing as a function of $t$ and approaches $0$ as $t\to\infty$.

For $r > 0$, define $B_r := \norm^{-1}([0,r]) \cap \stableorbit{N}{\bflambda}$. By assumption, $\limitorbit{N}{\bflambda}$ is contained in $S$, and therefore also in $\interior{S}$. By the Morse lemma (cf.\ \cite[Lemma 5.4.9]{abraham_marsden_ratiu88}, \cite{durfee83}), we may take $r$ sufficiently small that $B_r$ is contained in $\interior{S}$ and is homeomorphic to a closed ball, $\interior{B_r} = \norm^{-1}([0,r)) \cap \stableorbit{N}{\bflambda}$ is homeomorphic to an open ball, and $B_r\setminus\interior{B_r} = \norm^{-1}(r) \cap \stableorbit{N}{\bflambda}$ is homeomorphic to a sphere. (In fact, by letting the gradient flow act on $B_r$, we get that $B_r$ is homeomorphic to a closed ball for all $r > 0$, though we will not need to use this.)

We now define two functions $t_r, t_\partial : \stableorbit{N}{\bflambda}\setminus\{\limitorbit{N}{\bflambda}\} \to \mathbb{R}$, as follows. Given $L\in\stableorbit{N}{\bflambda}\setminus\{\limitorbit{N}{\bflambda}\}$, let $L(t)\in\Orbit_{\bflambda}$ denote the gradient flow beginning at $L$. By \cref{boundary_time}, there exists $t_0\in\mathbb{R}$ such that $L(t_0)\notin S$. In particular, $\nu(L(t_0)) > r$. Since $\nu(L(t))$ is strictly decreasing as a function of $t$ and approaches $0$ as $t\to\infty$, there exists a unique $t\in\mathbb{R}$ such that $\nu(L(t)) = r$, which we define to be $t_r(L)$. Now observe that by assumption, we have $L(t)\notin S$ for all $t \le t_0$, and we also have $L(t_r(L))\in S$. Therefore we may define $t_\partial(L) := \inf\{t\in\mathbb{R} : L(t)\in S\}$. Again by assumption, we have $L(t_\partial(L))\in S\setminus\interior{S}$ and $L(t)\in\interior{S}$ for all $t > t_\partial(L)$.

We claim that $t_r$ and $t_\partial$ are continuous functions on $\stableorbit{N}{\bflambda}\setminus\{\limitorbit{N}{\bflambda}\}$. First we prove that $t_r$ is continuous. It suffices to show that given an open interval $I\subseteq\mathbb{R}$, the preimage $t_r^{-1}(I)$ is open. To this end, let $L\in t_r^{-1}(I)$, and let $L(t)\in\Orbit_{\bflambda}$ denote the gradient flow beginning at $L$. Take $t_1,t_2\in I$ such that $t_1 < t_r(L) < t_2$. Let $r_1 := \nu(L(t_1))$ and $r_2 := \nu(L(t_2))$, so that $r_1 > r > r_2$ by \ref{defn_lyapunov_decreasing}. For $M\in\Orbit_{\bflambda}$, let $M(t)\in\Orbit_{\bflambda}$ denote the gradient flow starting at $M$. By \cref{flow_continuity}, the function $M\mapsto \nu(M(t_1))$ is continuous. Hence there exists an open neighborhood $U_1$ of $L$ such that for all $M\in U_1$, we have $\nu(M(t_1)) > r$. Similarly, there exists an open neighborhood $U_2$ of $L$ such that for all $M\in U_2$, we have $r > \nu(M(t_2))$. Let $U := U_1 \cap U_2$, which is an open neighborhood of $L$. Then for all $M\in U$, we have $\nu(M(t_1)) > r > \nu(M(t_2))$, so \ref{defn_lyapunov_decreasing} implies that $t_r(M) \in (t_1, t_2) \subseteq I$. That is, $U\subseteq t_r^{-1}(I)$, and hence $t_r^{-1}(I)$ is open.

Now we prove that $t_\partial$ is continuous, by a similar argument. Let $L\in t_\partial^{-1}(I)$, where $I\subseteq\mathbb{R}$ is an open interval, and take $t_1, t_2 \in I$ such that $t_1 < t_\partial(L) < t_2$. Observe that $L(t_1)\in\stableorbit{N}{\bflambda}\setminus S$ and $L(t_2)\in\interior{S}$, and that both sets $\stableorbit{N}{\bflambda}\setminus S$ and $\interior{S}$ are open. Hence there exists an open neighborhood $U_1$ of $L$ such that for all $M\in U_1$, we have $M(t_1)\in\stableorbit{N}{\bflambda}\setminus S$. Similarly, there exists an open neighborhood $U_2$ of $L$ such that for all $M\in U_2$, we have $M(t_2)\in\interior{S}$. Then $U := U_1\cap U_2$ is an open neighborhood of $L$ contained in $t_\partial^{-1}(I)$. Thus $t_\partial$ is continuous.

We now define maps $\alpha : S \to B_r$ and $\beta : B_r \to S$ as follows. If $L\neq\limitorbit{N}{\bflambda}$, we set
$$
\alpha(L) := L(t_r(L) - t_\partial(L)) \quad \text{ and } \quad \beta(L) := L(t_\partial(L) - t_r(L)),
$$
where $L(t)\in\Orbit_{\bflambda}$ denotes the gradient flow beginning at $L$. We also set $\alpha(\limitorbit{N}{\bflambda}) := \limitorbit{N}{\bflambda}$ and $\beta(\limitorbit{N}{\bflambda}) := \limitorbit{N}{\bflambda}$. We can verify that $\alpha$ and $\beta$ are well-defined, and that they are inverses of each other. Also note that $\alpha(S\setminus S^\circ)\subseteq B_r\setminus B_r^\circ$ and $\beta(B_r\setminus B_r^\circ)\subseteq S\setminus S^\circ$. Thus $\alpha$ restricts to a bijection from $S\setminus S^\circ$ to $B_r\setminus B_r^\circ$, and hence also restricts to a bijection from $S^\circ$ to $B_r^\circ$.

Therefore to complete the proof, it suffices to show that $\alpha$ and $\beta$ are continuous. We prove that $\alpha$ is continuous; because $S$ is compact, this then implies that $\beta = \alpha^{-1}$ is continuous. By \cref{flow_continuity} and since $t_r$ and $t_\partial$ are continuous, we have that $\alpha$ is continuous except possibly at $\limitorbit{N}{\bflambda}$. Now observe that every open neighborhood of $\limitorbit{N}{\bflambda}$ in $B_r$ contains the open subset $\nu^{-1}([0,\varepsilon))$ for some $\varepsilon > 0$. By \ref{defn_lyapunov_decreasing} we have $\alpha(\nu^{-1}([0,\varepsilon)) \subseteq \nu^{-1}([0,\varepsilon))$, so $\alpha$ is continuous at $\limitorbit{N}{\bflambda}$.
\end{proof}

\begin{cor}[{Galashin, Karp, and Lam \cite[Theorem 1]{galashin_karp_lam19}}]\label{orbit_ball}
Let $\bflambda\in\mathbb{R}^n$ be weakly decreasing. Then $\Orbit_{\bflambda}^{\ge 0}$ is homeomorphic to a closed ball, its interior $\Orbit_{\bflambda}^{>0}$ is homeomorphic to an open ball, and its boundary $\Orbit_{\bflambda}^{\ge 0}\setminus\Orbit_{\bflambda}^{>0}$ is homeomorphic to a sphere.
\end{cor}

The fact that $\Orbit_{\bflambda}^{>0}$ is homeomorphic to an open ball was originally proved by Rietsch \cite[Theorem 2.8]{rietsch99}.

\begin{proof}
We apply \cref{compact_invariant_ball}, taking $S$ to be $\Orbit_{\bflambda}^{\ge 0}$, and taking $-N \in \Orbit_{\bfmu}$ such that the gradient flow on $\Orbit_{\bflambda}$ with respect to $N$ in the K\"{a}hler metric strictly preserves positivity. (For example, we may take $\ii N\in\gl_n^{>0}$, by \cref{positivity_preserving_kahler_grassmannian} and \cref{positivity_preserving_kahler_full}.) Let us verify that the hypotheses of \cref{compact_invariant_ball} are satisfied. Setting $K := \{i \in [n-1] : \lambda_i > \lambda_{i+1}\}$, we have $\mu_k > \mu_{k+1}$ for all $k \in K$ by \cref{eigenvalue_condition_satisfied}. Also, $S$ is compact since it is a closed subset of the compact space $\Orbit_{\bflambda}$, and $S$ is contained in $\stableorbit{N}{\bflambda}$ by \cref{stable_manifold_orbit}\ref{stable_manifold_orbit_positivity} (using \cref{eigenvalue_condition_satisfied}).
\end{proof}

\begin{rmk}
In subsequent work, Galashin, Karp, and Lam \cite[Theorem 1.1]{galashin_karp_lam22} proved the stronger result that the cell decomposition \eqref{cell_decomposition_equation} (as well as its analogue in general Lie type) is a regular CW complex, confirming a conjecture of Williams \cite[Section 7]{williams07}. In particular, the closure of each cell $C_{v,w}$ is homeomorphic to a closed ball, and its boundary is homeomorphic to a sphere. The arguments employed in \cite{galashin_karp_lam22} are different than those of \cite{galashin_karp_lam, galashin_karp_lam19}, and in particular do not employ contractive flows. It would be very interesting to find a proof that \eqref{cell_decomposition_equation} is a regular CW complex along the lines of the arguments in this section.
\end{rmk}

Recall the Pl\"{u}cker-nonnegative part $\PFl{K}{n}^{\Delta\ge 0}$ of $\PFl{K}{n}(\mathbb{C})$ from \cref{defn_plucker_positive}. We now use \cref{compact_invariant_ball} to show that $\PFl{K}{n}^{\Delta\ge 0}$ is homeomorphic to a closed ball. We remark that Rietsch \cite[Lemma 5.2]{rietsch98} used a similar construction to show that $\PFl{K}{n}^{\Delta >0}$ is contractible. We will need the following result from \cite{bloch_karp2}:
\begin{lem}[{Bloch and Karp \cite{bloch_karp2}}]\label{plucker_positive_interior}
Let $K\subseteq [n-1]$. Then $\PFl{K}{n}^{\Delta > 0}$ is the interior of $\PFl{K}{n}^{\Delta \ge 0}$.
\end{lem}

\begin{cor}\label{plucker_ball}
Let $K\subseteq [n-1]$. Then $\PFl{K}{n}^{\Delta\ge 0}$ is homeomorphic to a closed ball, its interior $\PFl{K}{n}^{\Delta > 0}$ is homeomorphic to an open ball, and its boundary $\PFl{K}{n}^{\Delta\ge 0}\setminus\PFl{K}{n}^{\Delta > 0}$ is homeomorphic to a sphere.
\end{cor}

\begin{proof}
Take $\bflambda\in\mathbb{R}^n$ weakly decreasing such that $K = \{i\in [n-1] : \lambda_i > \lambda_{i+1}\}$. We apply \cref{compact_invariant_ball}, taking $S\subseteq\Orbit_{\bflambda}$ to be the image of $\PFl{K}{n}^{\Delta\ge 0}$ under \eqref{Fl_to_orbit_map}, and taking $-N \in \Orbit_{\bfmu}$ such that the gradient flow on $\Orbit_{\bfomega{k}}$ with respect to $N$ in the K\"{a}hler metric strictly preserves positivity for all $k\in K$. (For example, we may take $\ii N\in\gl_n^{>0}$, by \cref{positivity_preserving_kahler_grassmannian}.) Let us verify that the hypotheses of \cref{compact_invariant_ball} are satisfied. We have $\mu_k > \mu_{k+1}$ for all $k \in K$, by \cref{eigenvalue_condition_satisfied} applied to $\Orbit_{\bfomega{k}}$. Also, $S$ is compact since it is a closed subset of the compact space $\Orbit_{\bflambda}$, and $S$ is contained in $\stableorbit{N}{\bflambda}$ by applying \cref{stable_manifold_grassmannian}\ref{stable_manifold_grassmannian_positivity} for all $k\in K$ (using \cref{eigenvalue_condition_satisfied}). By \cref{plucker_positive_interior}, $\interior{S}$ is the image of $\PFl{K}{n}^{\Delta >0}$ under \eqref{Fl_to_orbit_map}. Therefore any flow beginning in $S$ remains in $\interior{S}$ for all positive time.  \end{proof}

\begin{rmk}\label{symmetric_to_general}
Recall from \cref{gradient_flow_kahler} that the gradient flow on $\Orbit_{\bflambda}$ with respect to $N$ in the K\"{a}hler metric corresponds to the flow $V(t) = \exp(t\ii N)V_0$ on the $\PFl{K}{n}(\mathbb{C})$. We can consider the same flow on $\PFl{K}{n}(\mathbb{C})$ with $\ii N$ replaced by any $M\in\gl_n(\mathbb{R})$ (not necessarily symmetric), and much of the analysis of this section can be replicated in this case. We do not pursue this here, since it is outside the scope of adjoint orbits.
\end{rmk}

\section{Gradient flows on amplituhedra}\label{sec_amplituhedron}

\noindent In this section we study gradient flows on the {\itshape amplituhedron} $\mathcal{A}_{n,k,m}(Z)$, a subset of the Grassmannian $\Gr_{k,k+m}(\mathbb{C})$ defined in terms of an auxiliary matrix $Z$ (see \cref{defn_amplituhedron}). It generalizes both the totally nonnegative Grassmannian $\Gr_{k,n}^{\ge 0}$ (which we obtain when $k+m = n$) and a cyclic polytope (which we obtain when $k=1$). Amplituhedra were introduced by Arkani-Hamed and Trnka \cite{arkani-hamed_trnka14} in order to give a geometric basis for calculating scattering amplitudes in planar $\mathcal{N}=4$ supersymmetric Yang--Mills theory. The case relevant for physics is when $m=4$, but amplituhedra are interesting mathematical objects for any $m$.

There has been a lot of work studying the geometric properties of amplituhedra, including determining the homeomorphism type. It is expected that $\mathcal{A}_{n,k,m}(Z)$ is homeomorphic to a closed ball of dimension $km$. This is known when $k+m = n$ \cite[Theorem 1.1]{galashin_karp_lam} (since every such amplituhedron is homeomorphic to $\Gr_{k,n}^{\ge 0}$), when $k=1$ (since every convex polytope is homeomorphic to a closed ball), when $m=1$ \cite[Corollary 6.18]{karp_williams19} (cf.\ \cite[Corollary 1.2]{karp_machacek}), for the family of {\itshape cyclically symmetric amplituhedra} \cite[Theorem 1.2]{galashin_karp_lam}, and when $n-k-m = 1$ with $m$ even \cite[Theorem 1.8]{blagojevic_galashin_palic_ziegler19}.

We show that a new family of amplituhedra are also homeomorphic to closed balls, which we call {\itshape twisted Vandermonde amplituhedra} (see \cref{twisted_vandermonde_amplituhedra_ball}). This family includes all amplituhedra with $n-k-m \le 2$ (see \cref{amplituhedra_ball_case_ball}). Our argument is based on the proof of \cite[Theorem 1.2]{galashin_karp_lam}, which uses contractive flows to show that cyclically symmetric amplituhedra are homeomorphic to closed balls. (However, we note that the family of twisted Vandermonde amplituhedra does not include the cyclically symmetric amplituhedra; see \cref{vandermonde_flags_corner_remark} for further discussion.)

\subsection{Background}\label{sec_amplituhedron_background}
We now define amplituhedra.
\begin{defn}\label{defn_amplituhedron}
Let $n,k,m\in\mathbb{N}$ such that $k+m \le n$, and let $Z$ be a complex $(k+m)\times n$ matrix of rank $k+m$. We also regard $Z$ as a linear map $\mathbb{C}^n \to \mathbb{C}^{k+m}$. We introduce the rational map
\begin{align}\label{defn_amplituhedron_projection}
\tilde{Z} : \Gr_{k,n}(\mathbb{C}) \dashrightarrow \Gr_{k,k+m}(\mathbb{C}), \quad V \mapsto \{Z(v) : v\in V\},
\end{align}
which is defined whenever $V\cap\ker(Z) = \{0\}$.

Now suppose that $Z$ is real and its $(k+m)\times (k+m)$ minors are all positive. Then by \cite[Section 4]{arkani-hamed_trnka14} (cf.\ \cite[Section 4]{karp17}), $\tilde{Z}$ is defined on $\Gr_{k,n}^{\ge 0}$. We denote the image $\tilde{Z}(\Gr_{k,n}^{\ge 0})$ by $\mathcal{A}_{n,k,m}(Z)$, called a {\itshape (tree) amplituhedron}.
\end{defn}

In \eqref{defn_amplituhedron_projection}, $V$ is a $k$-dimensional subspace of $\mathbb{C}^n$. If we instead regard $V$ as an $n\times k$ matrix modulo column operations, then $\tilde{Z}(V) = ZV$.

We point out two special cases of \cref{defn_amplituhedron}. First, if $k+m = n$, then up to a linear change of coordinates, we may assume that $Z = \I_n$, so that $\mathcal{A}_{n,k,m}(Z)$ is the totally nonnegative Grassmannian $\Gr_{k,n}^{\ge 0}$. Second, if $k=1$, then it follows from work of Sturmfels \cite{sturmfels88} that $\mathcal{A}_{n,k,m}(Z)$ is an {\itshape alternating polytope} (a special kind of cyclic polytope) in $\mathbb{P}^m(\mathbb{C})$.

\begin{rmk}\label{amplituhedron_generalizations}
We note that \cref{defn_amplituhedron} can be generalized in various ways. The tree amplituhedron $\mathcal{A}_{n,k,m}(Z)$ is related to the tree-level term of the scattering amplitude; there are also {\itshape loop amplituhedra} corresponding to the higher-order terms of the amplitude \cite{arkani-hamed_trnka14} (cf.\ \cref{naive_part}). Alternatively, we can relax the condition that $Z$ has positive $(k+m)\times (k+m)$ minors, or replace $\Gr_{k,n}^{\ge 0}$ by the closure of a cell in its cell decomposition. The corresponding image under $\tilde{Z}$ is called a {\itshape Grassmann polytope}, studied by Lam \cite{lam16}. Yet another generalization is provided by replacing $\Gr_{k,n}(\mathbb{C})$ by an arbitrary partial flag variety $\PFl{K}{n}(\mathbb{C})$, giving the {\itshape flag polytopes} introduced by Arkani-Hamed, Bai, and Lam \cite[Section 6.5]{arkani-hamed_bai_lam17}. While we will focus on the case of tree amplituhedra $\mathcal{A}_{n,k,m}(Z)$, many of the results and techniques in this section apply more generally.
\end{rmk}

\begin{eg}\label{eg_amplituhedron}
Let $n := 4$ and $k+m := 3$, and take $Z$ to be the matrix
$$
Z := \begin{bmatrix}
1 & 0 & 0 & a \\
0 & 1 & 0 & -b \\
0 & 0 & 1 & c
\end{bmatrix}, \quad \text{ where } a, b, c > 0.
$$
Note that the $3\times 3$ minors of $Z$ are all positive, so $Z$ defines an amplituhedron $\mathcal{A}_{4,k,m}(Z) = \tilde{Z}(\Gr_{k,4}^{\ge 0})$. When $k=1$ and $m=2$, the map $\tilde{Z} : \Gr_{1,4}^{\ge 0} \to \Gr_{1,3}(\mathbb{C})$ is given by
\begin{multline*}
(x_0 : x_1 : x_2 : x_3) \in \mathbb{P}^3_{\ge 0} \;\mapsto\; \\
x_0(1 : 0 : 0) + x_1(0 : 1 : 0) + x_2(0 : 0 : 1) + x_3(a : -b : c) \in \mathbb{P}^2(\mathbb{C}),
\end{multline*}
and $\mathcal{A}_{4,1,2}(Z)$ is the quadrilateral in $\mathbb{P}^2(\mathbb{C})$ with vertices $(1 : 0 : 0)$, $(0 : 1 : 0)$, $(0 : 0 : 1)$, and $(a : -b: c)$. When $k=2$ and $m=1$, by work of Karp and Williams \cite[Theorem 6.16]{karp_williams19}, we can identify the amplituhedron $\mathcal{A}_{4,2,1}(Z)$ with the bounded complex of a {\itshape cyclic hyperplane arrangement} of $4$ hyperplanes in $\mathbb{R}^2$.
\end{eg}

\begin{rmk}\label{amplituhedron_orthonormal_basis}
Let $Z$ be a complex $(k+m)\times n$ matrix of rank $k+m$, let $g\in\GL_{k+m}(\mathbb{C})$, and set $Z' := gZ$. Then $\tilde{Z'} = g\tilde{Z}$, so the rational map $\tilde{Z}$ only depends on $\ker(Z)$ (or equivalently, the row span of $Z$), up to a linear change of coordinates on $\mathbb{C}^{k+m}$. In particular, we may assume (as it will turn out to be convenient) that the rows of $Z$ are orthonormal, i.e., $Z\adjoint{Z} = \I_{k+m}$. Further, if $Z$ is real and its $(k+m)\times (k+m)$ minors are all positive, and $g$ is real with $\det(g) > 0$, then $Z'$ is real and its $(k+m)\times (k+m)$ minors are all positive. Therefore the amplituhedron $\mathcal{A}_{n,k,m}(Z) \subseteq \Gr_{k,k+m}(\mathbb{C})$ only depends on $\ker(Z)$, and we may assume that $Z\transpose{Z} = \I_{k+m}$.
\end{rmk}

\subsection{Projecting gradient flows}\label{sec_amplituhedra_gradient_projection}

In this subsection, we determine when the rational map $\tilde{Z} : \Gr_{k,n}(\mathbb{C}) \dashrightarrow \Gr_{k,k+m}(\mathbb{C})$ from \eqref{defn_amplituhedron_projection} projects gradient flows on $\Gr_{k,n}(\mathbb{C})$ in a coherent way, where we identify $\Gr_{k,n}(\mathbb{C})$ and $\Gr_{k,k+m}(\mathbb{C})$ with adjoint orbits via \eqref{Fl_to_orbit_map}. By this, we mean that for any two points $V,W\in\Gr_{k,n}(\mathbb{C})$ such that $\tilde{Z}(V) = \tilde{Z}(W)$, the gradient flows beginning at $V$ and $W$ have the same image under $\tilde{Z}$. It turns out that if this is the case, then up to a linear change of coordinates (cf.\ \cref{amplituhedron_orthonormal_basis}), the projected gradient flows are also gradient flows on $\Gr_{k,k+m}(\mathbb{C})$. Since we are working with Grassmannians, the three metrics discussed in \cref{sec_gradient} are the same up to dilation (see \cref{grassmannian_metrics}). We will find it most convenient to use the description of the gradient flows given in \cref{gradient_flow_kahler}.

We will use the following description of the fibers of $\tilde{Z}$; see \cite[Proposition 3.12]{karp_williams19} for a closely related result.
\begin{lem}\label{amplituhedron_projection_equality}
Let $Z$ be a complex $(k+m)\times n$ matrix of rank $k+m$, and let $V,W\in\Gr_{k,n}(\mathbb{C})$ such that $\tilde{Z}(V)$ and $\tilde{Z}(W)$ are defined. Then
\begin{gather*}
\tilde{Z}(V) = \tilde{Z}(W) \quad \iff \quad V + \ker(Z) = W + \ker(Z).
\end{gather*}

\end{lem}

\begin{proof}
($\Rightarrow$) Suppose that $\tilde{Z}(V) = \tilde{Z}(W)$. We show that $V \subseteq W + \ker(Z)$; we similarly have $W \subseteq V + \ker(Z)$, which implies the result. To this end, let $v\in V$. Then $Zv \in \tilde{Z}(V) = \tilde{Z}(W)$, so $Zv = Zw$ for some $w\in W$. Then $v - w \in \ker(Z)$, so $v \in W + \ker(Z)$.

($\Leftarrow$) Suppose that $V + \ker(Z) = W + \ker(Z)$. We show that $\tilde{Z}(V) \subseteq \tilde{Z}(W)$; we similarly have $\tilde{Z}(W) \subseteq \tilde{Z}(V)$. To this end, let $Z(v)$ be an element of $\tilde{Z}(V)$, where $v\in V$. Then $v = w + x$ for some $w \in W$ and $x\in\ker(Z)$, so $Zv = Zw \in \tilde{Z}(W)$.
\end{proof}

\begin{prop}\label{amplituhedron_projected_flows}
Let $Z$ be a complex $(k+m)\times n$ matrix of rank $k+m$, where $k,m\ge 1$, and let $N\in\uu_n$. Then the following conditions are equivalent.
\begin{enumerate}[label=(\roman*), leftmargin=*, itemsep=2pt]
\item\label{amplituhedron_projected_flows_coherent} The rational map $\tilde{Z} : \Gr_{k,n}(\mathbb{C}) \dashrightarrow \Gr_{k,k+m}(\mathbb{C})$ coherently projects the gradient flows on $\Gr_{k,n}(\mathbb{C})$ with respect to $N$. That is, for all gradient flows $V(t)$ and $W(t)$ in $\Gr_{k,n}(\mathbb{C})$ with respect to $N$ such that $\tilde{Z}(V_0) = \tilde{Z}(W_0)$, we have $\tilde{Z}(V(t)) = \tilde{Z}(W(t))$ for all $t$.
\item\label{amplituhedron_projected_flows_preserving} We have $N(\ker(Z))\subseteq \ker(Z)$.
\item\label{amplituhedron_projected_flows_equation} There exists $M\in\gl_{k+m}(\mathbb{C})$ such that $ZN = MZ$, namely, $M = ZN\adjoint{Z}(Z\adjoint{Z})^{-1}$.
\end{enumerate}

\end{prop}

We observe that in general, the element $M$ in part \ref{amplituhedron_projected_flows_equation} does not necessarily lie in $\uu_{k+m}$. However, under the assumption $Z\adjoint{Z} = \I_{k+m}$ (cf.\ \cref{amplituhedron_orthonormal_basis}), we have $M = ZN\adjoint{Z} \in \uu_{k+m}$.
\begin{proof}
We use the description of the gradient flow with respect to $N$ from \eqref{gradient_flow_kahler_equation}.

\ref{amplituhedron_projected_flows_coherent} $\Rightarrow$ \ref{amplituhedron_projected_flows_preserving}: Suppose that $\tilde{Z}$ coherently projects the gradient flows with respect to $N$. It suffices to prove that $\exp(t\ii N)\ker(Z) = \ker(Z)$ for all $t\in\mathbb{R}$. We will show that given a nonzero $x\in\ker(Z)$, we have $\exp(t\ii N)x \in \ker(Z)$.

To this end, let $V\in\Gr_{k,n}(\mathbb{C})$ such that $V\cap\ker(Z) = \{0\}$, so that $\tilde{Z}(V)$ is defined. Take $W\in\Gr_{k,n}(\mathbb{C})$ such that $W \subseteq V + \spn(x)$, $W \neq V$, and $x\notin W$. Note that $W\cap\ker(Z) = \{0\}$, so $\tilde{Z}(W)$ is defined, and $v + x \in W$ for some $v\in V$. Also, let $T\subseteq\mathbb{C}$ denote the set of $t\in\mathbb{C}$ such that $\tilde{Z}(\exp(t\ii NV))$ is not defined, i.e., $\exp(t\ii N)V\cap\ker(Z) \neq \{0\}$. Viewing $V$ as an $n\times k$ matrix and $\ker(Z)$ as an $n\times (n-k-m)$ matrix, we see that $T$ is the common zero set of the $(n-m)\times (n-m)$ minors of the concatenation of $\exp(t\ii N)V$ and $\ker(Z)$. Each such minor is an analytic function of $t$, and because $0\notin T$, we get that $T$ is discrete.

Since $W + \ker(Z) \subseteq V + \ker(Z)$, we have $\tilde{Z}(V) = \tilde{Z}(W)$ by \cref{amplituhedron_projection_equality}. Therefore by assumption, we have $\tilde{Z}(\exp(t\ii N)V) = \tilde{Z}(\exp(t\ii N)W)$ for all $t\in\mathbb{R}$. Again by \cref{amplituhedron_projection_equality}, we have $\exp(t\ii N)V + \ker(Z) = \exp(t\ii N)W + \ker(Z)$ for all $t\in\mathbb{R}\setminus T$. Multiplying by $\exp(-t\ii N)$, we get
$$
V + \exp(-t\ii N)\ker(Z) = W + \exp(-t\ii N)\ker(Z) \quad \text{ for all } t\in\mathbb{R}\setminus T.
$$
Since $v + x \in W \subseteq V + \exp(-t\ii N)\ker(Z)$, we obtain
$$
x \in V + \exp(-t\ii N)\ker(Z) \quad \text{ for all } t\in\mathbb{R}\setminus T.
$$

The conclusion above holds for all $V\in\Gr_{k,n}(\mathbb{C})$ such that $V\cap\ker(Z) = \{0\}$; considering $k+1$ generic such $V$, since $m\ge 1$ we obtain
$$
x \in \exp(-t\ii N)\ker(Z) \quad \text{ for all } t\in\mathbb{R} \text{ not contained in some discrete set}.
$$
By continuity, we get $x \in \exp(-t\ii N)\ker(Z)$ for all $t\in\mathbb{R}$, as desired.

\ref{amplituhedron_projected_flows_preserving} $\Rightarrow$ \ref{amplituhedron_projected_flows_equation}: Suppose that $N(\ker(Z))\subseteq \ker(Z)$. Since $Z : \mathbb{C}^n \to \mathbb{C}^{k+m}$ is surjective, we can simply define $M$ by $MZx := ZNx$ for all $x\in\mathbb{C}^n$. This is well-defined because if $Zx = 0$, then $ZNx = 0$. 

\ref{amplituhedron_projected_flows_equation} $\Rightarrow$ \ref{amplituhedron_projected_flows_coherent}: Suppose that there exists $M\in\gl_{k+m}(\mathbb{C})$ such that $ZN = MZ$. Then $Z\exp(t\ii N) = \exp(t\ii M)Z$ for all $t\in\mathbb{R}$, so
\begin{align}\label{amplituhedron_commuting flows}
\tilde{Z}(\exp(t\ii N)V) = \exp(t\ii M)\tilde{Z}(V) \;\text{ in } \Gr_{k,k+m}(\mathbb{C}) \quad \text{ for all } V\in\Gr_{k,n}(\mathbb{C}) \text{ and } t\in\mathbb{R}.
\end{align}
In particular, the projection of the gradient flow with respect to $N$ beginning at $V\in\Gr_{k,n}(\mathbb{C})$ only depends on $\tilde{Z}(V)$.
\end{proof}

\begin{cor}\label{projected_flows_gradient}
Let $Z$ be a complex $(k+m)\times n$ matrix such that $Z\adjoint{Z} = \I_{k+m}$. Let $N\in\uu_n$ such that $N(\ker(Z)) \subseteq \ker(Z)$, and set $M := ZN\adjoint{Z}\in\uu_{k+m}$. Then the rational map $\tilde{Z} : \Gr_{k,n}(\mathbb{C}) \dashrightarrow \Gr_{k,k+m}(\mathbb{C})$ takes gradient flows with respect to $N$ to gradient flows with respect to $M$ (given by \eqref{gradient_flow_kahler_equation} and respecting the parameter $t$).
\end{cor}

\begin{proof}
This follows from the implication \ref{amplituhedron_projected_flows_preserving} $\Rightarrow$ \ref{amplituhedron_projected_flows_equation} of \cref{amplituhedron_projected_flows}, along with \eqref{amplituhedron_commuting flows} (which both hold for all $k,m \ge 0$).
\end{proof}

\subsection{Gradient flows preserving amplituhedra}\label{sec_amplituhedra_gradient_preserving}
In this subsection, we show that if $\tilde{Z}$ projects a positivity-preserving gradient flow in a coherent way, then the projected gradient flow preserves the corresponding amplituhedron. In order to state our result precisely, we make the following analogue of \cref{defn_positivity_preserving} for amplituhedra.
\begin{defn}\label{defn_amplituhedron_preserving}
Let $Z$ be a real $(k+m)\times n$ matrix whose $(k+m)\times (k+m)$ minors are all positive, and consider the amplituhedron $\mathcal{A}_{n,k,m}(Z)\subseteq\Gr_{k,k+m}(\mathbb{C})$. We say that a flow $V(t)$ on $\Gr_{k,k+m}(\mathbb{C})$ {\itshape weakly preserves $\mathcal{A}_{n,k,m}(Z)$} if
$$
V(t)\in\mathcal{A}_{n,k,m}(Z) \quad \text{ for all } V_0\in\mathcal{A}_{n,k,m}(Z) \text{ and } t \ge 0,
$$
and {\itshape strictly preserves $\mathcal{A}_{n,k,m}(Z)$} if
$$
V(t)\in\interior{\mathcal{A}_{n,k,m}(Z)} \quad \text{ for all } V_0\in\mathcal{A}_{n,k,m}(Z) \text{ and } t>0,
$$
where $\interior{\mathcal{A}_{n,k,m}(Z)}$ denotes the interior of $\mathcal{A}_{n,k,m}(Z)$.
\end{defn}

We will need the following result of Galashin and Lam \cite{galashin_lam20}:
\begin{lem}[Galashin and Lam {\cite[Lemma 9.4]{galashin_lam20}}]\label{amplituhedron_interior}
Let $Z$ be a real $(k+m)\times n$ matrix whose $(k+m)\times (k+m)$ minors are all positive. Then $\tilde{Z}(V) \in \interior{\mathcal{A}_{n,k,m}(Z)}$ for all $V\in\Gr_{k,n}^{>0}$.
\end{lem}

As in \cref{sec_amplituhedra_gradient_projection}, we identify the Grassmannians $\Gr_{k,n}(\mathbb{C})$ and $\Gr_{k,k+m}(\mathbb{C})$ with adjoint orbits via \eqref{Fl_to_orbit_map}. We also recall the stable manifold defined in \cref{defn_stable_manifold}.
\begin{prop}\label{amplituhedron_preserving_flows}
Let $Z$ be a real $(k+m)\times n$ matrix whose $(k+m)\times (k+m)$ minors are all positive and such that $Z\transpose{Z} = \I_{k+m}$. Let $N\in\uu_n$ such that $N(\ker(Z)) \subseteq \ker(Z)$, and set $M := ZN\adjoint{Z}\in\uu_{k+m}$.
\begin{enumerate}[label=(\roman*), leftmargin=*, itemsep=2pt]
\item\label{amplituhedron_preserving_flows_weak} If the gradient flow on $\Gr_{k,n}(\mathbb{C})$ with respect to $N$ weakly preserves positivity, then the gradient flow on $\Gr_{k,k+m}(\mathbb{C})$ with respect to $M$ weakly preserves $\mathcal{A}_{n,k,m}(Z)$.
\item\label{amplituhedron_preserving_flows_strict} If the gradient flow on $\Gr_{k,n}(\mathbb{C})$ with respect to $N$ strictly preserves positivity, then the gradient flow on $\Gr_{k,k+m}(\mathbb{C})$ with respect to $M$ strictly preserves $\mathcal{A}_{n,k,m}(Z)$. Moreover, the stable manifold for $M$ inside $\Gr_{k,k+m}(\mathbb{C})$ is well-defined (i.e.\ if $-M\in\Orbit_{\bfmu}$, then $\mu_k > \mu_{k+1}$), and it contains $\mathcal{A}_{n,k,m}(Z)$.
\end{enumerate}

\end{prop}

\begin{proof}
By \cref{projected_flows_gradient} and \cref{amplituhedron_interior}, if the gradient flow on $\Gr_{k,n}(\mathbb{C})$ with respect to $N$ weakly (respectively, strictly) preserves positivity, then the gradient flow on $\Gr_{k,k+m}(\mathbb{C})$ with respect to $M$ weakly (respectively, strictly) preserves $\mathcal{A}_{n,k,m}(Z)$. It remains to show that, assuming the gradient flow on $\Gr_{k,n}(\mathbb{C})$ with respect to $N$ strictly preserves positivity, the stable manifold for $M$ inside $\Gr_{k,k+m}(\mathbb{C})$ contains $\mathcal{A}_{n,k,m}(Z)$. Let $W^\infty\in\Gr_{k,n}^{>0}$ denote the subspace of $\mathbb{C}^n$ spanned by the eigenvectors of $\ii N$ corresponding to its $k$ largest eigenvalues, which is well-defined by \cref{eigenvalue_condition_satisfied}. Let $V^\infty := \tilde{Z}(W^\infty)$, which lies in $\interior{\mathcal{A}_{n,k,m}(Z)}$ by \cref{amplituhedron_interior}. By assumption, the spectrum of $\ii M$ equals the spectrum of $\ii N$ minus the spectrum of $\ii N$ restricted to $\ker(Z)$. Since $\tilde{Z}(W^\infty)$ is defined, we have $W^\infty\cap\ker(Z) = \{0\}$, and so the $k$ largest eigenvalues of $\ii M$ and $\ii N$ coincide. In particular, the stable manifold for the gradient flow on $\Orbit_{\bfomega{k}}$ with respect to $M$ as in \cref{defn_stable_manifold} is well-defined, and the equilibrium point therein corresponds via \eqref{Fl_to_orbit_map} to $V^\infty$. Therefore by \cref{stable_manifold_orbit}\ref{stable_manifold_orbit_positivity}, the stable manifold of $V^\infty$ inside $\Gr_{k,k+m}(\mathbb{C})$ contains $\mathcal{A}_{n,k,m}(Z)$.
\end{proof}

\begin{rmk}\label{amplituhedron_preserving_flows_remark}
In \cref{amplituhedron_preserving_flows}, the simultaneous conditions on $N\in\uu_n$ that $N(\ker(Z)) \subseteq \ker(Z)$ and that the gradient flow on $\Gr_{k,n}(\mathbb{C})$ with respect to $N$ preserves positivity are highly constraining. Rather than relying on the existence of such an $N$, it would be interesting to classify directly those $M\in\uu_{k+m}$ such that the gradient flow on $\Gr_{k,k+m}(\mathbb{C})$ with respect to $M$ preserves $\mathcal{A}_{n,k,m}(Z)$. This may be possible using the intrinsic descriptions of $\mathcal{A}_{n,k,m}(Z)$ conjectured by Arkani-Hamed, Thomas, and Trnka \cite{arkani-hamed_thomas_trnka18} (cf.\ \cite[Section 3.3]{karp_williams19}).
\end{rmk}

\subsection{Amplituhedra homeomorphic to a closed ball}\label{sec_amplituhedra_ball}
We now show that any amplituhedron $\mathcal{A}_{n,k,m}(Z)$ satisfying the hypotheses of \cref{amplituhedron_preserving_flows}\ref{amplituhedron_preserving_flows_strict} is homeomorphic to a closed ball.
\begin{lem}\label{amplituhedron_ball_hypotheses}
Let $n,k,m\in\mathbb{N}$ such that $k+m \le n$, and let $N\in\uu_n$ such that the gradient flow on $\Gr_{k+m,n}(\mathbb{C})$ with respect to $N$ strictly preserves positivity. Let $Z$ be a real $(k+m)\times n$ matrix whose rows form an orthonormal basis for the subspace spanned by the eigenvectors of $\ii N$ corresponding to the $k+m$ largest eigenvalues, so that in particular $Z\transpose{Z} = \I_{k+m}$.
\begin{enumerate}[label=(\roman*), leftmargin=*, itemsep=2pt]
\item\label{amplituhedron_ball_hypotheses_positive} All $(k+m)\times (k+m)$ minors of $Z$ are positive (perhaps after negating a row of $Z$).
\item\label{amplituhedron_ball_hypotheses_compatible} We have $N(\ker(Z)) \subseteq \ker(Z)$.
\end{enumerate}

\end{lem}

\begin{proof}
\ref{amplituhedron_ball_hypotheses_positive} By \cref{tnn_Fl}\ref{tnn_Fl_tp}, it suffices to verify that the row span of $Z$ lies in $\Gr_{k+m,n}^{>0}$. This follows from \cref{eigenvalue_condition_satisfied} (which also shows that $Z$ is well-defined).

\ref{amplituhedron_ball_hypotheses_compatible} This follows from the fact that $\ker(Z)$ is spanned by the eigenvectors of $\ii N$ corresponding to the $n-k-m$ smallest eigenvalues.
\end{proof}

\begin{thm}\label{amplituhedra_ball}
Let $n,k,m\in\mathbb{N}$ such that $k+m \le n$, and let $N\in\uu_n$ such that the gradient flows on both $\Gr_{k,n}(\mathbb{C})$ and $\Gr_{k+m,n}(\mathbb{C})$ with respect to $N$ strictly preserve positivity. (Recall that such $N$ are explicitly described by \cref{positivity_preserving_kahler_grassmannian}.) Let $Z$ be any real $(k+m)\times n$ matrix whose rows form a basis for the subspace spanned by the eigenvectors of $\ii N$ corresponding to the $k+m$ largest eigenvalues. Then the amplituhedron $\mathcal{A}_{n,k,m}(Z)$ is well-defined (perhaps after negating a row of $Z$). It is homeomorphic to a closed ball, its interior is homeomorphic to an open ball, and its boundary is homeomorphic to a sphere.
\end{thm}

\begin{proof}
By \cref{amplituhedron_orthonormal_basis}, we may assume that the rows of $Z$ are orthonormal. Then by \cref{amplituhedron_ball_hypotheses}\ref{amplituhedron_ball_hypotheses_positive}, all $(k+m)\times (k+m)$ minors of $Z$ are positive (perhaps after negating a row of $Z$), so the amplituhedron $\mathcal{A}_{n,k,m}(Z)$ is well-defined. We also have $N(\ker(Z)) \subseteq \ker(Z)$. Consider the gradient flow on $\Gr_{k,k+m}(\mathbb{C})$ with respect to $M := ZN\transpose{Z}\in\uu_{k+m}$, where we identify $\Gr_{k,k+m}(\mathbb{C})$ with the adjoint orbit $\Orbit_{\bfomega{k}}$ via \eqref{Fl_to_orbit_map}. We apply \cref{compact_invariant_ball}, taking $S$ to be $\mathcal{A}_{n,k,m}(Z)$. The space $S$ is compact because it is the image of the compact space $\Gr_{k,n}^{\ge 0}$ under the continuous map $\tilde{Z}$. The remaining hypotheses of \cref{compact_invariant_ball} follow from \cref{amplituhedron_preserving_flows}\ref{amplituhedron_preserving_flows_strict}.
\end{proof}

\begin{rmk}
While \cref{amplituhedra_ball} applies only to a special subset of amplituhedra, we expect that every amplituhedron is homeomorphic to a closed ball. It would be interesting to determine whether this can be proved using \cref{compact_invariant_ball}, by constructing a contractive gradient flow on an arbitrary amplituhedron, or if only a distinguished subset of amplituhedra admit contractive gradient flows.
\end{rmk}

\subsection{Twisted Vandermonde amplituhedra}\label{sec_twisted_vandermonde_amplituhedra}
We now exhibit an explicit family of matrices $Z$ for which \cref{amplituhedra_ball} implies that the corresponding amplituhedra $\mathcal{A}_{n,k,m}(Z)$ are homeomorphic to closed balls. Our description will use the Vandermonde flags introduced in \cref{sec_tridiagonal_orbit} and the twist map $\twist$ from \cref{sec_twist}.
\begin{defn}\label{defn_twisted_vandermonde_amplituhedra}
Let $n,k,m\in\mathbb{N}$ such that $k+m \le n$. Let $V\in\Vand_n^{>0}$ be a totally positive Vandermonde flag, so that $\twist(V)\in\Fl_n^{>0}$ by \cref{twist_action}. Regarding $\twist(V)$ as a sequence of subspaces of $\mathbb{C}^n$, let $Z$ be a $(k+m)\times n$ real matrix whose rows form a basis for the subspace of dimension $k+m$. By \cref{tnn_Fl}\ref{tnn_Fl_tp}, all $(k+m) \times (k+m)$ minors of $Z$ are positive (perhaps after negating a row of $Z$). We call the corresponding amplituhedron $\mathcal{A}_{n,k,m}(Z)$ a {\itshape twisted Vandermonde amplituhedron}.
\end{defn}

We observe that the definition of $Z$ above depends only on $k+m$, not on $k$ or $m$. Therefore each such $Z$ gives rise to several different twisted Vandermonde amplituhedra.
\begin{eg}\label{eg_twisted_vandermonde_amplituhedra}
We give an example in the case $n := 3$ and $k+m := 2$. As in \cref{eg_tridiagonal_flag}, we consider the Vandermonde flag $\Kry{\bflambda}{x}$, where $\bflambda := (1, 0, -1)$ and $x\in\mathbb{P}^2_{>0}$. Then the twisted flag $\twist(\Kry{\bflambda}{x})\in\Fl_3^{>0}$ is represented by the matrix
$$
\begin{bmatrix}
\frac{x_1}{\sqrt{x_1^2 + x_2^2 + x_3^2}} & \frac{-x_2}{\sqrt{x_1^2 + x_2^2 + x_3^2}} & \frac{x_3}{\sqrt{x_1^2 + x_2^2 + x_3^2}} \\[10pt]
\frac{x_1(x_2^2 + 2x_3^2)}{\sqrt{(x_1^2 + x_2^2 + x_3^2)(x_1^2x_2^2 + 4x_1^2x_3^2 + x_2^2x_3^2)}} & \frac{x_2(x_1^2 - x_3^2)}{\sqrt{(x_1^2 + x_2^2 + x_3^2)(x_1^2x_2^2 + 4x_1^2x_3^2 + x_2^2x_3^2)}} & \frac{-x_3(2x_1^2 + x_2^2)}{\sqrt{(x_1^2 + x_2^2 + x_3^2)(x_1^2x_2^2 + 4x_1^2x_3^2 + x_2^2x_3^2)}} \\[10pt]
\frac{x_2x_3}{\sqrt{x_1^2x_2^2 + 4x_1^2x_3^2 + x_2^2x_3^2}} & \frac{2x_1x_3}{\sqrt{x_1^2x_2^2 + 4x_1^2x_3^2 + x_2^2x_3^2}} & \frac{x_1x_2}{\sqrt{x_1^2x_2^2 + 4x_1^2x_3^2 + x_2^2x_3^2}} \\[8pt]
\end{bmatrix}.
$$
Therefore the associated twisted Vandermonde amplituhedron $\mathcal{A}_{n,k,m}(Z)$ is defined by
\begin{gather*}
Z := \begin{bmatrix}
\frac{x_1}{\sqrt{x_1^2 + x_2^2 + x_3^2}} & \frac{x_1(x_2^2 + 2x_3^2)}{\sqrt{(x_1^2 + x_2^2 + x_3^2)(x_1^2x_2^2 + 4x_1^2x_3^2 + x_2^2x_3^2)}} & \frac{x_2x_3}{\sqrt{x_1^2x_2^2 + 4x_1^2x_3^2 + x_2^2x_3^2}} \\[10pt]
\frac{-x_2}{\sqrt{x_1^2 + x_2^2 + x_3^2}} & \frac{x_2(x_1^2 - x_3^2)}{\sqrt{(x_1^2 + x_2^2 + x_3^2)(x_1^2x_2^2 + 4x_1^2x_3^2 + x_2^2x_3^2)}} & \frac{2x_1x_3}{\sqrt{x_1^2x_2^2 + 4x_1^2x_3^2 + x_2^2x_3^2}} \\[8pt]
\end{bmatrix}.\qedhere
\end{gather*}

\end{eg}

\begin{cor}\label{twisted_vandermonde_amplituhedra_ball}
Every twisted Vandermonde amplituhedron is homeomorphic to a closed ball, its interior is homeomorphic to an open ball, and its boundary is homeomorphic to a sphere.
\end{cor}

\begin{proof}
Consider a twisted Vandermonde amplituhedron coming from the twisted Vandermonde flag $\twist(V)$. Let $-N\in\Orbit_{\bflambda}$ correspond to $\twist(V)$ under the map \eqref{Fl_to_orbit_map}. By \cref{tridiagonal_flag}, we have $\ii N\in\gl_n^{>0}$. Therefore by \cref{positivity_preserving_kahler_grassmannian}, the gradient flows on both $\Gr_{k,n}(\mathbb{C})$ and $\Gr_{k,k+m}(\mathbb{C})$ with respect to $N$ strictly preserve positivity. The result then follows from \cref{amplituhedra_ball}.
\end{proof}

\begin{rmk}\label{vandermonde_flags_corner_remark}
We note that the twisted Vandermonde amplituhedra are precisely those which arise in \cref{amplituhedra_ball} when the matrix $N$ is tridiagonal. Recall that in general, the matrices $N$ in \cref{amplituhedra_ball} are described by \cref{positivity_preserving_kahler_grassmannian}; for simplicity, here we assume that $k \ge 2$ or $k+m \le n-2$. When $m$ is odd, all such matrices $N$ are tridiagonal, and therefore the twisted Vandermonde amplituhedra are the only ones which arise in \cref{amplituhedra_ball}. However, when $m$ is even, the corner entry $(\ii N)_{n,1} = (\ii N)_{1,n}$ of $\ii N$ can be nonzero, of sign $(-1)^{k-1}$. (We may still assume that the entries $(\ii N)_{i,i+1}$ for $i = 1, \dots, n-1$ are nonzero; if some such entry is zero, we can use the cyclic action from \cref{cyclic_action} to transform $N$ into a tridiagonal matrix.)

We focus in this subsection on the case that the corner entry is zero because when $N$ is tridiagonal, we have an explicit description of the corresponding element of the flag variety $\Fl_n(\mathbb{C})$, by \cref{tridiagonal_flag}. It would be interesting to generalize this description to the case when the corner entry of $N$ is nonzero. The simplest such matrix $N$ is given by
$$
\ii N = \begin{bmatrix}
0 & 1 & 0 & \cdots & 0 & (-1)^{k-1} \\
1 & 0 & 1 & \cdots & 0 & 0 \\
0 & 1 & 0 & \cdots & 0 & 0 \\
\vdots & \vdots & \vdots & \ddots & \vdots & \vdots \\
0 & 0 & 0 & \cdots & 0 & 1 \\
(-1)^{k-1} & 0 & 0 & \cdots & 1 & 0
\end{bmatrix}.
$$
This is the matrix used by Galashin, Karp, and Lam \cite[Theorem 1]{galashin_karp_lam} to show that the totally nonnegative Grassmannian $\Gr_{k,n}^{\ge 0}$ is homeomorphic to a closed ball (cf.\ \cref{sec_homeomorphism}). They also studied the corresponding amplituhedron in \cite[Section 5]{galashin_karp_lam}, which they called the {\itshape cyclically symmetric amplituhedron}, because the cyclic action from \cref{cyclic_action} on $\Gr_{k,n}^{\ge 0}$ restricts coherently to a cyclic action on $\mathcal{A}_{n,k,m}(Z)$. For example, when $k=1$ and $m=2$, the cyclically symmetric amplituhedron is a regular $n$-gon. For this specific choice of $\ii N$, there is an elegant explicit description of the row span of the associated $(k+m)\times n$ matrix $Z$ \cite{karp19}.
\end{rmk}

\begin{rmk}
Even more generally, as discussed in \cref{symmetric_to_general}, we can replace $\ii N$ with $M\in\gl_n(\mathbb{R})$ (not necessarily symmetric), although this setup falls outside the orbit framework. If $M$ satisfies condition \cref{infinitesimal_k}\ref{infinitesimal_k_positive_entries} for both $k$ and $k+m$, we can conclude that the corresponding amplituhedron $\mathcal{A}_{n,k,m}(Z)$ is homeomorphic to a closed ball. Still, we expect that some significant new ideas are required to use this approach to show that every amplituhedron $\mathcal{A}_{n,k,m}(Z)$ is homeomorphic to a closed ball. We can justify this with a dimension count. Indeed, the space of amplituhedra $\mathcal{A}_{n,k,m}(Z)$ for all $Z$ is naturally indexed by $\Gr_{k+m,n}^{>0}$, which has dimension $(k+m)(n-k-m)$. On the other hand, consider the space of matrices $M\in\gl_n(\mathbb{R})$ satisfying the condition \cref{infinitesimal_k}\ref{infinitesimal_k_positive_entries} for both $k$ and $k+m$, modulo translation by scalar multiples of $\I_n$ and rescaling by $\mathbb{R}_{>0}$. Then assuming $k \ge 2$ or $k+m \le n-2$, the dimension of this space is either $3n-2$ (if $m$ is even) or $3n-4$ (if $m$ is odd). Note that when $4 \le k+m \le n-4$ and $n \ge 15$, we have $(k+m)(n-k-m) > 3n-2$. However, it may be possible to use the approach above to show that every amplituhedron $\mathcal{A}_{n,k,m}(Z)$ with $k+m \le 3$ or $n-k-m \le 3$ is homeomorphic to a closed ball, generalizing the arguments in \cref{sec_amplituhedra_ball_case}.
\end{rmk}

\subsection{The case when \texorpdfstring{$n-k-m \le 2$}{n-k-m is at most 2}}\label{sec_amplituhedra_ball_case}
In this subsection, we show that every amplituhedron $\mathcal{A}_{n,k,m}(Z)$ with either $k+m \le 2$ or $n-k-m \le 2$ is a twisted Vandermonde amplituhedron. It particular, every such amplituhedron is homeomorphic to a closed ball. Recall from \eqref{defn_tnn_Fl_surjections} that for any $K'\subseteq K$, we have a surjective projection map $\PFl{K}{n}^{>0} \twoheadrightarrow \PFl{K'}{n}^{>0}$. Also recall the space of totally positive Vandermonde flags $\Vand_n^{>0}$ from \cref{defn_vandermonde_flags}.
\begin{lem}\label{twist_projection_bijection}
Let $n\ge 2$.
\begin{enumerate}[label=(\roman*), leftmargin=*, itemsep=2pt]
\item\label{twist_projection_bijection_first} The projection map $\Fl_n^{>0} \to \PFl{\{1,2\}}{n}^{>0}$ is a bijection when restricted to $\twist(\Vand_n^{>0})$.
\item\label{twist_projection_bijection_last} The projection map $\Fl_n^{>0} \to \PFl{\{n-2,n-1\}}{n}^{>0}$ is a bijection when restricted to $\twist(\Vand_n^{>0})$.
\end{enumerate}

\end{lem}

\begin{proof}
\ref{twist_projection_bijection_first} By \cref{all_tridiagonal_flags}, it suffices to prove that that the map
$$
(\ii\gl_n^{>0})\cap\uu_n/{\sim} \hspace*{2pt}\to\hspace*{2pt} \PFl{\{1,2\}}{n}^{>0}, \quad g(\ii\Diag{\bflambda})g^{-1} \mapsto g
$$
is a bijection, that is, any given $V = (V_1, V_2)\in\PFl{\{1,2\}}{n}^{>0}$ has a unique preimage $L \in (\ii\gl_n^{>0})\cap\uu_n$, modulo translating $L$ by a scalar multiple of $\I_n$ and rescaling it by an element of $\mathbb{R}_{>0}$. We will show, equivalently, that $V$ has a unique preimage $L$ which lies in $\Jac_{\bflambda}^{>0}$, for some strictly decreasing $\bflambda\in\mathbb{R}^n$ with $\lambda_1 = 0$ and $\lambda_2 = -1$.

Recall the torus action from \cref{torus_action}. After replacing $V$ and $L$ by, respectively, $hV$ and $hLh^{-1}$ for some $h\in\H_n^{>0}$, we may assume that $V_1$ is spanned by $(1, \dots, 1)$. Now take a nonzero vector $y = (y_1, \dots, y_n)\in V_2$ orthogonal to $(1, \dots, 1)$, so that $y_1 + \cdots + y_n = 0$. By \cref{tnn_Fl}\ref{tnn_Fl_tp}, the $2\times 2$ minors of the matrix
$$
\begin{bmatrix}
1 & -y_1 \\
\vdots & \vdots \\
1 & -y_n
\end{bmatrix}
$$
are all positive (perhaps after replacing $y$ by $-y$), whence $y_1 > \cdots > y_n$. We must show that there is a unique $L \in (\ii\gl_n^{>0})\cap\uu_n$ satisfying the following two properties:
\begin{enumerate}[label=(\alph*), leftmargin=36pt, itemsep=2pt]
\item\label{twist_projection_bijection_first_eigenvectors} the vectors $(1, \dots, 1)$ and $y$ are eigenvectors of $-\ii L$ with eigenvalues $0$ and $-1$, respectively; and
\item\label{twist_projection_bijection_first_eigenvalues} the two largest eigenvalues of $-\ii L$ are $0$ and $-1$.
\end{enumerate}

First we show that there is a unique $L \in (\ii\gl_n^{>0})\cap\uu_n$ satisfying property \ref{twist_projection_bijection_first_eigenvectors}. Write
$$
L = \ii\begin{bmatrix}
b_1 & a_1 & 0 & \cdots & 0 \\
a_1 & b_2 & a_2 & \cdots & 0 \\
0 & a_2 & b_3 & \cdots & 0 \\
\vdots & \vdots & \vdots & \ddots & \vdots \\
0 & 0 & 0 & \cdots & b_n
\end{bmatrix}\in\uu_n, \quad \text{ where } a_1, \dots, a_{n-1}, b_1, \dots, b_n \in \mathbb{R}.
$$
Then property \ref{twist_projection_bijection_first_eigenvectors} holds if and only if
$$
a_{i-1} + b_i + a_i = 0 \quad \text{ and } \quad a_{i-1}y_{i-1} + b_iy_i + a_iy_{i+1} = -y_i \quad \text{ for } 1 \le i \le n,
$$
where we set $a_0, a_n := 0$. These equations have a unique solution, namely,
\begin{align}\label{twist_projection_bijection_first_solution}
a_i = \frac{y_1 + \cdots + y_i}{y_i - y_{i+1}}\; \text{ for } 1 \le i \le n-1 \quad \text{ and } \quad b_i = -a_{i-1} - a_i\; \text{ for } 1 \le i \le n.
\end{align}
Note that the conditions on $y$ imply that $a_i > 0$ for $1 \le i \le n-1$, so $L\in\ii\gl_n^{>0}$.

Now we verify that the matrix $L$ given by \eqref{twist_projection_bijection_first_solution} satisfies property \ref{twist_projection_bijection_first_eigenvalues}. Since $L(V_k) \subseteq V_k$ for $k=1,2$, we have that $V_k$ is spanned by some $k$ eigenvectors of $-\ii L$; we must show that these eigenvectors correspond to the $k$ largest eigenvalues. To this end, consider the gradient flow on $\Gr_{k,n}(\mathbb{C})$ with respect to $-L$, where we identify $\Gr_{k,n}(\mathbb{C})$ with the adjoint orbit $\Orbit_{\bfomega{k}}$ via \eqref{Fl_to_orbit_map}. By \cref{positivity_preserving_kahler_grassmannian}, this flow strictly preserves positivity, and by construction, $V_k$ is a totally positive equilibrium point. Therefore by \cref{stable_manifold_grassmannian}\ref{stable_manifold_grassmannian_positivity} (using \cref{eigenvalue_condition_satisfied}), we have that $V_k$ is spanned by the eigenvectors of $-\ii L$ corresponding to the $k$ largest eigenvalues, as desired.

\ref{twist_projection_bijection_last} By \eqref{krylov_rev} we have $\rev(\Vand_n^{>0}) = \Vand_n^{>0}$, whence $\flip(\twist(\Vand_n^{>0})) = \twist(\Vand_n^{>0})$ by \cref{twist_flip_rev}\ref{twist_flip_rev_Fl}. Therefore the result follows from part \ref{twist_projection_bijection_first} and \cref{flip_action}\ref{flip_action_tp}.
\end{proof}

\begin{rmk}
It is tempting to try to prove \cref{twist_projection_bijection}\ref{twist_projection_bijection_first} by observing that the projection map $\Fl_n^{>0} \to \PFl{\{1,2\}}{n}^{>0}$ is a bijection when restricted to $\Vand_n^{>0}$, and then applying bijectivity of the twist map $\twist$. However, there is good reason to expect such an argument may fail. Indeed, fix a strictly decreasing $\bflambda\in\mathbb{R}^n$. Then the projection map $\Fl_n^{>0} \to \PFl{\{1\}}{n}^{>0} = \mathbb{P}^{n-1}_{>0}$ is a bijection when restricted to $\{\Kry{\bflambda}{x} : x\in\mathbb{P}^{n-1}_{>0}\}$, as it sends $\Kry{\bflambda}{x}$ to $x$. But the projection map $\Fl_n^{>0} \to \mathbb{P}^{n-1}_{>0}$ is in general neither injective nor surjective when restricted to $\twist(\{\Kry{\bflambda}{x} : x\in\mathbb{P}^{n-1}_{>0}\})$. For example, let $\bflambda := (1,0,-1)$, as in \cref{eg_tridiagonal_flag}. Then the projection map sends $\twist(\Kry{\bflambda}{x})$ to
$$
(y_1 : y_2 : y_3) := \big(x_1\textstyle\sqrt{x_1^2x_2^2 + 4x_1^2x_3^2 + x_2^2x_3^2} \hspace*{2pt}:\hspace*{2pt} x_1(x_2^2 + 2x_3^2) \hspace*{2pt}:\hspace*{2pt} x_2x_3\textstyle\sqrt{x_1^2 + x_2^2 + x_3^2}\big) \in\mathbb{P}^2_{>0}.
$$
The points $x = (1 : 1 : 1)$ and $x = (\sqrt{10} : 4 : 2)$ have the same image, so the map is not injective. Also, any such $(y_1 : y_2 : y_3)$ satisfies the constraint $\min(y_1,y_3) < y_2$ (proof omitted), so the map is not surjective.
\end{rmk}

\begin{cor}\label{amplituhedra_ball_case_ball}
Every amplituhedron $\mathcal{A}_{n,k,m}(Z)$ with either $k+m \le 2$ or $n-k-m \le 2$ is a twisted Vandermonde amplituhedron. In particular, it is homeomorphic to a closed ball, its interior is homeomorphic to an open ball, and its boundary is homeomorphic to a sphere.
\end{cor}

\begin{proof}
Suppose that $k+m \le 2$ or $n-k-m \le 2$. Then \cref{twist_projection_bijection} implies that the projection map $\Fl_n^{>0} \to \Gr_{k+m,n}^{>0}$ is surjective when restricted to $\twist(\Vand_n^{>0})$, so every amplituhedron $\mathcal{A}_{n,k,m}(Z)$ is a twisted Vandermonde amplituhedron. The remaining statements follow from \cref{twisted_vandermonde_amplituhedra_ball}.
\end{proof}

\section{Symmetric Toda flow}\label{sec_toda}

\noindent The {\itshape Toda lattice} is an integrable Hamiltonian system which has been widely studied since it was introduced by Toda in 1967 \cite{toda67b}; see the survey of Kodama and Shipman \cite{kodama_shipman18}. By work of Flaschka \cite{flaschka74}, we may view the Toda lattice as a flow evolving on an adjoint orbit $\Orbit_{\bflambda}$. In this section, we observe that the Toda flow provides an example of a gradient flow on $\Orbit_{\bflambda}$. Curiously, this happens in two different ways: in both the normal metric and the K\"{a}hler metric. The Toda flow is also an example of a flow which weakly preserves positivity (in fact, in both the positive and negative time directions). As we discuss further in \cref{torus_remark}, while these results are largely implicit in the literature, we believe the explicit focus on total positivity offers a new perspective. In particular, a key role is played by the twist map $\twist$ introduced in \cref{sec_twist}, which facilitates the study of the Toda flow as a gradient flow in the K\"{a}hler metric. This generalizes and clarifies a construction of Bloch, Flaschka, and Ratiu \cite{bloch_flaschka_ratiu90}, as we explain in \cref{BFR_map}.

\subsection{Background}\label{sec_toda_background}

We introduce the (finite nonperiodic) Toda lattice; we refer to \cite{kodama_shipman18} for further details. It is the Hamiltonian system with Hamiltonian
$$
H(q_1, \dots, q_n, p_1, \dots, p_n) := \frac{1}{2}\sum_{i=1}^np_i^2 + \sum_{i=1}^{n-1}e^{q_i - q_{i+1}}.
$$
The Toda lattice may be interpreted as a system of $n$ points on a line of unit mass governed by an exponential potential.

Following Flaschka \cite{flaschka74}, we set
$$
a_i := \textstyle\frac{1}{2}e^{\frac{q_{i}-q_{i+1}}{2}} \; \text{ for } 1 \le i \le n-1 \quad \text{ and } \quad b_i :=  -\frac{1}{2}p_i \; \text{ for } 1 \le i \le n.
$$
Then the Hamiltonian equations become (with $a_0, a_n := 0$)
$$
\dot{a}_i = a_i(b_{i+1} - b_i) \quad \text{ and } \quad \dot{b}_i = 2(a_i^2 - a_{i-1}^2).
$$
We also let $L$ be the tridiagonal matrix
$$
L := \ii\begin{bmatrix}
b_1 & a_1 & 0 & \cdots & 0 \\
a_1 & b_2 & a_2 & \cdots & 0 \\
0 & a_2 & b_3 & \cdots & 0 \\
\vdots & \vdots & \vdots & \ddots & \vdots \\
0 & 0 & 0 & \cdots & b_n
\end{bmatrix}, \quad \text{ so that } \quad \kterm(-\ii L) = \begin{bmatrix}
0 & -a_1 & 0 & \cdots & 0 \\
a_1 & 0 & -a_2 & \cdots & 0 \\
0 & a_2 & 0 & \cdots & 0 \\
\vdots & \vdots & \vdots & \ddots & \vdots \\
0 & 0 & 0 & \cdots & 0
\end{bmatrix},
$$
where $\kterm$ was defined in \cref{defn_iwasawa_projection}. Then we can write the flow of the Toda lattice in the Lax form (cf.\ \cref{lax_flow}\ref{lax_flow_translation})
\begin{align}\label{toda_flaschka}
\dot{L}(t) = [L(t), \kterm(-\ii L(t))].
\end{align}
Therefore \eqref{toda_flaschka} defines a flow on the adjoint orbit $\Orbit_{\bflambda}$.

Above, $L$ was assumed to be $\ii$ times a real symmetric tridiagonal matrix, but more generally, we can consider any $L\in\Orbit_{\bflambda}$. We call the flow \eqref{toda_flaschka} defined on the tridiagonal part of $\uu_n$ the {\itshape tridiagonal symmetric Toda flow}, and call the flow defined on all of $\uu_n$ the {\itshape full symmetric Toda flow}, which was studied by Deift, Li, Nanda, and Tomei \cite{deift_li_nanda_tomei86}. (The term {\itshape symmetric} is conventional, since $L$ is usually defined to be a real symmetric matrix, without the factor of $\ii$. We prefer instead to work in $\Orbit_{\bflambda}$, following e.g.\ \cite{bloch_flaschka_ratiu90}.)

Symes \cite{symes80} found an explicit solution to \eqref{toda_flaschka} using the Iwasawa decomposition (cf.\ \cref{iwasawa_decomposition}). It can be verified directly.
\begin{thm}[{Symes \cite[Section 7]{symes80}, \cite[(3.2)]{symes82}}]\label{symes_solution}
Let $L(t)$ be a solution to the full symmetric Toda flow \eqref{toda_flaschka}, with $L_0\in\uu_n$. Then
\begin{align}\label{symes_solution_equation}
L(t) = \Kterm(\exp(-t\ii L_0))^{-1} L_0 \Kterm(\exp(-t\ii L_0)).
\end{align}

\end{thm}

We observe that using \eqref{symes_solution_equation}, one can read off the asymptotic behavior of $L(t)$ as $t \to \pm\infty$. In particular, the limits are both diagonal matrices; see \cref{sorting_remark}.

\begin{rmk}\label{kostant_toda}
Another important reformulation of the Toda lattice was given by Kostant \cite{kostant79} in terms of tridiagonal {\itshape Hessenberg matrices}, rather than symmetric matrices. This was generalized to all Hessenberg matrices by Ercolani, Flaschka, and Singer \cite{ercolani_flaschka_singer93}, and is known as the {\itshape full Kostant--Toda lattice}. The Kostant--Toda flows are in general more complicated than the symmetric Toda flows; for example, they are not necessarily complete. As is the case for the symmetric Toda flow, total positivity plays an important role for the Kostant--Toda flow, as shown by Gekhtman and Shapiro \cite{gekhtman_shapiro97} and Kodama and Williams \cite{kodama_williams15}. It would be interesting to explore this connection further. We leave this to future work, since the Kostant--Toda flow does not directly fit into the framework of adjoint orbits considered in this paper.
\end{rmk}

\subsection{The Toda flow as a gradient flow in the normal metric}\label{sec_toda_normal}
In this subsection we consider the tridiagonal symmetric Toda flow. Bloch \cite{bloch90} observed that
\begin{align}\label{tridiagonal_bracket}
\kterm(-\ii L) = [L, -\ii\Diag{n-1, \dots, 1, 0}]  \quad \text{ for $L\in\uu_n$ tridiagonal}.
\end{align}
Therefore the following result holds:
\begin{thm}[{Bloch \cite[Section 6]{bloch90}}]\label{tridiagonal_symmetric_toda_gradient}
Set $N := -\ii\Diag{n-1, \dots, 1, 0} \in \uu_n$, and let $L_0\in\uu_n$ be tridiagonal. Then the tridiagonal symmetric Toda flow \eqref{toda_flaschka} beginning at $L_0$ can be written as
$$
\dot{L}(t) = [L(t), [L(t), N]].
$$
In particular (cf.\ \cref{gradient_flow_normal}), the tridiagonal symmetric Toda flow restricted to $\Orbit_{\bflambda}$ is the gradient flow with respect to $N$ in the normal metric.
\end{thm}

\begin{rmk}\label{full_symmetric_toda_gradient_remark}
In general, for $L\in\Orbit_{\bflambda}$ not assumed to be tridiagonal, the equality \eqref{tridiagonal_bracket} fails to hold, and \eqref{toda_flaschka} is not a gradient flow in the normal metric. Nevertheless, De Mari and Pedroni \cite[Theorem 5.1]{de_mari_pedroni99} (cf.\ \cite[Proposition 2.3]{bloch_gekhtman98}) generalized \cref{tridiagonal_symmetric_toda_gradient} to the full symmetric Toda flow, by showing that it is a gradient flow in a modification of the normal metric. Bloch and Gekhtman \cite[Section 2.3]{bloch_gekhtman98} proved an analogous result for the full Kostant--Toda flow.
\end{rmk}

\subsection{The Toda flow as a twisted gradient flow in the K\"{a}hler metric}\label{sec_toda_kahler}
In this subsection we consider the full symmetric Toda flow, restricted to the totally nonnegative part $\Orbit_{\bflambda}^{\ge 0}$ of an adjoint orbit. Our analysis is based on Symes's formula \eqref{symes_solution_equation}, and the twist map introduced in \cref{sec_twist}.
\begin{defn}\label{defn_twist_orbit}
Recall the twist map $\twist : \Fl_n^{\ge 0} \to \Fl_n^{\ge 0}$ from \cref{defn_twist}. Given any strictly decreasing $\bflambda\in\mathbb{R}^n$, via the identification \eqref{Fl_to_orbit_map}, we may also regard the twist map as a map $\twistorbit : \Orbit_{\bflambda}^{\ge 0} \to \Orbit_{\bflambda}^{\ge 0}$. Explicitly, it is the involution defined as
$$
\twistorbit(g(\ii\Diag{\bflambda})g^{-1}) := \iota(g)(\ii\Diag{\bflambda})(\iota(g))^{-1} = \delta_ng^{-1}(\ii\Diag{\bflambda})g\delta_n \quad \text{ for all } g\in\U_n^{\ge 0}.
$$
(If $\bflambda\in\mathbb{R}^n$ is weakly decreasing but not strictly decreasing, then $\twistorbit$ is undefined; cf.\ \cref{twist_extension}.)
\end{defn}

We also recall the cell decomposition \eqref{cell_decomposition_equation} of $\Fl_n^{\ge 0}$. If $\bflambda\in\mathbb{R}^n$ is strictly decreasing, this induces a cell decomposition of $\Orbit_{\bflambda}^{\ge 0}$ via \eqref{Fl_to_orbit_map}.

\begin{thm}\label{full_symmetric_toda_gradient}
Let $\bflambda\in\mathbb{R}^n$ be strictly decreasing, and set $N := -\ii\Diag{\bflambda}\in\uu_n$.
\begin{enumerate}[label=(\roman*), leftmargin=*, itemsep=2pt]
\item\label{full_symmetric_toda_gradient_positivity} The full symmetric Toda flow on $\Orbit_{\bflambda}$ weakly preserves positivity in both the positive and negative time directions. That is, if $L(t)$ evolves according to \eqref{toda_flaschka} beginning at $L_0\in\Orbit_{\bflambda}^{\ge 0}$, then $L(t)\in\Orbit_{\bflambda}^{\ge 0}$ for all $t\in\mathbb{R}$. Moreover, $L(t)$ (for $t\in\mathbb{R})$ is contained in a single cell of the cell decomposition \eqref{cell_decomposition_equation} of $\Orbit_{\bflambda}^{\ge 0}$.
\item\label{full_symmetric_toda_twist} The full symmetric Toda flow restricted to $\Orbit_{\bflambda}^{\ge 0}$ is the twisted gradient flow with respect to $N$ in the K\"{a}hler metric. That is, if $L(t)$ evolves according to \eqref{toda_flaschka} beginning at $L_0\in\Orbit_{\bflambda}^{\ge 0}$, then $\twistorbit(L(t))$ is the gradient flow \eqref{gradient_flow_kahler_equation_iwasawa} with respect to $N$ in the K\"{a}hler metric beginning at $\twistorbit(L_0)\in\Orbit_{\bflambda}^{\ge 0}$.
\end{enumerate}

\end{thm}

We observe that because the matrix $N$ above is diagonal, the twisted flow $\twistorbit(L(t))$ is contained in a torus orbit of $\Orbit_{\bflambda}$. This fact is relevant if we wish to map such flows to a moment polytope, as considered by Bloch, Flaschka, and Ratiu \cite{bloch_flaschka_ratiu90} and Kodama and Williams \cite{kodama_williams15}. We discuss this further in \cref{BFR_map}.

\begin{proof}
Let $L_0\in\Orbit_{\bflambda}^{\ge 0}$, and set $M_0 := \twistorbit(L_0)\in\Orbit_{\bflambda}^{\ge 0}$. Let $L(t)$ evolve according to \eqref{toda_flaschka} beginning at $L_0$, and let $M(t)$ be the gradient flow with respect to $N$ in the K\"{a}hler metric beginning at $M_0$. Using \cref{twist_action}, it suffices to verify the following two facts.
\begin{enumerate}[label=(\roman*), leftmargin=*, itemsep=2pt]
\item\label{full_symmetric_toda_gradient_proof_nonnegativity} We have $M(t)\in\Orbit_{\bflambda}^{\ge 0}$ for all $t\in\mathbb{R}$. Moreover, $M(t)$ (for $t\in\mathbb{R})$ is contained in a single cell of the cell decomposition \eqref{cell_decomposition_equation} of $\Orbit_{\bflambda}^{\ge 0}$.
\item\label{full_symmetric_toda_gradient_proof_twist} We have $L(t) = \twistorbit(M(t))$ for all $t\in\mathbb{R}$.
\end{enumerate}

Since $M_0\in\Orbit_{\bflambda}^{\ge 0}$, we can write $M_0 = g_0(\ii\Diag{\bflambda})g_0^{-1}$ for some $g_0\in\U_n^{\ge 0}$. Then we define $g(t)\in\U_n$ as in \eqref{gradient_flow_kahler_equation_iwasawa}, so that $M(t) = g(t)(\ii\Diag{\bflambda})g(t)^{-1}$ for all $t\in\mathbb{R}$:
$$
g(t) := \Kterm(\exp(t\ii N)g_0) = \Kterm(\Diag{e^{t\lambda_1}, \dots, e^{t\lambda_n}}g_0).
$$
Since $\Diag{e^{t\lambda_1}, \dots, e^{t\lambda_n}}\in\H_n^{>0}$, and each cell of \eqref{cell_decomposition_equation} is preserved by the action of $\H_n^{>0}$, we obtain part \ref{full_symmetric_toda_gradient_proof_nonnegativity} above.

Now observe that
$$
\twistorbit(M(t)) = \delta_n g(t)^{-1}(\ii\Diag{\bflambda})g(t)\delta_n.
$$
In particular, taking $t=0$ we obtain
$$
L_0 = \twistorbit(M_0) = \delta_n g_0^{-1}(\ii\Diag{\bflambda})g_0\delta_n.
$$
Therefore using \eqref{symes_solution_equation}, in order to prove part \ref{full_symmetric_toda_gradient_proof_twist} above, it suffices to show that
$$
g(t) = g_0\delta_n\Kterm(\exp(-t\ii L_0))\delta_n.
$$
This equality follows from $\exp(-t\ii L_0) = \delta_ng_0^{-1}\Diag{e^{t\lambda_1}, \dots, e^{t\lambda_n}}g_0\delta_n$, along with the fact that $\Kterm$ commutes with both left multiplication by $\U_n$ and right multiplication by $\H_n(\mathbb{C})$.
\end{proof}

\begin{rmk}\label{torus_remark}
The fact that a trajectory $L(t)$ of the full symmetric Toda flow \eqref{toda_flaschka} can be realized as a gradient flow in a torus orbit of a flag variety is well-known (see e.g.\ \cite[Section 5.2]{singer91}). This was first observed in the tridiagonal case by Moser \cite[(1.4)]{moser75}, who embedded $\Jac_{\bflambda}^{>0}$ inside $\mathbb{P}^{n-1}_{>0}$ (cf.\ \cref{moser_variables}). The new insight provided here is that there is a canonical such embedding when $L_0\in\Orbit_{\bflambda}^{\ge 0}$, which is given by a smooth map defined on all of $\Orbit_{\bflambda}^{\ge 0}$ (namely, the twist map). The subtlety of constructing such a map in general was noted by Ercolani, Flaschka, and Singer \cite[Remark p.\ 194]{ercolani_flaschka_singer93}; also see \cref{BFR_map}. For a related approach to this problem, see \cite[Theorem 1]{martinez_torres_tomei}.
\end{rmk}

\begin{rmk}\label{BFR_map}
As we alluded to in \cref{BFR_remark}, when we restrict the domain of the twist map $\twistorbit$ from $\Orbit_{\bflambda}^{\ge 0}$ to the tridiagonal subset $\Jac_{\bflambda}^{\ge 0}$ (discussed in \cref{sec_tridiagonal_orbit}), it specializes to a map constructed by Bloch, Flaschka, and Ratiu \cite{bloch_flaschka_ratiu90} in general Lie type, and denoted $\iota$. (The map $\iota$ of \cite{bloch_flaschka_ratiu90} is different from the map we denote by the same letter in \cref{defn_positive_inverse}. We also emphasize that in general, the image $\twistorbit(\Jac_{\bflambda}^{\ge 0})$ is not contained in $\Jac_{\bflambda}^{\ge 0}$.) The context in which the map $\iota$ appeared in \cite{bloch_flaschka_ratiu90} is similar to the one in the current discussion, namely, in order to realize the Toda flow on $\Jac_{\bflambda}^{\ge 0}$ as a gradient flow compatible with the torus action; see \cite[Theorem p.\ 63]{bloch_flaschka_ratiu90}. The ultimate goal in \cite{bloch_flaschka_ratiu90} was to prove \cref{jacobi_manifold}, by mapping $\Jac_{\bflambda}^{\ge 0}$ to its moment polytope. It turns out that the usual moment map is neither injective nor surjective on $\Jac_{\bflambda}^{\ge 0}$, but if we first apply the map $\iota$, we obtain a homeomorphism onto the moment polytope which restricts to a diffeomorphism from $\Jac_{\bflambda}^{>0}$ onto its interior.

The subtlety in constructing the maps $\twistorbit$ and $\iota$ is to pick a canonical representative in $\O_n$ (out of a possible $2^n$) for an arbitrary element of $\Fl_n(\mathbb{R})$. It is impossible to pick a smooth representative over all of $\Fl_n(\mathbb{R})$, which is why in defining $\twist$ we restrict to the totally nonnegative part $\Fl_n^{\ge 0}$ and pick the representative in which all left-justified minors are nonnegative. In \cite{bloch_flaschka_ratiu90}, the representative in $\O_n$ is chosen to be the one in which the first row is positive. This is ultimately equivalent to our choice (up to multiplying by $\delta_n$) when we restrict to $\Fl_n^{>0}$, but on the boundary of $\Fl_n^{\ge 0}$ some entries of the first row of the matrix representative may be zero; see \cref{first_row_positive}. This issue necessitated in \cite{bloch_flaschka_ratiu90} an intricate analysis involving the Bruhat decomposition. (The embedding $\Jac_{\bflambda}^{>0}\hookrightarrow\mathbb{P}^{n-1}_{>0}$ of Moser mentioned in \cref{torus_remark} does not extend to the closure $\Jac_{\bflambda}^{\ge 0}$ for similar reasons.) We find the perspective of total positivity gives a natural way to define and extend the map $\iota$, which requires no special consideration at the boundary.
\end{rmk}

\begin{rmk}\label{tridiagonal_real_to_positive}
In the case of the real tridiagonal symmetric Toda flow, there is no loss of generality in restricting to the totally nonnegative part $\Jac_{\bflambda}^{\ge 0}$. That is, suppose we are given $L_0\in\uu_n$ such that $-\ii L_0$ is a real tridiagonal symmetric matrix. Then we can conjugate $L_0$ by an element of the form $\Diag{\pm 1, \dots, \pm 1}$ so that the off-diagonal entries of $-\ii L_0$ become nonnegative, whence $L_0\in\Jac_{\bflambda}^{\ge 0}$. On the other hand, this conjugation commutes with the flow \eqref{toda_flaschka}. We note that this reduction to the totally nonnegative case from the real case does not extend to the complex case, nor to the real full symmetric case.
\end{rmk}

\begin{rmk}\label{kodama_williams_remark}
Kodama and Williams \cite[Section 5]{kodama_williams15} proved a result analogous to \cref{full_symmetric_toda_gradient}\ref{full_symmetric_toda_gradient_positivity} for the full Kostant--Toda flow. Namely, to any point in $\Fl_n^{\ge 0}$ they associate a Hessenberg matrix, and show that the corresponding Kostant--Toda flow is complete; moreover, when the flow is mapped back to $\Fl_n^{\ge 0}$, it is contained inside a single cell of the cell decomposition \eqref{cell_decomposition_equation}. (In the case of the top-dimensional cell $\Fl_n^{>0}$, this is a special case of an earlier result of Gekhtman and Shapiro \cite[Theorem 2]{gekhtman_shapiro97}.) Kodama and Williams further translate their results to the full symmetric Toda lattice \cite[Section 7]{kodama_williams15}, following a procedure of Bloch and Gekhtman \cite{bloch_gekhtman98}. This translation employs a different convention than we use for mapping between $\Fl_n(\mathbb{C})$ and $\Orbit_{\bflambda}$ (cf.\ \cref{left_right_action}), and getting between the two (in the totally nonnegative case) requires applying the twist map. In particular, \cref{full_symmetric_toda_gradient}\ref{full_symmetric_toda_gradient_positivity} follows from \cite[Proposition 7.8]{kodama_williams15} once we know the properties of the twist map given in \cref{twist_action}.
\end{rmk}

\begin{rmk}\label{gladwell_remark}
Gladwell \cite[Theorem 2]{gladwell02} proved a result analogous to \cref{full_symmetric_toda_gradient}\ref{full_symmetric_toda_gradient_positivity} for totally positive matrices. Namely, let $M\in\gl_n(\mathbb{R})$ be symmetric, and let $L(t)$ evolve according to \eqref{toda_flaschka} beginning at $L_0 := \ii M \in \uu_n$. Write $L(t) = \ii M(t)$. Gladwell showed that if $M\in\GL_n^{\ge 0}$, then $M(t)\in\GL_n^{\ge 0}$ for all $t \in\mathbb{R}$; and if $M\in\GL_n^{>0}$, then $M(t)\in\GL_n^{>0}$ for all $t \in\mathbb{R}$. We observe that this result neither directly implies, nor is directly implied by, \cref{full_symmetric_toda_gradient}\ref{full_symmetric_toda_gradient_positivity}.
\end{rmk}

\begin{rmk}\label{sorting_remark}
We note that the Toda flow does not quite fit into the framework of \cref{sec_ball}, because it only weakly (rather than strictly) preserves positivity. In particular, we cannot use the Toda flow to show that $\Orbit_{\bflambda}^{\ge 0}$ is homeomorphic to a closed ball. Nevertheless, we can apply \cref{stable_manifold_grassmannian_remark} (along with \cref{full_symmetric_toda_gradient}) to obtain the {\itshape sorting property} for the full symmetric Toda flow restricted to $\Orbit_{\bflambda}^{>0}$. Namely, letting $L(t)\in\Orbit_{\bflambda}$ evolve according to \eqref{toda_flaschka} beginning at $L_0\in\Orbit_{\bflambda}^{>0}$, we have
$$
\lim_{t\to\infty}L(t) = \ii\Diag{\lambda_1, \dots, \lambda_n} \quad \text{ and } \quad \lim_{t\to -\infty}L(t) = \ii\Diag{\lambda_n, \dots, \lambda_1}.
$$
In general, Chernyakov, Sharygin, and Sorin \cite[Section 3.3]{chernyakov_sharygin_sorin14} (cf.\ \cite[Theorem 7.9]{kodama_williams15} and \cite[Theorem 2]{martinez_torres_tomei}) showed that the limits of $L(t)$ as $t\to\pm\infty$ are diagonal matrices determined by the Schubert and opposite Schubert cells containing $L_0$.
\end{rmk}

\bibliographystyle{alpha}
\bibliography{ref}

\end{document}